\documentclass[leqno, a4paper,11pt]{amsart}
\pdfoutput=1

\makeatletter
\def\@tocline#1#2#3#4#5#6#7{\relax
  \ifnum #1>\c@tocdepth 
  \else
    \par \addpenalty\@secpenalty\addvspace{#2}%
    \begingroup \hyphenpenalty\@M
    \@ifempty{#4}{%
      \@tempdima\csname r@tocindent\number#1\endcsname\relax
    }{%
      \@tempdima#4\relax
    }%
    \parindent\z@ \leftskip#3\relax \advance\leftskip\@tempdima\relax
    \rightskip\@pnumwidth plus4em \parfillskip-\@pnumwidth
    #5\leavevmode\hskip-\@tempdima
      \ifcase #1
       \or\or \hskip 1em \or \hskip 2em \else \hskip 3em \fi%
      #6\nobreak\relax
    \hfill\hbox to\@pnumwidth{\@tocpagenum{#7}}\par
    \nobreak
    \endgroup
  \fi}
\makeatother
\newcounter{dummy}
\usepackage{hyperref}
\hypersetup{
 colorlinks,
 linkcolor={blue!90!black},
 citecolor={red!80!black},
 urlcolor={blue!50!black}
}

\usepackage[full]{textcomp}
\usepackage[varqu,varl]{inconsolata} 
\usepackage{setspace}
\usepackage[margin=2.4cm]{geometry}
\usepackage{mathabx}
\usepackage{amsmath}
\usepackage{mathtools}
\usepackage{mathrsfs}
\usepackage[alphabetic]{amsrefs}
\usepackage{comment}
\usepackage{stmaryrd}
\usepackage{array}
\usepackage{multicol}

\usepackage[usenames,dvipsnames]{color}
\usepackage{tikz-cd}
\usepackage{graphicx}
\usepackage{amsfonts,amssymb}

\newtheorem{thm}{Theorem}[section]
\newtheorem{lemma}[thm]{Lemma}
\newtheorem{cor}[thm]{Corollary}
\newtheorem{prop}[thm]{Proposition}

\theoremstyle{definition}
\newtheorem{example}[thm]{Example}
\newtheorem{notation}[thm]{Notation}
\newtheorem{remark}[thm]{Remark}
\newtheorem{definition}[thm]{Definition}

\newtheorem{question}[thm]{Question}

\numberwithin{equation}{section}

\newcommand{\Spec}{\mathrm{Spec}}
\newcommand{\Proj}{\mathrm{Proj}}
\newcommand{\Hilb}{\mathrm{Hilb}}
\newcommand{\Gr}{\mathrm{Gr}}
\newcommand{\Sym}{\mathrm{Sym}}
\newcommand{\Wed}{\bigwedge}

\newcommand{\Hom}{\mathrm{Hom}}
\newcommand{\Mat}{\mathrm{Mat}}
\newcommand{\Sing}{\mathrm{Sing}}
\newcommand{\Span}{\mathrm{Span}}
\newcommand{\GL}{\mathrm{GL}}
\newcommand{\LM}{\mathrm{LM}}
\newcommand{\iin}{\mathrm{in}}
\newcommand{\V}{\mathrm{V}}

\newcommand{\rank}{\mathrm{rank}}
\newcommand{\codim}{\mathrm{codim}}
\newcommand{\ord}{\mathrm{ord}}
\newcommand{\rev}{\mathrm{grevlex}}

\newcommand{\colength}{\mathrm{colength}}
\newcommand{\bideg}{\mathrm{bideg}}
\newcommand{\ch}{\mathrm{char}}

\renewcommand{\AA}{\mathbb{A}}
\renewcommand{\P}{\mathbb{P}}
\newcommand{\N}{\mathbb{N}}
\newcommand{\Z}{\mathbb{Z}}
\newcommand{\bS}{\mathbb{S}}
\newcommand{\FF}{\mathbb{F}}

\newcommand{\kk}{{\Bbbk}}

\newcommand{\mL}{\mathrm{L}}
\newcommand{\mR}{\mathrm{R}}

\newcommand{\mfI}{\mathfrak{I}}
\newcommand{\mfL}{\mathfrak{L}}

\newcommand{\mfK}{\mathfrak{K}}
\newcommand{\mfP}{\mathfrak{P}}
\newcommand{\mfQ}{\mathfrak{Q}}
\newcommand{\mfB}{\mathfrak{B}}
\newcommand{\mfa}{\mathfrak{a}}
\newcommand{\mfb}{\mathfrak{b}}
\newcommand{\mfp}{\mathfrak{p}}
\newcommand{\mfX}{\mathfrak{X}}
\newcommand{\mm}{{\mathfrak{m}}}

\newcommand{\wtX}{\widetilde{\mathfrak{X}}}

\newcommand{\bfX}{\mathbf{X}}
\newcommand{\bfY}{\mathbf{Y}}
\newcommand{\bfW}{\mathbf{W}}
\newcommand{\bfM}{\mathbf{M}}
\newcommand{\bfE}{\mathbf{E}}
\newcommand{\bfA}{\mathbf{A}}
\newcommand{\bfB}{\mathbf{B}}
\newcommand{\bfI}{\mathbf{I}}
\newcommand{\bfN}{\mathbf{N}}

\newcommand{\bfa}{\mathbf{a}}
\newcommand{\bfe}{\mathbf{e}}
\newcommand{\bfu}{\mathbf{u}}
\newcommand{\bfv}{\mathbf{v}}
\newcommand{\bfw}{\mathbf{w}}
\newcommand{\bfx}{\mathbf{x}}
\newcommand{\bfy}{\mathbf{y}}
\newcommand{\bfz}{\mathbf{z}}

\newcommand{\mcV}{\mathcal{V}}
\newcommand{\mcW}{\mathcal{W}}
\newcommand{\mcR}{\mathcal{R}}
\newcommand{\mcS}{\mathcal{S}}
\newcommand{\mcT}{\mathcal{T}}
\newcommand{\mcQ}{\mathcal{Q}}
\newcommand{\mcO}{{\mathcal{O}}}

\newcommand{\mcX}{\mathcal{X}}
\newcommand{\mcY}{\mathcal{Y}}
\newcommand{\mcA}{\mathcal{A}}
\newcommand{\mcB}{\mathcal{B}}
\newcommand{\mcC}{\mathcal{C}}
\newcommand{\mcD}{\mathcal{D}}
\newcommand{\mcE}{\mathcal{E}}

\newcommand{\mcU}{\mathcal{U}}

\newcommand{\CC}{{\mathcal{C}}}

\newcommand{\sd}{\mathrm{sd}}

\makeatletter
\newcommand\myitem[1][]{\item[#1]\refstepcounter{dummy}\def\@currentlabel{#1}}
\makeatother

\makeatletter
\@namedef{subjclassname@2020}{
  \textup{2020} Mathematics Subject Classification}
\makeatother

\begin{document}
\author[R.\,Ramkumar, A. \,Sammartano]{Ritvik~Ramkumar and Alessio~Sammartano}
\address{(Ritvik Ramkumar) Department of Mathematics\\Cornell University\\Ithaca, NY\\USA}
\email{ritvikr@cornell.edu}
\address{(Alessio Sammartano) Dipartimento di Matematica \\ Politecnico di Milano \\ Milan \\ Italy}
\email{alessio.sammartano@polimi.it}

\title{Rational singularities of nested Hilbert schemes}


\begin{abstract}
The Hilbert scheme of  points $\Hilb^n(S)$ of a smooth surface $S$ is a well-studied parameter space, lying at the interface of  algebraic geometry, commutative algebra, representation theory, combinatorics,  and mathematical physics. 
The foundational result is  a classical theorem of Fogarty, 
stating that $\Hilb^n(S)$ is a smooth variety of dimension $2n$. 
In recent years there has been growing interest in a natural generalization of $\Hilb^n(S)$,
 the \emph{nested Hilbert scheme} $\Hilb^{(n_1, n_2)}(S)$, 
which parametrizes nested pairs of zero-dimensional subschemes $Z_1 \supseteq  Z_2$ of $S$ with $\deg Z_i=n_i$.
In contrast to Fogarty’s theorem, 
$\Hilb^{(n_1, n_2)}(S)$ is almost always singular, 
and very little is known about its singularities. 
In this paper,
 we aim to advance the knowledge of the geometry of these
nested Hilbert schemes.
Work by Fogarty in the 70's shows that $\Hilb^{(n,1)}(S)$ is a normal Cohen-Macaulay variety, 
and Song more recently proved that  it has rational singularities.
In our main result, we prove that the nested Hilbert scheme $\Hilb^{(n,2)}(S)$ has rational singularities. 
We employ an array of tools from commutative algebra to prove this theorem.
Using Gr\"obner bases, we establish a connection between $\Hilb^{(n,2)}(S)$ and  a certain variety of matrices with an action of the general linear group. 
This variety of matrices plays a central role in our work, and we 
analyze it by  various algebraic techniques,
including the Kempf-Lascoux-Weyman technique of calculating syzygies, square-free Gr\"obner degenerations, 
and the Stanley-Reisner correspondence.
Along the way, we also obtain results on classes of irreducible and reducible  nested Hilbert schemes, dimension of singular loci, and  $F$-singularities in positive characteristic.
\end{abstract}

\maketitle

\vspace*{-0.5cm}

\tableofcontents

\newpage

\section{Introduction}

The Hilbert scheme of $n$ points on a smooth  surface $S$, 
denoted by $\Hilb^n(S)$, parametrizing closed zero-dimensional subschemes of $S$ of degree $n$, 
is a very important moduli space in algebraic geometry with far reaching connections to other areas of mathematics. 
Fogarty \cite{Fogarty} proved it is nonsingular of dimension $2n$.
When $S = \P^2$,
 Ellingsrud and Str\o mme \cite{ELLINGSRUD_STROMME} computed its homology, and Arcara,
Bertram, Coskun, and Huizenga \cite{ABCH} studied its birational geometry in great detail. 
By studying its cohomology rings,
 Nakajima connected $\Hilb^n(S)$ to representation theory \cite{NAKAJIMA}, while Haimain brought it into prominence in combinatorics through his proof of the $n!$ conjecture \cite{HAIMAN}. 
When $S$ is a K3 surface, the Hilbert scheme is a  Hyperk\"ahler manifold  \cite{BEAUVILLE}, an important class of varieties in symplectic geometry and mathematical physics.
See \cite{Gottsche,NakajimaBook} for an overview of the area.

In recent years,
 there has been growing interest in a natural generalization of $\Hilb^n(S)$,
 the \emph{nested Hilbert scheme} $\Hilb^{(n_1, n_2)}(S)$, 
which parametrizes nested pairs of zero-dimensional subschemes $Z_1 \supseteq  Z_2$ of $S$ with $\deg Z_i=n_i$.
See for example
\cite{Addington,BE, BFT,GNR,GSY,JJ2,JJ,Negut,OR,RY,SV,VFJ} and the references therein.
Cheah \cite{Cheah} showed that $\Hilb^{(n+1, n)}(S)$  are the only  smooth nested Hilbert schemes.
Among the remaining nested Hilbert schemes, 
the only well studied one is  $\Hilb^{(n,1)}(S)$. 
Fogarty \cite{Fogarty2} showed it is normal and Cohen-Macaulay,
 while Song \cite{Song} proved it has rational singularities (in characteristic $0$). 
These results rely crucially on the fact that 
$\Hilb^{(n,1)}(S)$ is in fact the universal family over $\Hilb^{n}(S)$.

The goal of this paper is to introduce  methods to study questions regarding the singularities of $\Hilb^{(n_1, n_2)}(S)$. 
Our main result is 

\begin{thm}
\label{MainTheorem}
Let $S$ be a smooth, connected surface over a field $\kk$ of characteristic 0.
The nested Hilbert scheme $\Hilb^{(n,2)}(S)$ is an irreducible $2n$-fold, non-singular in codimension 3, with  rational singularities. In particular, $\Hilb^{(n,2)}(S)$  is normal and  Cohen-Macaulay.
\end{thm}

It is quite interesting that nested Hilbert schemes seem to produce  classes of varieties with various degrees of intermediate behavior,
in contrast to classical Hilbert schemes of points $\Hilb^n(\P^d)$,
for which very little is known between the extreme cases
of smooth ones for $d = 2$ and terribly singular ones for $d \gg 0$
\cite{JJ};
see for instance \cite{DS}.

Our main approach to proving Theorem \ref{MainTheorem} will be to translate the problem 
into the study of certain explicit ideals,
 which can be treated with methods from commutative algebra.
First of all, by standard arguments one may assume $S = \AA^2$. 
Generic Gr\"obner degenerations then reduce the problem to the singularities  at some particular Borel-fixed points 
in the nested Hilbert scheme.
We prove that the natural morphism 
$\Hilb^{(n,2)}(\AA^2) \to \Hilb^{2}(\AA^2)$ is flat;
this shifts the focus to  the fiber over the Borel-fixed point of $\Hilb^{2}(\AA^2)$.
We  prove that this fiber  is reduced, and it can thus be identified with  the variety of matrices
$$
\mfX =
 \big\{ \bfW \in \mathrm{Mat}(n+1,n) \, : \,
I_n(\bfY + \bfW) \subseteq (y^2)  \subseteq \kk[y]\big\},
$$
where $\bfY$ is the $(n+1)\times n$ matrix with $y$ on the main diagonal and 0 elsewhere.
Using machinery from representation theory, linear algebra, and combinatorial commutative algebra,
 we prove
\begin{thm} 
Assume $\ch (\Bbbk) =0$. 
The  variety $\mfX \subseteq \Mat(n+1,n)$ is irreducible of dimension $n^2+n- 4$, with  rational singularities.
It is a cone over a projective subvariety of $ \P^{n^2+n-1}$  of degree 
$
\frac{1}{12}(n-1)n(n+1)(3n-2).
$
Moreover, it has a Cohen-Macaulay squarefree Gr\"obner degeneration.
\end{thm}
To prove this theorem,
 we show that $\mfX$ admits a desingularization that is a vector bundle over a flag variety,
exploiting the fact that $\mfX$ projects onto a  rank variety \cite{ES}. This allows us to use the powerful geometric technique of Kempf-Lascoux-Weyman to determine the degree of $\mfX$ and deduce that it has rational singularities. 
We use this and an analysis of a simplicial complex associated to $\mfX$  to construct a square-free  initial ideal  of $I_\mfX$ and show it is Cohen-Macaulay.

We point out that Gr\"obner bases  play an important  role throughout  the paper. 
This is mainly due to the fact that reducedness  is a highly nontrivial issue 
for Hilbert schemes of points, see e.g.
\cite{JJ2,JJ,Szachniewicz}.
In fact, 
proving that $\Hilb^{(n,2)}(S)$ is  reduced is among the hardest tasks of this paper; 
to the best of our knowledge, there are very few proofs of reducedness for (singular!) Hilbert schemes.
For this reason, set-theoretic arguments are too naive
for studying singularities.
We use Gr\"obner bases to obtain scheme-theoretic equations of the singularity in Section \ref{SectionLocalEquations},
and then to obtain the desired reducedness and flatness results in Sections \ref{SectionSquarefreeInitial} and \ref{SectionProofMainTheorem}

\subsection{Organization} \label{Organization}

We now describe in more detail  the contents  of the paper  and  the proof strategy for Theorem \ref{MainTheorem}.
Each section from \ref{SectionIrreducible} to \ref{SectionProofMainTheorem} corresponds to a main intermediate result or a reduction step in the proof of Theorem \ref{MainTheorem}. 
As such, they will be somewhat self-contained; 
on a first reading, one may choose to treat some of them as black boxes, and proceed in a nonlinear order.
The only exception is Section \ref{SectionLocalEquations}, where we introduce the algebraic objects that play a central role in the subsequent sections.

\smallskip
We begin in Section \ref{SectionPreliminaries} by reviewing some basic facts on nested Hilbert schemes and rational singularities 
as well as the reduction to the case
where the surface is $S=\AA^2$.

In Section \ref{SectionIrreducible}, 
we study the simultaneous smoothability of chains of finite subschemes of $\AA^2$,
by employing  a point-detaching technique due to Hartshorne.
We obtain the irreducibility of $\Hilb^{(n,2)}(\AA^2)$ and a description of its singular locus. 
We also apply the technique to recover all known classes of irreducible nested Hilbert schemes, in arbitrary characteristic, 
and construct classes of reducible nested Hilbert schemes $\Hilb^{(n_1, \ldots, n_k)}(\AA^2)$
for every $k \geq 5$.

Section \ref{SectionReductionCompressed} deals with the reduction to ``compressed pairs''.
By tracking the complete local rings
of  $\Hilb^{(n,2)}(\AA^2)$ along suitable generic Gr\"obner degenerations,
the study of  singularities for arbitrary $\Hilb^{(n,2)}(\AA^2)$ is reduced to the case  where $n$ is a triangular number
and to special pairs of the form $[\V((x,y)^r) \supseteq \V(x,y^2)]\in \Hilb^{(n,2)}(\AA^2)$.

In Section \ref{SectionLocalEquations}, 
we  introduce two graded ideals that will be the main players in the rest of the paper.
Through an analysis of Gr\"obner strata in $\Hilb^n(\AA^2)$ and  the division algorithm for their universal families, 
we determine a polynomial ideal $\mfL$ which defines an open neighborhood of the compressed pair $[\V((x,y)^r) \supseteq \V(x,y^2)]$
in the nested Hilbert scheme.
We also determine explicitly the ideal $\mfI$ which defines an open subset of the scheme-theoretic fiber of the  morphism 
$\Hilb^{(n,2)}(\AA^2)\rightarrow \Hilb^2(\AA^2)$
over the ``origin'' $[\V(x,y^2)]$.

The next two sections are devoted to the fiber  $\V(\mfI)$.
A priori, this scheme may be non-reduced.
Section \ref{SectionGeometricTechnique}  treats the reduced scheme  $\mfX = \V(\sqrt{\mfI})$, which is 
a variety of  $(n+1)\times n$ matrices with an action of $\GL_n$.
We determine a desingularization of $\mfX$ in the form of a vector bundle over a flag variety.
This allows us to apply the Kempf-Lascoux-Weyman technique and prove that $\mfX$ is a normal Cohen-Macaulay variety with rational singularities, and  compute its degree.
In Section \ref{SectionSquarefreeInitial},
 we 
study a simplicial complex $\Delta$ associated to $\mfI$. 
Combining an enumerative analysis of $\Delta$ with the conclusions of Section \ref{SectionGeometricTechnique}, 
we prove that $\mfI$ has a Gr\"obner basis whose initial ideal is the Stanley-Reisner ideal $I_\Delta$,
and  deduce that $\mfI$ is a prime ideal, so    $\V(\mfI)=\mfX$ is in fact reduced.

Finally, in Section \ref{SectionProofMainTheorem}, we build upon the main results of all the previous sections to prove that $\mfL$ is prime,  the  morphism $\Hilb^{(n,2)}(\AA^2)\rightarrow \Hilb^2(\AA^2)$
 is flat, and deduce  Theorem \ref{MainTheorem}.

To conclude the paper, in Section \ref{SectionConclusions}, we discuss  three groups of questions suggested by our main theorems and methods:
extensions of our results to  positive characteristic,
 to arbitrary two-step nested Hilbert schemes $\Hilb^{(n_1, n_2)}(\AA^2)$,
and to a broad class of  varieties of matrices related to $\mfX$.

\subsection{Notation}
Throughout the paper, $\kk$ is an algebraically  closed  field.
We assume $\ch(\kk)=0$ in Sections \ref{SectionGeometricTechnique} and \ref{SectionProofMainTheorem} and in Theorem \ref{TheoremIPrime}.
The assumption  $\kk=\widebar{\kk}$ is not restrictive, 
see Remark \ref{RemarkAlgebraicallyClosed}.

Denote $R = \kk[x,y]$, $\mm = (x,y)$, and  $\AA^2 = \Spec(R)$.
Let $\V(I)$ be the subscheme defined by an ideal $I$, 
and   $I_Z$  the ideal of a subscheme $Z$.
We use $\LM(\cdot)$ to denote  leading monomials with respect to a given term order, and $\mathrm{in}(\cdot)$ for initial ideals.
A {\bf pair} is a nested chain $Z_1 \supseteq Z_2$  of two finite subschemes of $\AA^2$. 
We denote  by $[Z_1 \supseteq Z_2]$ the corresponding $\kk$-point on the nested Hilbert scheme.

\section{Preliminaries}\label{SectionPreliminaries}
In this section,
 we will review the definition of the nested Hilbert scheme of points  and summarize some of the known geometric results when the scheme is a smooth surface. 
We refer to \cite[Section 4]{S06} and \cite[Chapter 1]{CheahThesis} for more details. 
We also review some properties of rational singularities.

\begin{definition} \label{DefinitionHilbertScheme}
Let $S$ be a quasi-projective $\kk$-scheme and let $\lambda = (\lambda_1,\dots,\lambda_k) \in \mathbf{N}^k$ be a non-increasing sequence of natural numbers; the latter is usually called a {\bf partition}.
 There exists a quasi-projective $\kk$-scheme, denoted by $\Hilb^\lambda(S/\, \kk)$ or simply $\Hilb^\lambda(S)$,
parametrizing {\bf chains} of closed subschemes $Z_1\supseteq \cdots \supseteq Z_k$ of $S$ where $\text{length}(Z_i) = \lambda_i$. 
It is called the {\bf nested Hilbert scheme} of points of $S$ over $\kk$.
\end{definition}

If the scheme $S$ is connected, 
then the nested Hilbert scheme $\Hilb^{\lambda}(S/ \, \kk)$ is also connected.

\begin{remark}\label{RemarkAlgebraicallyClosed}
If $\kk'\to \kk$ is a subfield, there is a natural isomorphism 
$$\Hilb^{\lambda}(S/\kk') \times_{\kk'} \kk \to \Hilb^{\lambda}(S\times_{\kk'} \kk/\kk). $$ 
For this reason, we work without loss of generality over an algebraically closed field $\kk$.
\end{remark}

The interesting partitions $\lambda$ are those where the entries are all positive and distinct, since deleting zeros  or repeated entries will give the same nested Hilbert scheme. 
However, for technical reasons, it is convenient to allow for  zeros and repeated entries,
that is, we allow the possibilities that   $Z_i = Z_{i+1}$ or $Z_i = \emptyset$.
For instance, if
 $\lambda_2,\dots,\lambda_k = 0$, then $\Hilb^\lambda(S/ \, \kk) $ is the classical Hilbert scheme of $\lambda_1$ points on $S$.

While not much is known in general about $\Hilb^{\lambda}(S)$, the cases when $S$ is a smooth curve or surface have attracted considerable interest. 
Since the focus of this work is on smooth surfaces, we review some of the major structural results in this case.

\begin{thm} \emph{
Let $S$ be a smooth connected surface and $\lambda = (\lambda_1,\dots,\lambda_k)$ be a partition such that $\lambda_1 > \lambda_2 > \cdots > \lambda_k >0$.
\begin{enumerate}
\item \cite[Theorem]{Cheah} The nested Hilbert scheme $\Hilb^\lambda(S/ \, \kk)$ is smooth if and only if either 
$k=1$ or $k = 2$ and $\lambda_1 - \lambda_2 = 1$. 
\item \cite[Section 7]{Fogarty2} The nested Hilbert scheme $\Hilb^{(\lambda_1,1)}(S/ \, \kk)$ is integral, normal and Cohen-Macaulay.
\item \cite[Theorem 1.1]{Song} The nested Hilbert scheme $\Hilb^{(\lambda_1,1)}(S/ \, \mathbb{C})$ has rational singularities.
\item \cite{Addington,Negut,RT} The nested Hilbert scheme $\Hilb^{(\lambda_1,\lambda_1-1,\lambda_1-2)}(S/ \, \mathbb{C})$ is a local complete intersection with klt singularities.
\end{enumerate}
}
\end{thm}

We now  recall the definition and  some well-known facts about rational singularities.

\begin{definition} Let $X$ be a reduced $\kk$-scheme and $\text{char}(\kk)=0$. A {\bf resolution of singularities} of $X$ is a proper birational morphism $f:Z \to X$ with $Z$ a smooth $\kk$-scheme. The scheme $X$ is said to have {\bf rational singularities} if $X$ is normal and for any resolution of singularities $f:Z \to X$ we have $R^if_{\star}\mcO_{Z} =0$  for all $i>0$.
\end{definition}

\begin{lemma} \label{LemmaPropertiesRationalSingularities} 
Assume  $\ch(\kk)=0$.
Let $X, Y$ be affine $\kk$-schemes of finite type,
with $Y$ smooth.
\begin{enumerate}
\item If $X$ has a rational singularity at a point $p$, then $X$ has rational singularities in a neighborhood of $p$.
\item Let $\pi : X \to Y$ be a flat morphism, $p\in X$ and $q = \pi(p)$.
If the fiber $\pi^{-1}(q)$ has a rational singularity at $p$, 
then $X$ has a rational singularity at $p$.
\item If $ X  \times Y$ has rational singularities, then $X$ has rational singularities.
\item $\mcO_{X,p}$ has a rational singularity if and only if the completion $\widehat{\mcO_{X,p}}$ has a rational singularity.
\item If $X$ has rational singularities, then it is Cohen-Macaulay. 
\end{enumerate} 
\end{lemma}
\begin{proof} 
Item (1) is \cite[Th\'eor\`eme 4]{Elkik}, while item (2) is \cite[Th\'eor\`eme 2]{Elkik}.
Since all fibers of   $X \times Y \rightarrow Y$ are isomorphic to $X$, 
a resolution of singularities of $X$ induces a simultaneous resolution of  $X \times Y$ over $Y$.
Item (3) now follows from \cite[Th\'eor\`eme 3]{Elkik}.

By \cite[Lemma 15.52.6]{stacks}, $\mcO_{X,p}$ is normal if and only if $\widehat{\mcO_{X,p}}$ is normal,
so in item (4) we may assume that $\mcO_{X,p}$ is normal.
Denote $X_p = \Spec(\mcO_{X,p})$  and $\widehat{X_p} = \Spec(\widehat{\mcO_{X,p}})$,
and let $\pi: Z \to X_p$ be a resolution of singularities.
By \cite[Lemma 16.1 (ii)]{Lipman}, the base change 
$\hat{\pi}:\widehat{Z} = Z\times_{X_p} \widehat{X_p} \to  \widehat{X_p}$ is a resolution of singularities. 
By flat base change \cite[Lemma 30.5.2 (1)]{stacks}, there is an isomorphism
\begin{equation} \label{RScompletion}
R^i\pi_{\star}\mcO_{Z} \otimes_{\mcO_{X,p}}  \widehat{\mcO_{X,p}}
\cong
R^i\widehat{\pi}_{\star}\mcO_{\widehat{Z}}.
\end{equation}
Since  $\widehat{\mcO_{X,p}}$ is faithfully flat over $\mcO_{X,p}$, 
we conclude that 
$R^i\pi_{\star}\mcO_{Z}$ vanishes if and only if 
$R^i\widehat{\pi}_{\star}\mcO_{\widehat{Z}} $ vanishes,
and this completes the proof of item (4).

Finally, item (5) can be found in  \cite[Section 1.2.5]{WeymanBook}.
\end{proof}

The final goal of this section is to reduce the study of singularities of  $\Hilb^{\lambda}(S/\kk)$ to $\Hilb^{\lambda}(\AA^2/\kk)$.

\begin{lemma} \label{LemmaReductionPlane}
Let $S$ be a smooth surface.
The   completion of the  local ring of  $\Hilb^{\lambda}(S)$ at any closed point
 is isomorphic to the completion of the  local ring of   $\Hilb^{\lambda}(\AA^2)$  at some closed point.
\end{lemma}
\begin{proof}
Let $[Z_1 \supseteq \cdots \supseteq Z_k] \in \Hilb^{\lambda}(S)$ be a $\kk$-point.
Since $S$ is a smooth surface over ${\kk}$, there is an open neighborhood $U$ of $Z_1$ and an \'etale morphism $f:U \to \AA^2$ \cite[Section 29.36]{stacks}. 
By \cite[Lemma 4.4]{BehrendFantechi}, 
there is an open neighborhood $\tilde{U}$ of $[Z_1 \supseteq \cdots \supseteq Z_k]$ and an induced \'etale morphism $\tilde{f}:\tilde{U} \to \Hilb^{\lambda}(\AA^2)$; one can check that the Lemma remains true for nested Hilbert schemes. In particular, the completion of $\Hilb^{\lambda}(S)$ at $[Z_1 \supseteq \cdots \supseteq Z_k]$ is isomorphic to the completion of $\Hilb^{\lambda}(\AA^2)$ at $[f(Z_1) \supseteq \cdots \supseteq f(Z_k)]$. 
\end{proof}

\begin{cor}
\label{CorollaryReductionPlane}
Let $S$ be a smooth surface.
\begin{enumerate}
\item 
 Assume $\ch(\kk)=0$.
If $\Hilb^{\lambda}(\AA^2)$ has rational singularities, then so does $\Hilb^{\lambda}(S)$. 
\item Let $k\in \N$.
If $\Hilb^{\lambda}(\AA^2)$ is nonsingular in codimension $k$, 
then so is  $\Hilb^{\lambda}(S)$. 
\end{enumerate}
\end{cor}

\begin{proof}
Item (1) follows by Lemma \ref{LemmaReductionPlane} and Lemma \ref{LemmaPropertiesRationalSingularities} (4).
Item (2) follows by Lemma \ref{LemmaReductionPlane} and 
\cite[Theorem 23.9]{Matsumura}.
\end{proof}

\section{Irreducibility of nested Hilbert schemes}\label{SectionIrreducible}

The question of irreducibility of parameter spaces is a natural  problem in algebraic geometry. 
In this section, we prove the irreducibility of $\Hilb^{(m,2)}(\AA^2)$ and obtain some further  information on its geometry, by refining  classical deformation techniques for two-dimensional regular local rings.
In particular, the goal of this section is to prove the following result.

\begin{thm}\label{TheoremIrreducibleNested}
Let $m\in \N$.
The nested Hilbert scheme $\Hilb^{(m,2)}(\AA^2)$ is  irreducible of dimension  $2m$, and it is nonsingular in codimension 3.
\end{thm}

At the end of this section, we also discuss the question of irreducibility for more general nested Hilbert schemes $\Hilb^{(m_1, \ldots, m_d)}(\AA^2)$.

As in the classical case of $\Hilb^{m}(\AA^2)$,
there is a smooth and irreducible open subset
$\mcU \subseteq \Hilb^{(m,2)}(\AA^2)$,
which has dimension $2m$ and 
parametrizes pairs of smooth (i.e., reduced) finite subschemes of $\AA^2$.
Its closure is called the {\bf smoothable component} of $\Hilb^{(m,2)}(\AA^2)$
The irreducibility of $\Hilb^{(m,2)}(\AA^2)$  is thus equivalent to the statement that every 
pair $[Z_1\supseteq Z_2]$ is  smoothable, that is, it is a limit of pairs of smooth subschemes of $\AA^2$.

We say that a pair $[Z_1\supseteq Z_2]$ is {\bf irreducible} if  $Z_1$ and $Z_2$ are irreducible, equivalently, if $Z_1$ is irreducible.
Considering the unique decomposition of $Z_1$ into irreducible subschemes, 
it follows that 
every pair of subschemes has a unique decomposition into irreducible pairs.
Specifically, suppose that $[Z_1\supseteq Z_2]$ is a pair and $Z_1 = Z_{1,1} \cup \cdots \cup Z_{1,r}$ is the unique decomposition of $Z_1$ into irreducible zero-dimensional subschemes of $\AA^2$. 
Then   $Z_2= Z_{2,1} \cup \cdots \cup Z_{2,r}$ for some (possibly empty) 
irreducible zero-dimensional subschemes $Z_{2,i} \subseteq Z_{1,i}$,
and   $[Z_1\supseteq Z_2]$ is the union of the irreducible pairs $[Z_{1,i}\supseteq Z_{2,i}]$.

The concept of cleavability  is closely related to that of smoothability.
Following  \cite{BBKT}, to {\bf cleave} an irreducible subscheme $Z \subseteq \AA^2$ means to express it as a limit of reducible subschemes.
Likewise, 
to  cleave an irreducible pair $[Z_1\supseteq Z_2]$  means to express it as a limit of  reducible pairs.
In practice, this amount to giving  a flat family of pairs that cleaves at least the larger subscheme $Z_1$.

We recall a cleaving technique from \cite[Chapter 8]{hartshornedef}.
The  ($\mm$-adic) {\bf order}  of  $f\in R=\kk[x,y]$ is 
$\ord(f) = \inf\{ s \in \N \, : \, f \notin \mm^{s+1}\}$,
and the order of an $\mm-$primary ideal $I\subseteq R$
is $\ord(I) = \inf\{\ord(f) \, : f \in I\}$.
Denote the associated graded ring of $R$ by $\mathrm{gr}_\mm(R)$,
the initial form of  $f\in R$ by $f^*\in \mathrm{gr}_\mm(R)$,
and the ideal of initial forms of $I \subseteq R$ by  $I^*\subseteq \mathrm{gr}_\mm(R) $.
Recall the definition of colon ideal $I:\ell = \{ f \in R \, : \, \ell f \in I\}$.

\begin{lemma}\label{LemmaHartshorne}
Let $Z \subseteq \AA^2$ be an irreducible subscheme supported at $\V(\mm)$.
Let $f \in I_Z$ be a polynomial with $\ord(f) = \ord(I_Z)$, and $\ell\in R$ a (homogeneous) linear form
such that $f^*, \ell^* \in \mathrm{gr}_\mm(R)$ form a regular sequence.
The formula  
\begin{equation}\label{EqHarsthorneCleaving}
I^{(t)}= (f) + (\ell-t) (I_Z:\ell)
\end{equation}
  defines a flat family over $\AA^1 = \Spec ( \Bbbk[t])$.
We have $I^{(0)}=I_Z$ and
$I^{(t)}= (f, \ell-t) \cap (I_Z:\ell)$ for $t \ne 0$.
Thus, 
$Z$ is the limit of a reducible subscheme  
$X^{(t)} \cup Y^{(t)}$ for $t \ne 0$,
where $X^{(t)}=\V(I_Z:\ell)$ is supported at $V(\mm)$
and $Y^{(t)}=\V(f, \ell-t)$ has support disjoint form $\V(\mm)$.
\end{lemma}
\begin{proof}
Up to a change of coordinates, we may assume that $\ell = y$.
Since $f^*, \ell^*\in \mathrm{gr}_\mm(R)$ form a regular sequence,
it follows that  $f$ contains the term $x^{\ord(I_Z)}$. 
The statements now follow from   the proof of \cite[Lemma 8.12]{hartshornedef}.
\end{proof}

Observe that a general linear form $\ell$ will satisfy the assumption of Lemma \ref{LemmaHartshorne}.

\begin{prop}\label{PropositionCleaving}
Every irreducible pair
$[Z_1 \supseteq Z_2]\in\Hilb^{(m_1,m_2)}(\AA^2)$, where $m_1 \geq 2$ and $m_2 \in \{1,2\}$,
 is cleavable.
\end{prop}

\begin{proof}
We may assume that $Z_1$ and $Z_2$ are supported at $\V(\mm)$.
Denote $I= I_{Z_1} \subseteq J=I_{Z_2}$.
We will show that the pair can be cleaved, using the deformation $I^{(t)}$ of $I$ provided by \eqref{EqHarsthorneCleaving},
and appropriately deforming  $J$ to some $J^{(t)}$ such that $I^{(t)} \subseteq J^{(t)}$ for all $t$.

If $m_2=1$, then  $J= \mm$.
 Since $m_1 >1$,  a general linear form $\ell $ will satisfy $\ell \notin I$ and hence $I:\ell \subseteq \mm$.
 We cleave the pair using  \eqref{EqHarsthorneCleaving}  and setting $J^{(t)}=J$.
From now on, assume $m_2=2$.

If $\ord(I)=1$, then $I = (x,y^r) \subseteq J  =(x,y^2)$, up to changing coordinates.
We cleave the pair  setting
$I^{(t)}= (x) + (y-t) (I:y)$
and 
$J^{(t)}= (x) + (y-t) (J:y)=(x, y^2-ty)$.
If $\ord(I)\geq 3$, then $I:\ell \subseteq \mm^2 \subseteq J$ for every $\ell$,
so we may cleave setting $J^{(t)}=J$.
From now on, assume $\ord(I)=2$.

We consider the quadratic part $[I^*]_2$ of $I^*$.
If $\dim_\kk [I^*]_2 = 3$, then $I=\mm^2$, and we may assume $J  =(x,y^2)$.
We cleave the pair  setting
$I^{(t)}= (x^2) + (y-t) (I:y)$
and 
$J^{(t)}= (x) + (y-t) (J:y)$.
If $\dim_\kk [I^*]_2 = 1$, then 
we  choose  a (general) $\ell$ so that it is coprime with the unique quadric $f^*\in [I^*]_2$.
It follows that $(I:\ell)^* \subseteq I^*:\ell \subseteq \mm^2$,
hence $I:\ell \subseteq \mm^2$ and 
we cleave the pair setting $J^{(t)}=J$.
From now on, assume $\dim_\kk[I^*]_2 =2$,
so $I = (q_1, q_2) + I'$ with $I' \subseteq \mm^3$
and $[I^*]_2 = \Span_\kk( q_1^*, q_2^* )$.

If $q_1, q_2$ have a common factor, then, up to changing coordinates, 
we may assume $(q_1, q_2) = x \mm$.
Let $w \in J$ with $\ord(w)=\ord(J)=1$.
We cleave  the pair setting
$I^{(t)}= (xw) + (\ell-t) (I:\ell)$
and 
$J^{(t)}= (w) + (\ell-t) (J:\ell)$
with general $\ell$.

If $q_1, q_2$ are coprime, 
then $(q_1,q_2)$ is  a complete intersection with Hilbert function $(1,2,1,\ldots,1)$.
Using \cite[Corollary 1.2]{S79}, we may  assume
$(q_1, q_2) =(xy, ux^2-y^p)$, where $u\notin \mm$  and $p \geq 2$.
We have  $J=(w)+\mm^2$ for some $w\in \mm$ with $\ord(w)=1$,
and we may assume $w$ to be homogeneous.

If $p \geq 3$, then 
$[I^*]_2 = \Span_\kk( x^2, xy)$.
If $\gcd(x,w)=1$, let $q \in I$ with $q^* = x^2$ and observe that $J= (w,q)$.
We cleave the pair setting 
$I^{(t)}= (q) + (w-t) (I:w)$
and
$J^{(t)}= (q, w-t)$.
Otherwise, we have  $x=w$ up to units,
and we cleave the pair setting 
$I^{(t)}= (xy) + (\ell-t) (I:\ell)$
and
$J^{(t)}= (w) +  (\ell-t) (J:\ell)$.

Finally, suppose $p=2$,
so $[I^*]_2 = \Span_\kk(  xy, u^*x^2-y^2)$.
If $\gcd(x,w)=\gcd(y,w)=1$, then $J = (w,xy)$,
 and we cleave the pair setting
$I^{(t)}= (xy) + (w-t) (I:w)$,
$J^{(t)}= 
 (xy, w-t)$.
Otherwise,
we cleave the pair setting
$I^{(t)}= (xy) + (\ell-t) (I:\ell)$,
$J^{(t)}= (w) + (\ell-t)(J:\ell)$.
\end{proof}

\begin{cor}\label{CorollaryIrreducibleNested}
The nested Hilbert scheme $\Hilb^{(m,2)}(\AA^2)$ is irreducible, of dimension $2m$.
\end{cor}
\begin{proof}
Applying  Proposition \ref{PropositionCleaving}
repeatedly, 
 every pair $[Z_1 \supseteq Z_2]\in\Hilb^{(m,2)}(\AA^2)$ is a limit of  pairs of reduced schemes, that is, 
$\Hilb^{(m,2)}(\AA^2)$ coincides with the smoothable component.
\end{proof}

Corollary \ref{CorollaryIrreducibleNested} implies  that $\Hilb^{(m,2)}(\AA^2)$ is generically smooth, 
since a pair of reduced schemes is a smooth point.
In fact, we can determine the exact codimension of the singular locus.

\begin{prop}\label{PropositionSingularLocus}
The singular locus of    $\Hilb^{(m,2)}(\AA^2)$ has codimension 4, if $m$ is at least $4$.
\end{prop}

\begin{proof}
A pair is a smooth point on   $\Hilb^{(m,2)}(\AA^2)$  if and only if 
all its irreducible components are smooth points on their respective nested Hilbert schemes;
this follows for instance from the tangent space formula \cite[Section 0.4]{Cheah}.
Moreover, any irreducible component of a pair can be deformed to a pair of reduced subschemes using Proposition \ref{PropositionCleaving} repeatedly.

Consider a pair   $[Z_1\supseteq Z_2]\in \Sing(\Hilb^{(m,2)}(\AA^2))$.
At least one of its irreducible components is a singular point on its own nested Hilbert scheme.
The union of the  remaining components can be deformed to a pair of reduced subschemes. 
We deduce that $\Sing(\Hilb^{(m,2)}(\AA^2))$ is contained in  the closure of the locus 
 $\mathcal{L}\subseteq \Hilb^{(m,2)}(\AA^2)$  which parametrizes  pairs $[Z_1\supseteq Z_2]$ satisfying the following condition:
one irreducible component   $[\V(I)\supseteq \V(J)]$  of $[Z_1\supseteq Z_2]$
is a singular point on its nested Hilbert scheme,
while all other components are pairs of reduced subschemes, that is, 
points of either $\Hilb^{(1,1)}(\AA^2)$ or $\Hilb^{(1,0)}(\AA^2)$.
Conversely, we have  $\mathcal{L}\subseteq \Sing(\Hilb^{(m,2)}(\AA^2))$ by the previous paragraph,
and, since the singular locus is closed, we conclude that $\Sing(\Hilb^{(m,2)}(\AA^2))= \overline{\mathcal{L}}$.
Thus, 
it suffices to  show that $\codim(\mathcal{L}) = 4$.

We stratify $\mathcal{L}$ by locally closed subsets $\mathcal{L}_{r,s}$ 
according to  $r= \colength(I), s= \colength(J)$,
where $[\V(I)\supseteq \V(J)]$ is the unique singular component of a pair parametrized by $\mathcal{L}$.
Let  $P\in \AA^2$ denote the support of $\V(I),\V(J)$.
We have $I \subseteq I_P^2$,
since any pair with $I \not\subseteq I_P^2$ is of the form 
$[\V(x,y^r) \supseteq \V(x,y^s)]$
 up to changing coordinates,
and it is easy to check that $[\V(x,y^r) \supseteq \V(x,y^s)]$ is a smooth point 
using the tangent space formula.
We also have $s \geq 1$, since $\Hilb^{(r,0)}(\AA^2) = \Hilb^{r}(\AA^2)$ is smooth,
thus $J \ne R$, equivalently,
$J \subseteq I_P$.
Finally, if $r=3 $, then $s=1$, since $\Hilb^{(3,2)}(\AA^2)$ is smooth.
Thus, the relevant vectors $(r,s)$ are $(3,1)$ and  those  with $r\geq 4$ and $s \in \{1,2\}$.

Let $\mathcal{M}_{r,s}\subseteq \Hilb^{(r,s)}(\AA^2)$ be the locus
of irreducible pairs   $[\V(I)\supseteq \V(J)]$ such
that $I \subseteq I_P^2$, where $P$ is the support of $\V(I)$,
and let $\mathcal{N}_{r,s} = \mathcal{M}_{r,s}\cap \Sing(\Hilb^{(r,s)}(\AA^2))$.
By the discussion above, $\mathcal{L}_{r,s}$ is isomorphic to the following open subscheme of 
$\mathcal{N}_{r,s} \times \Hilb^{(m-r,2-s)}(\AA^2)$
$$
\mathcal{L}_{r,s}
\cong 
\Big\{ \big([Y_2 \supseteq Y_1], [X_2 \supseteq X_1]\big)\,: \, X_2 \text{ is reduced and disjoint from the support of } Y_2  \Big\}.
$$
Since $\dim \Hilb^{(m-r,2-s)}(\AA^2) = 2m-2r$,
in order to show that $\codim(\mathcal{L}) = 4$
it suffices to show that $\codim(\mathcal{N}_{r,s}) \geq 4$ for each $(r,s)$,
and $\codim(\mathcal{N}_{r,s}) = 4$ for at least one $(r,s)$.

If we  fix $P = \V(\mm)$,
the locus of $\mm$-primary ideals $I\subseteq \mm^2$ of colength $r$ 
has dimension at most $r-2$  by \cite[Th\'eor\`eme III.3.1]{Briancon}.
However, if $r=3$ this locus consists of a single point $I=\mm^2$.
The locus of  $\mm$-primary ideals $J$ of colength $2$ 
has dimension 1, 
and the only  $\mm$-primary ideal of colength $1$ is $J= \mm$.
By varying  the point $P\in \AA^2$,
we obtain
$\dim(\mathcal{M}_{r,1}) \leq r$ and
$\dim(\mathcal{M}_{r,2}) \leq r+1$ if $r \geq 4$,
while $\dim(\mathcal{M}_{3,1}) =2$.
Since $\dim \Hilb^{(r,s)}(\AA^2) = 2r$, 
we deduce  $\codim(\mathcal{N}_{s,r}) \geq \codim(\mathcal{M}_{s,r}) \geq 4$ for all $(r,s) $ with $s \in \{1,2\}$ and $(r,s) \ne (4,2)$.
Moreover, 
we have $\mathcal{N}_{3,1}=\mathcal{M}_{3,1} = \{ [V(I_P^2) \supseteq V(I_P)] \, : \, P \in \AA^2\}\subseteq \Sing(\Hilb^{(3,1)}(\AA^2))$,
since every pair of the form $[V(I_P^2) \supseteq V(I_P)]$ is singular, as it follows by the tangent space formula.
Therefore, $\codim (\mathcal{N}_{3,1})=6-2=4$.

It remains to show that $\codim(\mathcal{N}_{4,2}) \geq 4$.
For $(r,s)= (4,2)$, the argument above shows that $\codim(\mathcal{M}_{4,2}) \geq 3$.
To show that $\codim(\mathcal{N}_{4,2}) \geq 4$,
we observe that $\mathcal{M}_{4,2}$ is irreducible and its general member parametrizes smooth points of $\Hilb^{(4,2)}(\AA^2)$.
An $\mm$-primary ideal $I\subseteq \mm^2$ of colength $4$ 
is necessarily homogeneous with Hilbert function $(1,2,1)$,
and these ideals form a locus $\Hilb^4(\AA^2)$ which is irreducible of dimension $2$
by \cite[Th\'eor\`eme III.3.1]{Briancon}.
By accounting for the ideals $J$ and varying the point $P \in\AA^2$,
it follows that $\mathcal{M}_{4,2}$ is irreducible of dimension 5. A simple computation shows that the tangent space to  $\Hilb^{(4,2)}(\AA^2)$ at the pair
 $[\V(x^2,y^2)\supseteq \V(x,y^2)]$ has dimension 8. 
 Thus, the pair is a smooth point in $\Hilb^{(4,2)}(\AA^2)$, and it lies in $\mathcal{M}_{4,2}$. 
 It follows that the general member of $\mathcal{M}_{4,2}$ is a smooth point of $\Hilb^{(4,2)}(\AA^2)$. 
 Thus, $\codim(\mathcal{N}_{4,2}) = 8- \dim(\mathcal{N}_{4,2}) \geq 8 - (\dim(\mathcal{M}_{4,2})-1) = 4$ as required.
\end{proof}

Combining  Corollary \ref{CorollaryIrreducibleNested} and Proposition \ref{PropositionSingularLocus},
we have proved Theorem \ref{TheoremIrreducibleNested}.

\smallskip
We conclude this section with a  discussion about arbitrary nested Hilbert schemes $\Hilb^{\lambda}(\AA^2)$, 
where $\lambda =(m_1, \ldots, m_d)$ is  any integer partition.
The  irreducibility of $\Hilb^{\lambda}(\AA^2)$ has been proved only in few cases,
by various techniques:
 $\lambda=(m,m-1)$ \cite[Theorem 3.0.1]{Cheah}, 
$\lambda=(m,m-2)$ \cite[Proposition 6]{GH}, 
  $\lambda=(m,m-1,m-2)$  \cite[Section 3.A]{Addington}, 
   $\lambda=(m,1)$  \cite[Corollary 7.3]{Fogarty2},
   and $\lambda= (m,m-1,m-2,1),(m,m-2,1),(m,m-1,1)$
   \cite[Theorem 1.1, Corollary 1.2]{RT}.
The technique we employed in this section can be used to study the  irreducibility 
of $\Hilb^{\lambda}(\AA^2)$ in general;
moreover, it works in arbitrary characteristic.
For instance, it can be used to give a quick proof of the irreducibility of $\Hilb^{(m,m-1,m-2,1)}(\AA^2)$, and, as a consequence, of all the other cases mentioned above.

\begin{prop}
The nested Hilbert scheme $\Hilb^{(m,m-1,m-2,1)}(\AA^2)$ is irreducible of dimension $2m$.
\end{prop}
\begin{proof}
First, we show that every irreducible chain $[\V(I_1) \supseteq \cdots \supseteq \V(I_4)] \in \Hilb^{(m,m-1,m-2,1)}(\AA^2)$ supported at $\V(\mm)$ is cleavable.
Let $k = \min \{i : \ord(I_i)=\ord(I_3)\}$, and pick $f_k \in I_k$ with $\ord(f_k)=\ord(I_k)$.
Since $\colength(I_1)-\colength(I_k) \leq 2$, we have $I_1 \cap \Span_\kk( f_k, xf_k, yf_k) \ne \{0\}$.
It follows that we can pick $f_i \in I_i$ for $i = 1, 2, 3$ with
$\ord(f_i)=\ord(I_i)$ and
 $(f_1) \subseteq (f_2)\subseteq (f_3) $.
 Observe that $I_4 = \mm$.
We cleave the chain by setting $I^{(t)}_i = (f_i ) + (\ell-t)(I_i : \ell)$ for $i =1,2, 3$, 
and $I^{(t)}_4 = (f_3,\ell-t)$ if $\ord(f_3)=1$,
$I^{(t)}_4 = \mm$ if $\ord(f_3)>1$,
where $\ell$ is a general linear form.
As a byproduct, we also deduce cleavability  for any subset of $(m,m-1,m-2,1)$.
The vector of lengths of each irreducible component of the chain $[\V(I^{(t)}_1) \supseteq \cdots \supseteq \V(I^{(t)}_4)]$, for $t \ne 0$, is again of the form $(m',m'-1,m'-2,1)$ for some $m'$, or a subset of it.
Cleaving each irreducible chain repeatedly, we conclude that  $\Hilb^{(m,m-1,m-2,1)}(\AA^2)$ coincides with its smoothable component.
\end{proof}

However, 
very little is known besides these cases, 
and achieving a complete classification is perhaps hopeless.
Intuitively, 
the issue is that
deforming two or more subschemes simultaneously, 
while preserving the inclusion throughout the deformation, is much harder than just deforming one subscheme.
The last result of this section  provides concrete evidence to support this claim.
It is known that examples of reducible nested Hilbert schemes exist, cf. \cite[Section 3.A]{Addington} and \cite[Theorem 1.4]{RT}.
Here, we  apply a variation of Iarrobino's method in \cite{Iarrobino} to show that
$\Hilb^\lambda(\AA^2)$ may be reducible as soon as  $\lambda$ has at least $5$ parts.

\begin{prop}\label{PropositionReducibleNested}
For each $d\geq 5$, there exist integers  $m_1 > \cdots > m_d$
such that  the nested Hilbert scheme $\Hilb^{(m_1, \ldots, m_d)}(\AA^2)$ is reducible.
\end{prop}

\begin{proof}
Fix $d \geq 5$, and consider integers  $r\geq d$.
Let $\lambda = (m_1,\ldots, m_d)$  be the partition 
defined by $m_i = {r+1-i \choose 2}  + \lfloor \frac{r+1-i}{2}\rfloor$.
We consider the following locus $\mathcal{F}\subseteq\Hilb^\lambda(\AA^2)$
$$
\mathcal{F}= \big\{  [\V(I_1) \supseteq \cdots \supseteq \V(I_d)] \, : \, \mm^{r+1-i} \subseteq I_i \subseteq \mm^{r-i} \big\}.
$$
For every $i = 1, \ldots, d$, the ideal $I_i$ is homogeneous and uniquely determined by 
its graded component $W_i = [I_i]_{r-i} \subseteq [R]_{r-i}$.
Conversely, any choice $(W_1, \ldots, W_d)$ of such subspaces defines a point in $\mathcal{F}$.
To summarize, $\mathcal{F}$ has a  parametrization by the product of Grassmannians
$$
\mathcal{G} = \prod_{i=1}^d\Gr\left(\,\left\lceil \frac{r+1-i}{2}\right\rceil\,,\,r+1-i\right)
$$
since $\dim_\kk [R]_{r-i} = r+1-i$ and  $\dim_\kk [I_i]_{r-i} = \dim_\kk(R/\mm^{r+1-i}) - \dim_\kk(R/I_i) = \lceil \frac{r+1-i}{2}\rceil$.
Thus,
$$
 \dim \mathcal{F} = \dim \mathcal{G} = \sum_{i=1}^d \left\lceil \frac{r+1-i}{2}\right\rceil \cdot \left\lfloor \frac{r+1-i}{2}\right\rfloor
 \geq 
 \sum_{i=1}^d  \frac{(r+1-i)(r-i)}{4} 
=: f(d,r).
$$
Thus, a sufficient condition for $\Hilb^\lambda(\AA^2)$ to be reducible is 
$f(d,r) > r^2 \geq 2m_1$,
the dimension of the smoothable component. 
Since $f(d,r)$ has leading term $\frac{d}{4}r^2$ and $d \geq 5$,
this will happen for $r \gg 0$.
\end{proof}

It is likely that  loci of dimension greater than $2 m_1$ exist in $\Hilb^\lambda(\AA^2)$  for some partition  $\lambda $ with $d=4$ parts;
in fact, the one we construct in the proof of Proposition \ref{PropositionReducibleNested} has dimension very close to $2m_1$.
On the other hand, we believe that the most interesting case is  $d=2$, so we propose  the following problem.

\begin{question}
Is the nested Hilbert scheme $\Hilb^{(m_1,m_2)}(\AA^2)$ irreducible for every $m_1>m_2$?
\end{question}

\section{Reduction to compressed pairs}\label{SectionReductionCompressed}

Let $n\in \N$. 
The {\bf $n$-th compressed pair} is  $C_n=[\V(\mm^n)\supseteq \V(x,y^2)] \in \Hilb^{(m,2)}(\AA^2)$,
where
 $m = {n +1\choose 2}$.
The goal of this section is to prove the following reduction to compressed pairs.

\begin{thm}\label{TheoremReductionToCompressedPair}
Let $n\in \N$ and $m = {n+1 	\choose 2}$.
Assume $\ch(\kk)=0$.
 If  $\Hilb^{(m,2)}(\AA^2)$ has a rational singularity at  $C_n$, then $\Hilb^{(n,2)}(\AA^2)$ has rational singularities.
\end{thm}

First, by means of   generic initial ideals, we construct suitable degenerations to $C_n$.
In the next two results we use  the graded reverse lexicographic order, denoted by $\rev$,  as the term order in $R=\Bbbk[x,y]$. 
See \cite[Chapter 15]{Eisenbud} for details on Borel-fixed ideals and generic initial ideals.

\begin{lemma}\label{LemmaInitialOnePoint}
Let $Z \subseteq \AA^2$ be a finite subscheme 
and $Q \in \AA^2$ a general point.
Assume that $B = \iin(I_Z)\subseteq R$ is Borel-fixed.
Let $B'$ be the  Borel-fixed ideal $B'\subseteq B$ such that the vector space
$B/B'$ has dimension 1 and is spanned by the lowest monomial in  $B$ with respect to $\rev$.
Then $\iin(I_{Z}\cap I_Q) = \iin(B\cap I_Q) = B'$.
\end{lemma}
\begin{proof}
Let $n$ denote the length of $Z$,
so that $\colength(B) = n$ and $\colength(B') = n+1$.
Since $Q$ is general, we may assume that $Q$ does not belong to the supports of $Z$ or $ \V(B)$.
In particular, $\colength(I_Z\cap I_Q) =\colength(B\cap I_Q) = n+1$.

First, we prove that $\iin(B \cap I_Q) = B'$.
Let $\bfu \in B$ be the lowest monomial in $B$ with respect to $\rev$.
Since $Q$ is a general point,
 we may assume that the evaluation of $\bfu$ at $Q$ is non-zero.
For every other monomial $\bfv \in B$, 
that is, for every $\bfv \in B'$, consider the polynomial
$
f_\bfv = \bfv - \alpha_{\bfv} \mathbf{u} 
$
where $\alpha_\bfv\in \kk$ is the unique scalar such that
$f_\bfv(Q)=0$.
Note that $f_\bfv \in B \cap I_Q$, 
and  $\LM(f_\bfv) = \bfv$ by the choice of $\bfu$.
It follows  that 
 $B' \subseteq \iin(B \cap I_Q) $,
 but both ideals have  colength $n+1$,
so $\iin(B \cap I_Q) = B'$ as desired.

Now, we prove that  $\iin(I_{Z}\cap I_Q) = B'$.
Both are monomial ideals of colength $n+1$  contained in $B = \iin(I_Z)$, which has colength $n$.
Therefore, it suffices to show that the vector space $B/\iin(I_Z\cap I_Q)$ is spanned by the monomial $\bfu$, equivalently, that
 $\bfu \notin \iin(I_{Z}\cap I_Q)$.
 By the choice of $\bfu$, 
there exists a unique monic polynomial $f_\bfu \in I_Z$ such that $\LM(f_\bfu) = \bfu$.
Since $Q$ is general, we may assume that $f_\bfu(Q)\ne 0$, that is, 
$f_\bfu \notin I_{Z}\cap I_Q$.
Thus $I_{Z}\cap I_Q$ contains no (monic) polynomial whose leading monomial is $\bfu$, 
so $\bfu \notin  \iin(I_{Z}\cap I_Q)$ as desired.
\end{proof}

\begin{prop} \label{PropositionGrobnerDegenerationToCompressed}
Let $[Z_1\supseteq Z_2]\in\Hilb^{(n,2)}(\AA^2)$ be a pair, and let $m = {n+1 	\choose 2}$.
 There is a rational curve $\gamma:\AA^1 \rightarrow \Hilb^{(m,2)}(\AA^2)$
satisfying the following conditions:
\begin{enumerate}
\item   $\gamma(1) = [Z_1\cup Y\supseteq Z_2]$, where  $Y\subseteq \AA^2$ is reduced and disjoint from the support of $Z_1$;
\item  $\gamma(0) $ is the compressed pair $C_n$;
\item  the local rings $\mcO_{\Hilb^{(m,2)}(\AA^2), \gamma(t)}$ are all isomorphic for $t \ne 0$.
\end{enumerate}
\end{prop}
\begin{proof}
Consider the Borel-fixed ideals $B= \mathrm{gin}(I_{Z_1})$ and $ \mathrm{gin}(I_{Z_2})=(x,y^2)$.
Up to changing coordinates, we may assume that 
$B= \iin(I_{Z_1})$ and $ \iin(I_{Z_2})=(x,y^2)$.
Let $Y = \{Q_1, \ldots, Q_{m-n}\}\subseteq \AA^2$ be a general set of reduced points.
We claim that $\iin(I_{Z_1\cup Y}) = (x,y)^n$.
Since $I_{Z_1\cup Y} = I_{Z_1} \cap I_{Q_1} \cap \cdots \cap I_{Q_{m-n}}$, we may compute $\iin(I_{Z_1\cup Y})$ by applying Lemma \ref{LemmaInitialOnePoint} repeatedly.
We deduce that  $\iin(I_{Z_1\cup Y})$ is obtained from $B$ by removing $m-n$ monomials in the lowest possible degrees.
Since $Z_1$ has length $n$, we have $\mm^n \subseteq \iin(I_{Z_1})$, 
and, since $\mm^n$ has colength $m$, we conclude that $\iin(I_{Z_1\cup Y}) = \mm^n$.

By \cite[Proposition 15.16, Exercise 15.12]{Eisenbud},
there exists a  weight $w\in \mathbb{Z}^2$
such that $\iin_w(I_{Z_1 \cup Y})  = (x,y)^n $ 
and $\iin_w(I_{Z_2}) =  (x,y^2)$.
The Gr\"obner degeneration 
\cite[Theorem 15.17]{Eisenbud}
yields two ideals $I_1^{(t)}, I_2^{(t)} \subseteq R[t]$
which define flat families over $\AA^1 = \Spec ( \Bbbk[t])$ and
such that $I_1^{(1)}=I_{Z_1 \cup Y} \subseteq  I_2^{(1)} = I_{Z_2}$
and
$I_1^{(0)}=(x,y)^n \subseteq  I_2^{(0)} = (x,y^2)$.
They are obtained by homogenizing with respect to  $w$,
and this implies $I_1^{(t)}\subseteq  I_2^{(t)}$  for all $t$, 
defining thus a curve $\gamma:\AA^1  \rightarrow \Hilb^{(m,2)}(\AA^2)$.
Finally, for every $t \ne 0$  the automorphism of $\AA^2$ 
defined by $x \mapsto t^{w_1}x,\, y \mapsto t^{w_2}y$
induces an automorphism of   $\Hilb^{(m,2)}(\AA^2)$
taking the pair $[\V(I_1^{(t)})\supseteq \V(I_2^{(t)})]$ to $[\V(I_1^{(1)})\supseteq \V(I_2^{(1)})]$,
so  item (3) follows.
\end{proof}

We now track how  singularities change along this curve.

\begin{lemma}\label{LemmaCompleteLocalRings}
Let $[Z_1\supseteq Z_2]\in\Hilb^{(n_1,n_2)}(\AA^2)$ and $[Y_1 \supseteq Y_2] \in \Hilb^{(m_1,m_2)}(\AA^2)$ be pairs such that $Z_1$ and $Y_1$ are disjoint. There is an isomorphism of complete local rings
\begin{equation}
\label{disjointisom}
\widehat{\mcO}_{\Hilb^{(n_1+m_1,n_2+m_2)}(\AA^2), [Z_1 \cup Y_1 \supseteq Z_2 \cup Y_2]} \cong 
	\widehat{\mcO}_{\Hilb^{(n_1,n_2)}(\AA^2) \times \Hilb^{(m_1,m_2)} (\AA^2), ([Z_1 \supseteq Z_2],[Y_1 \supseteq Y_2])} .
\end{equation}
\end{lemma}

\begin{proof}  
We will prove this lemma by studying the associated functor of Artin rings. 
A standard reference for this is \cite[Chapter 15]{hartshornedef}; we will closely follow the notation used in that chapter.  
We begin with the case when $Y_2, Z_2 = \emptyset$ and prove a more general statement.  Let $\mathcal{C}$ denote the category of local, Artinian $\kk$-algebras. 
For an integer $k$ and a sequence of subschemes $W_1,\dots,W_k \subseteq \AA^s$, 
let $F_{W_1,\dots,W_k}:\CC \to \text{Sets}$ denote the local Hilbert functor that maps 
$$
E \mapsto \{(\mcX_1,\dots,\mcX_k): \mcX_i \subseteq \AA^s_{E} \text{ is closed, flat over } \Spec(E) \text{ and} \, (\mcX_i)_{\kk} = W_i\},
$$
where $(\mcX_i)_{\kk}$ denotes the special fiber.  Let $W_1,W_2 \subseteq \AA^s$ be disjoint subschemes. We claim that the natural transformation $\Psi:F_{W_1,W_2} \to F_{W_1 \cup W_2}$ that maps
$$
\Psi(E): F_{W_1,W_2}(E) \to F_{W_1 \cup W_2}(E), \quad (\mcX_1,\mcX_2) \mapsto \mcX_1 \sqcup \mcX_2
$$
is an isomorphism.
Given $(\mcX_1,\mcX_2) \in F_{W_1,W_2}(E)$, we have $(\mcX_i)_{\text{red}} = ((\mcX_i)_{\kk})_{\text{red}}  = (W_i)_{\text{red}}$ since $E$ is Artinian, and this implies that $\mcX_1 \cup \mcX_2$ is a disjoint union. 
Moreover, since each $\mcX_i$ is flat over $\Spec(E)$, the disjoint union $\mcX_1 \sqcup \mcX_2$ is flat over $\Spec(E)$; 
thus, $\Psi(E)$ is a well defined map. 
Now, assume that $\Psi(E)(\mcX_1,\mcX_2) = \Psi(E)(\mcX_1',\mcX_2')$. 
Since $\mcX_i$ and $\mcX_i'$ are supported on $(W_i)_{\text{red}}$, we must have $\mcX_i' = \mcX_i$; thus, $\Psi(E)$ is injective. 
Finally, if $\mcX \in \Psi(E)(W_1,W_2)$, then it must be of the form $\mcX_1 \cup \mcX_2$ with $\mcX_i$ supported on $(W_i)_{\text{red}}$. 
Since $\mcX$ is flat over $\Spec(E)$, each of the disjoint components $\mcX_1$ and $\mcX_2$ must be flat over $\Spec(E)$. 
Thus, $(\mcX_1,\mcX_2) \in F_{W_1,W_2}(E)$ and $\Psi(E)(\mcX_1,\mcX_2) = \mcX$, verifying surjectivity.  It follows that the formal functors $\widehat{F}_{W_1,W_2}$ and $\widehat{F}_{W_1 \cup W_2}$ are isomorphic. 
This implies that the objects $\widehat{\mcO}_{\Hilb (\AA^s),[W_1 \cup W_2]}$ and $\widehat{\mcO}_{\Hilb(\AA^s) \times \Hilb (\AA^s),([W_1],[W_2])}$,
  which pro-represent $F_{W_1 \cup W_2}$ and $F_{W_1,W_2}$, respectively, are isomorphic \cite[Proposition 23.3]{hartshornedef}. 

The analysis above extends naturally to the nested case.  
Given a sequence of subschemes $W_1,W_1',\dots,W_k,W_k' \subseteq \AA^s$ with $W_i \supseteq W_i'$,  define the functor $F_{\{[W_i \supseteq W_i']\}_{i=1}^k}:\CC \to \text{Sets}$ that maps 
$$
E \mapsto 
	\{([\mcX_i  \supseteq \mcY_i])_{i=1}^k:  \mcY_i\subseteq \mcX_i \subseteq \AA^n_{E} \text{ are closed, flat over } \Spec(E) \text{ and} \, (\mcX_i)_{\kk} = W_i, (\mcY_i)_{\kk} = W_i'\}.
$$
The argument in the first paragraph shows that if $W_1,W_2$ are disjoint, the formal functors $\widehat{F}_{\{[W_i \supseteq W_i']\}_{i=1}^2}$ and $\widehat{F}_{[W_1 \cup W_2 \supseteq W_1' \cup W_2']}$ are isomorphic. 
By \cite[Proposition 23.3]{hartshornedef}, the associated pro-objects are isomorphic.
The statement of the lemma follows by taking $k=s=2$, $W_i = Z_i$ and $W_i' = Y_i$.
\end{proof}

\begin{proof}[Proof of Theorem \ref{TheoremReductionToCompressedPair}]
Assume that $\Hilb^{(m,2)}(\AA^2)$ has a rational singularity at $C_n$, 
and consider an arbitrary pair $[Z_1 \supseteq Z_2]\in \Hilb^{(n,2)}(\AA^2)$. 
We use the curve $\gamma:\AA^1 \rightarrow \Hilb^{(m,2)}(\AA^2)$ of  Proposition \ref{PropositionGrobnerDegenerationToCompressed}.
By Lemma \ref{LemmaPropertiesRationalSingularities} (1), 
$\Hilb^{(m,2)}(\AA^2)$ has a rational singularity at some  $\gamma(t)$ with $t\ne 0$,
and  by Proposition \ref{PropositionGrobnerDegenerationToCompressed} (3),
this is the case for every $t$,
in particular at the pair $\gamma(1) = [Z_1 \cup Y \supseteq Z_2]$.
Using Lemma \ref{LemmaPropertiesRationalSingularities} (4) and Lemma \ref{LemmaCompleteLocalRings}, 
we see that  $ \Hilb^{(n,2)} (\AA^2) \times \Hilb^{m-n}(\AA^2)$ has a rational singularity 
at $\big([Z_1\supseteq Z_2],[Y]\big)$. 
Since $\Hilb^{m-n}(\AA^2)$ is smooth,  Lemma \ref{LemmaPropertiesRationalSingularities} (3) implies that $\Hilb^{(n,2)} (\AA^2)$ has a rational singularity at $[Z_1\supseteq Z_2]$.
\end{proof}

\section{Local equations   around compressed pairs}\label{SectionLocalEquations}

The goal of this section is to describe scheme-theoretic equations of an affine open neighborhood of the compressed pair 
 $C_n=[\V(\mm^n)\supseteq \V(x,y^2)]$ in the nested Hilbert scheme $\Hilb^{(m,2)}(\AA^2)$.
The main result is Theorem \ref{TheoremNeighborhoodL},
which produces a neighborhood in the form of a closed subscheme $\V(\mfL) \subseteq \Spec (B)$
embedded in an affine open subset $\Spec(B)\subseteq  \Hilb^m(\AA^2) \times \Hilb^2(\AA^2)$
 and defined in terms of  Gr\"obner strata.
Then we focus on the restriction $\V(\mfI)$ of this neighborhood 
to the fiber of the natural map $\Hilb^{(m,2)}(\AA^2) \rightarrow   \Hilb^2(\AA^2)$
over the point $[\V(x,y^2)]$.

This section marks a transition, 
from the study of  nested Hilbert schemes in general, 
to  the two affine schemes $\V(\mfL)$ and $\V(\mfI)$, 
which will be the main objects  for the rest of the paper. 
As such, the rings and ideals introduced here will appear again in the subsequent sections.
Specifically, $\mfI$ is the  focus of Sections \ref{SectionGeometricTechnique} and \ref{SectionSquarefreeInitial}, 
while $\mfL$ is the  focus of Section \ref{SectionProofMainTheorem}.
 
For the purposes of the section,
while we still work in the affine plane, 
it is convenient to switch to projective coordinates,
 and consider $\AA^2 = \Spec(R)$ as an open subscheme of $ \P^2 = \Proj\, P$, 
 where $R=\kk[x,y], P = \kk[x,y,z]$.
Likewise, we consider $\Hilb^r (\AA^2)$ and $\Hilb^{(m,2)}(\AA^2)$ as open subschemes of $\Hilb^r (\P^2)$ and  $\Hilb^{(m,2)}(\P^2)$, respectively.
When we write $[\V(I)]\in\Hilb^r (\P^2)$, we implicitly  assume that $I\subseteq P$ is a saturated homogeneous ideal.

Let $n\in\N$ and $m = {n + 1\choose 2}$.
The Hilbert scheme $\Hilb^m(\P^2)$ is smooth and irreducible of dimension $2m = n^2+n$.
The ideal $ \mm^n = I_n(\bfX)\subseteq R$ is generated by the maximal minors of an $(n+1)\times n$ syzygy matrix
$$
\bfX = 
\begin{bmatrix}
y & 0 &   \cdots & 0\\
-x & y &   \cdots  & 0\\
0 & -x &     \cdots & 0 \\
\cdots & \cdots  & \cdots  & \cdots \\
0 & 0 &  \cdots  & y \\
0 & 0 &  \cdots  & -x \\
\end{bmatrix}.
$$
By \cite{Schaps}, see also \cite[Theorem 8.3]{hartshornedef},
deformations of $\V(\mm^n)$ correspond to  deformations of the Hilbert-Burch matrix  $\bfX$.
More precisely, we have  the following lemma.

\begin{lemma}\label{LemmaNeighborhoodsClassicalHilbertScheme}
Let $n\in\N$ and
 $m = {n + 1\choose 2}$.
 Consider the graded reverse lexicographic order on $P$.
The locus 
 $$
 \mcW = \big\{ [\V(I)] \in \Hilb^m(\P^2) \, : \, \iin(I)=(x,y)^n\big\} 
 $$
is an open subscheme of $\Hilb^m(\AA^2)$.
There is an isomorphism 
$\AA^{n^2+n}=\Mat(n+1,n)  \cong \mcW$
defined by 
$$
\bfW\mapsto \big[\V(I_n(\bfX+z\bfW))\big].
$$
Furthermore, the locus 
 $$
 \mcV = \big\{ [\V(I)] \in \Hilb^2(\P^2) \, : \, \iin(I)=(x,y^2)\big\} 
 $$
is an open subscheme of $\Hilb^2(\AA^2)$.
There is an isomorphism 
$\AA^{4}  \cong \mcV$
defined by 
$$
(v_1, v_2, v_3, v_4)\mapsto \big[(x+v_1y+v_2z,  y^2+v_3yz+v_4z^2)\big].
$$
\end{lemma}

\begin{proof}
We only present the proof for $\mcW$, as the one for $\mcV$ is analogous.
The locus $\mcW$ is a locally closed subscheme of $\Hilb^m(\P^2)$ by 
\cite[Theorem 2.1]{NS}.
If $I\subseteq P$ is  homogeneous  with $\iin(I)=(x,y)^n$, then $z$ is a non-zerodivisor on $P/I$,
thus $\V(I) \subseteq \AA^2$ and in fact $\mcW \subseteq \Hilb^m(\AA^2)$.
Denote the  map of the statement 
by 
$\zeta : \Mat(n+1,n)  \rightarrow \mcW$,
defined by 
$\zeta(\bfW)=\big[\V(I_n(\bfX+z\bfW))\big]$.
It is shown in  \cite[Theorem 6.8]{Constantinescu} that $\zeta$
is well defined  and  bijective.
It follows that $\mcW$ has codimension  0 in $\Hilb^m(\AA^2)$, so it is a smooth open subscheme. 

We show that $\zeta$ is in fact an isomorphism.
First, we claim that $\zeta$ induces  isomorphisms on tangent spaces.
The group $\kk^{\ast}$ acts on $\Mat(n+1,n)$ by $\bfW \mapsto c\bfW$,
and    on $P$ by  $x\mapsto x,y\mapsto y, z \mapsto cz$, for  $c \in \kk^{\ast}$. 
This induces an action of $\kk^{\ast}$ on $\mathcal{W}$,
 and $\zeta$ is an equivariant morphism.
 Since the origin $\mathbf{0} \in \Mat(n+1,n)$ lies in the closure of every orbit, 
 by upper semicontinuity it suffices to show that  $\text{d} \zeta_\mathbf{0}: \text{T}_{\mathbf{0}} \Mat(n+1,n) \to \text{T}_{[(x,y)^n]} \Hilb^m(\P^2)$ is an isomorphism. 
We have the identification 
$
\text{T}_{\mathbf{0}}\Mat(n+1,n) 
= \Hom(\kk[w_{i,j}],\kk[\epsilon]/(\epsilon^2)) 
= \Span_\kk( \gamma_{1,1}, \ldots, \gamma_{n+1,n})$ 
where $\gamma_{i,j}$  maps $ w_{i,j} \mapsto \epsilon$ and all other variables to $0$. 
Analogously, let $\bfE_{i,j}$ be the  matrix with the only non-zero entry being $\epsilon$ in the $(i,j)$-th position. 
Then  $I_n(\bfX+z\bfE_{i,j})$ is a flat deformation of $(x,y)^n$ over $\kk[\epsilon]/(\epsilon^2)$. 
In particular, the collection $\{I_n(\bfX+z\bfE_{i,j})\}_{i,j}$ 
is a basis for $ \text{T}_{[(x,y)^n]} \Hilb^m(\P^2)$ \cite[Proposition 2.3]{hartshornedef}. 
It follows that $d\zeta_{\mathbf{0}}$ is an isomorphism since it maps $\gamma_{i,j}$ to $I_n(\bfX+z\bfE_{i,j})$.

The fact that $\zeta$ is an isomorphism  follows  now from Zariski's Main Theorem. 
Specifically,
by
Grothendieck's form of 
Zariski's Main Theorem  \cite[III.9 (IV)]{Mumford},
there is a factorization $\zeta : \Mat(n+1,n)  \xrightarrow{\,\iota\,} \mathcal{Y}\xrightarrow{\,\overline{\zeta}\,} \mcW$,
where $\iota$ is a dense open immersion and $\overline{\zeta}$ is a finite morphism.
Since $\zeta$ is bijective, 
 $\overline{\zeta}$ must be surjective and generically injective.
 It follows by \cite[Theorem 14.9]{Harris}
 that  $\overline{\zeta}$ is an isomorphism on an open subset of $\mathcal{Y}$,
 i.e., it is
 birational. 
Since $\mcW$ is normal, $\overline{\zeta}$ must be an isomorphism by the original form of 
Zariski's Main Theorem \cite[III.9 (I)]{Mumford}.
Since $\zeta = \overline{\zeta}\circ \iota$ is surjective and $\iota$ is an open immersion, it follows that $\zeta$ is an isomorphism.
\end{proof}

\begin{definition}\label{DefinitionRingsAB}
Let $n \in \N$ and define the polynomial rings
$$
A= \kk[w_{i,j}] = \kk[w_{1,1},\ldots, w_{n+1,n}],
\quad
B = A \otimes_\kk \kk[v_1, v_2, v_3, v_4].
$$
Denote by $\bfW=(w_{i,j})$  the generic $(n+1)\times n $ matrix whose entries are the variables of $A$.
\end{definition}

Thus, we have identifications 
$$ \Spec (A) = \Mat(n+1,n) \cong \mcW \subseteq \Hilb^m(\AA^2)
\quad
\text{and}
\quad
\Spec(B) \cong \mcW \times \mcV  \subseteq \Hilb^m(\AA^2) \times \Hilb^2(\AA^2),$$
where $\mcW, \mcV$ are the open subsets of Lemma \ref{LemmaNeighborhoodsClassicalHilbertScheme}.
The points $[\V(\mm^n)]$ and $C_n$ correspond to the origins of the affine spaces $\Spec (A)$ and $\Spec  (B)$, respectively.

\begin{definition}\label{DefinitionRingT}
Define the polynomial ring
$T = B \otimes_\kk P=\Bbbk[w_{i,j},v_h,x,y,z] $.
 Define a bigrading $\bideg(\cdot)= (\deg_1(\cdot), \deg_2(\cdot))$ on $T$ by setting
$$
  \deg_1(x) = \deg_1(y) = \deg_1(z)=1, \quad \deg_1(w_{i,j}) =\deg_1(v_h) =0,
  $$
$$
\deg_2(x) = \deg_2(y) = \deg_2(w_{i,j}) =\deg_2(v_2)=\deg_2(v_3) =1, \,\,
\deg_2(z) = \deg_2(v_1)=0, \,\, \deg_2(v_4)=2.
$$
Compare bidegrees by the lexicographic order on $\N^2$. 

Define a term order on $T$ as follows.
Fix an ordering of the variables such that $x > y > z > w_{i,j}, v_h$ for all $i,j,h$.
We consider the pure lexicographic order induced by this ordering and refined by the bigrading $\bideg(\cdot)$.
Equivalently, given two monomials $\bfu, \bfu' \in T$,
we set $\bfu > \bfu'$ if:
\begin{itemize}
\item $\deg_1(\bfu) > \deg_1(\bfu')$, or
\item $\deg_1(\bfu) = \deg_1(\bfu')$ and $\deg_2(\bfu) > \deg_2(\bfu')$, or
\item  $\bideg(\bfu) = \bideg(\bfu')$ and $\bfu >\bfu'$ 
in the pure lexicographic order.
\end{itemize}
\end{definition}

We define the following polynomials of $T$:
\begin{equation*}
\begin{split}
\Gamma_1&= x+v_1y+v_2z, \qquad\Gamma_2 = y^2+v_3yz+v_4z^2,\\
\qquad
\Delta_i &= \text{ maximal minor of } \bfX+z\bfW \text{ obtained by deleting row }i.
\end{split}
\end{equation*}
By  Lemma \ref{LemmaNeighborhoodsClassicalHilbertScheme},
the ideal 
$I_n(\bfX+z\bfW)\subseteq T$
defines the universal family  $U_m \subseteq \P^2\times \Hilb^m(\P^2)$, restricted to the open set $\P^2 \times \mcW$, and then  extended to $\P^2 \times \mcW \times \mcV$. 
Likewise, the ideal $(\Gamma_1,\Gamma_2)\subseteq T$
defines the universal family  $U_2 \subseteq \P^2\times \Hilb^2(\P^2)$, restricted to the open set $\P^2 \times \mcV$, and then extended to $\P^2 \times \mcW \times \mcV$. 
An important  observation is  that these  ideals are bigraded.
The grading $\deg_1(\cdot)$ is inherited from the natural grading of $\P^2$ and, as such,
it is the correct grading to use in the proof of Theorem \ref{TheoremNeighborhoodL}.
On the other hand,
the  grading $\deg_2(\cdot)$ is constructed in such a way that the variables of $B$ also play a non-trivial role in the bigrading.
Keeping track of both gradings simultaneously will allow us to show that both ideals of interest $\mfL$ and $\mfI$ are graded (with respect to $\deg_2(\cdot)$),
a fact that will be used extensively in the following sections.

The polynomials
$\Gamma_1$ and $ \Gamma_2$ 
have coprime leading monomials, 
hence, they form a Gr\"obner basis with initial ideal $(x,y^2)\subseteq T$.
We reduce $\Delta_i$ modulo  $\{\Gamma_1,\Gamma_2\}$. 
Since the term order is compatible with the bigrading by construction, 
the division algorithm 
\cite[Section 15.3]{Eisenbud}
yields a bigraded reduction equation
\begin{equation}\label{EqDivisionAlgorithmDelta}
\Delta_i = \alpha_i (x+v_1y+v_2z)  + \beta_i (y^2+v_3yz+v_4z^2) + g_iyz^{n-1}+G_iz^{n}
\end{equation}
such that  $\LM(\alpha_i x), \LM(\beta_i y^2) \leq \LM(\Delta_i),$
and  $\rho = g_iyz^{n-1}+G_iz^n$ is the unique remainder.
In particular, no term in $\rho$ is divisible by $x$ or  $y^2$.
Since $\Delta_i$ is bigraded  with $\bideg(\Delta_i)=(n,n)$,
it follows that  $g_i, G_i$ are bigraded of bidegree $(0,n-1)$ and $(0,n)$, respectively, 
and that $g_i, G_i \in B$.
We define the ideal
\begin{equation}\label{EqDefinitionIdealL}
\mfL =(g_1, \ldots, g_{n+1}, G_1, \ldots, G_{n+1})\subseteq B.
\end{equation}
We point out that $\mfL$ is well defined because of the uniqueness of the remainder in the division by a Gr\"obner basis.

\begin{thm}\label{TheoremNeighborhoodL}
Let $n\in\N$ and
 $m = {n + 1\choose 2}$.
There is an  open subscheme of $ \Hilb^{(m,2)}(\AA^2)$ containing the compressed pair $C_n$ and
isomorphic to  $\V(\mfL) \subseteq \Spec( B) $.
\end{thm}

\begin{proof}
We follow the construction  of the nested Hilbert scheme  in 
\cite[Theorem 4.5.1]{S06}.
It is explicitly realized as a vanishing scheme  $\V(\varphi) \subseteq \mathrm{H}_1 \times \mathrm{H}_2$,
where $\mathrm{H}_1=\Hilb^m(\P^2), \mathrm{H}_2= \Hilb^2(\P^2)$, 
and  $\varphi$ is a map of locally free sheaves on $\mathrm{H}_1 \times \mathrm{H}_2$, which we now describe.

Let $\mu $ be a sufficiently large integer so that the ideal sheaf of every  $[Z]\in \Hilb^m(\P^2)$ is $\mu$-regular.
We use the symbol  $[\cdot]_d$ to
denote  graded components of degree $d$.
Let $H^0(\P^2, \mathcal{O}_{\P^2}(\mu)) = [P]_\mu$,
denote by $q: \P^2 \times \mathrm{H}_1 \times \mathrm{H}_2 \rightarrow \mathrm{H}_1 \times \mathrm{H}_2$  the projection map, 
and let $\mathcal{I}_i$ be the ideal sheaf of the universal family on $\mathrm{H}_i$ extended to $ \P^2 \times \mathrm{H}_1 \times \mathrm{H}_2$.
The map $\varphi$ is the composition of natural maps
$$
\varphi : q_\star \mathcal{I}_1(\mu) \subseteq [P]_\mu \otimes_\kk \mathcal{O}_{\mathrm{H}_1 \times \mathrm{H}_2} \rightarrow [P]_\mu \otimes_\kk \mathcal{O}_{\mathrm{H}_1 \times \mathrm{H}_2}/q_\star \mathcal{I}_2(\mu).
$$

In order to prove the theorem, we verify that
$\V(\varphi)\cap \Spec (B) = \V(\mfL)$,
where
$\Spec (B) = \mcV \times \mcW \subseteq \mathrm{H}_1 \times \mathrm{H}_2$ is the open subscheme  obtained from Lemma \ref{LemmaNeighborhoodsClassicalHilbertScheme}.
Thus, we restrict to $\Spec (B)$,
and  we can identify  $\mathcal{I}_1$  with $I_n(\bfX+z\bfW)$,
 $\mathcal{I}_2$ with $(\Gamma_1,\Gamma_2)$,
and $\varphi$ with the map of  free $B$-modules
$$
\varphi : 
[I_n(\bfX+z\bfW)]_\mu  \otimes_\kk B
\subseteq  [P]_\mu \otimes_\kk B
\rightarrow [P]_\mu  \otimes_\kk B/ [(\Gamma_1,\Gamma_2)]_\mu
=  [T/(\Gamma_1,\Gamma_2)]_\mu,
$$
where we use the grading $\deg_1(\cdot)$ for $T$,
so that $[(\Gamma_1,\Gamma_2)]_\mu \subseteq [T]_\mu = [P]_\mu  \otimes_\kk B$.
The target $[T/(\Gamma_1,\Gamma_2)]_\mu$ is a free $B$-module with basis $yz^{\mu-1}, z^{\mu}$,
hence $\varphi$ is represented by a matrix $\Phi$ with two rows and with entries in $B$.
Each column of $\Phi$ corresponds to a generator $h \in [I_n(\bfX+z\bfW)]_\mu$,
and its two entries are the coefficients of $yz^{\mu-1}, z^{\mu}$
in the reduction of $h$ modulo $ [(\Gamma_1,\Gamma_2)]_\mu$.
We need to show that $\mfL = I_1(\Phi)$, the ideal of entries of $\Phi$.

Since 
$\{\Gamma_1,\Gamma_2\}$ is a Gr\"obner basis with initial ideal $(x,y^2)$,
there is a surjective multiplication map
$$
[I_n(\bfX+z\bfW)]_n 
\otimes_\kk
\big(
[(\Gamma_1,\Gamma_2)]_{\mu-n}
\oplus
 \Span_\kk( yz^{\mu-n-1}, z^{\mu-n})
 \big)
\rightarrow 
[I_n(\bfX+z\bfW)]_\mu.
 $$
Therefore, 
it suffices to  consider the $2n+2$ elements 
$\Delta_1 yz^{\mu-1}, \ldots, \Delta_{n+1} yz^{\mu-1}, \Delta_1 z^{\mu}, \ldots, \Delta_{n+1} z^{\mu}$.
Multiplying
the reduction equation \eqref{EqDivisionAlgorithmDelta}
by $yz^{\mu-n-1}$ we obtain
\begin{align*}
\Delta_i yz^{\mu-n-1} &= \alpha yz^{\mu-n-1}\Gamma_1  + \beta yz^{\mu-n-1} \Gamma_2 + g_iy^2z^{\mu-2}+G_iyz^{\mu-1}\\
&= \alpha yz^{\mu-n-1}\Gamma_1  + \beta yz^{\mu-n-1} \Gamma_2 + g_i(\Gamma_2- v_3 yz-v_4 z^2)z^{\mu-2}+G_iyz^{\mu-1}\\
&= \alpha yz^{\mu-n-1}\Gamma_1  + (\beta yz^{\mu-n-1} + g_i z^{\mu-2}) \Gamma_2 
+
(-v_3g_i +G_i) yz^{\mu-1} -g_iv_4z^\mu,
\end{align*}
so the entries of these $n+1$ columns of $\Phi$ all lie in $ \mfL$.
Multiplying
the reduction equation \eqref{EqDivisionAlgorithmDelta}
by $z^{\mu-n}$,
we see that the image of $\Delta_i z^{\mu}$ modulo $ [(\Gamma_1,\Gamma_2)]_\mu$
is $g_i yz^{\mu-1}+G_iz^\mu$,
so the entries of these $n+1$ columns of $\Phi$ are the generators of $ \mfL$, and the proof is concluded.
\end{proof}

We will not work with $\mfL$ or determine the polynomials $g_i, G_i$ directly.
Rather, the core of the proof of Theorem \ref{MainTheorem}
 revolves around 
the  image of $\mfL$ in $A$, 
equivalently, 
 the fiber of the natural map $\Hilb^{(m,2)}(\AA^2) \to \Hilb^2(\AA^2)$ over the origin $[\V(x,y^2)]$.

The polynomial ring $A$ is an algebra retract of $B$,
 hence it may be regarded both as a subring and as a quotient;
likewise for $A \otimes_\kk P= \kk[w_{i,j},x,y,z]$ and $T = B \otimes_\kk P$.
In both cases we use   $\overline{\cdot}$ to denote  images under the quotient map, 
that is, setting $v_h =0$ for all $h$.
We equip the polynomial ring $A \otimes_\kk P$ with  the bigrading 
and term order induced from $T$.
Observe that  $\overline{\Delta_i}=\Delta_i$, since $\Delta_i \in A\otimes_\kk P$.
Applying $\overline{\cdot}$ to  the reduction equation \eqref{EqDivisionAlgorithmDelta} gives
\begin{equation}\label{EqDivisionAlgorithmDeltaBar}
\Delta_i = \overline{\alpha_i} x  + \overline{\beta_i} y^2 + f_iyz^{n-1}+F_iz^n,
\end{equation}
where we set 
\begin{equation}\label{EqDefinitionFiGi}
f_i = \overline{g_i}
\quad
\text{ and }
\quad
 F_i = \overline{G_i}
\quad
\text{ for }
i = 1, \ldots, n+1.
\end{equation}
Since $\LM(\alpha_i x)\leq \LM(\Delta_i)$ in \eqref{EqDivisionAlgorithmDelta}, 
 all the terms of $\alpha_i x$ are bounded by  $\LM(\Delta_i)$,
and this implies that $\LM(\overline{\alpha_i} x)\leq \LM(\Delta_i)$,
since the terms of $ \overline{\alpha_i} x$ are among those of $\alpha_i x$.
Likewise, we have $\LM(\overline{\beta_i} y^2)\leq \LM(\Delta_i)$.
Finally, the terms of $\overline{\rho}= f_iyz^{n-1}+F_iz^n$ are not divisible by $x,y^2$, since 
they are among those of $\rho$.
It follows that  \eqref{EqDivisionAlgorithmDeltaBar} is a reduction equation,
that is, it  satisfies the requirements of  the division algorithm for the term order of $A \otimes_\kk P$.
In particular, $\overline{\rho}= f_iyz^{n-1}+F_iz^n$ is the unique remainder of $\Delta_i$ 
modulo  $\{x, y^2\}$.
We define the ideal
\begin{equation}\label{EqDefinitionIdealI}
\mfI =(f_1, \ldots, f_{n+1}, F_1, \ldots, F_{n+1})\subseteq A.
\end{equation}

Clearly,  $\mfI = \overline{\mfL}$, the image of $\mfL$ in $A$, 
and $\V(\mfI)\subseteq \mcW$ is the scheme-theoretic fiber 
of the map
$\V(\mfL) \hookrightarrow
\mcW \times \mcV \to \mcV$
over the origin of $\mcV$. 
The discussion above gives  us more precise information,
and  it allows us to express 
the polynomials $f_i, F_i$   in terms of the generic matrix $\bfW$.

\begin{notation}
Denote by $\bfW^i$  the $n\times n$ submatrix of $\bfW$ obtained by deleting row $i$,
 by  $\bfW^{(i,j),k}$  the $(n-1)\times (n-1)$ submatrix of $\bfW$ obtained by deleting rows $i,j$ and column $k$,
where $i<j$.
For simplicity, we set $\det\bfW^{(i,j),k} = - \det \bfW^{(j,i),k}$ if $i>j$, and declare $\det \bfW^{(i,i),k}=0$.
\end{notation}

\begin{lemma}\label{LemmaDescriptionFi}
For each $i = 1, \ldots, n+1$,
we have $F_i = \det \bfW^i$ and 
$
f_i = \sum_{h=1}^{n}\det\bfW^{(h,i),h}.
$
\end{lemma}
\begin{proof}
Define $\bfX^i$ analogously to $\bfW^i$, then
 $\Delta_i =  \det(\bfX^i + z \bfW^i)$.
From  \eqref{EqDivisionAlgorithmDeltaBar} we see that
 $f_i$ and $F_i$ are the coefficients of $yz^{n-1}$ and $z^{n}$ in $\Delta_i$, 
 respectively.
It follows immediately that $F_i = \det \bfW^i$.

Expand   
$\Delta_i = \det(\bfX^i + z \bfW^i)$ 
as a sum of $n!$ products. 
The terms with $yz^{n-1}$ are obtained by picking an entry $y$ from $\bfX^i$, 
and the remaining $n-1$ entries from $z \bfW^i$.
Thus, for each occurrence of $y$ in $\bfX^i$, 
the contribution to the coefficient of $yz^{n-1}$ in $\Delta_i$ is a signed  determinant of the submatrix of $\bfW^i$ 
obtained by deleting the row and column corresponding to $y$.
We now determine these signs.

The variable $y$ appears in $\bfX^i$ in positions $(h,h)$ with $h<i$, and positions $(h,h+1)$ with $i\leq h < n$. 
For  $h<i$, the variable $y$ in position $(h,h)$ carries a sign of $(-1)^{h+h}=+1$, 
and the corresponding submatrix of $\bfW^i$ is $\bfW^{(h,i),h}$.
Thus, the contribution to $f_i$ is $\det \bfW^{(h,i),h}$.
For  $i\leq h < n$, the variable $y$ in position $(h,h+1)$  in $\bfX^i$ carries a sign of $(-1)^{h+h+1}=-1$, 
and the corresponding submatrix of $\bfW^i$ is $\bfW^{(i,h+1),h+1}$.
Thus, the contribution to $f_i$ is $-\det \bfW^{(i,h+1),h+1}=\det \bfW^{(h+1,i),h+1}$.
Finally, $\det \bfW^{(i,i),i}=0$ by convention.
We conclude that $
f_i = \sum_{h=1}^{n}\det\bfW^{(h,i),h}
$
as desired.
\end{proof}

Lemma \ref{LemmaDescriptionFi} shows that the scheme $\V(\mfI)$ is squeezed between the two generic determinantal varieties
$\V(I_{n-1}(\bfW))$ and $\V(I_n(\bfW))$,
which  have codimension 6 and 2 respectively.
We will see in Corollary \ref{CorollaryXIrreducibleDimension} that $\V(\mfI)$ is in fact irreducible of codimension 4.
It is also possible to show that the generating set of $\mfI$ is not minimal,
and
that a minimal system of generators is $\mfI  = (f_1, \ldots, f_{n+1}, F_{n+1})$,
but we will not need this fact.
Finally, we will prove in Theorem \ref{TheoremIPrime} that $\mfI$ is a prime ideal;
 however, at the moment it is not even clear whether $\mfI$ is  radical.

The next section is devoted to the variety $\mfX= \V(\sqrt{\mfI}) \subseteq \Spec (A)=  \Mat(n+1,n)$,  cut out set-theoretically by $\mfI$.
In preparation, we record here the following simple fact.
Denote
\begin{equation}\label{EqMatrixY}
\bfY=
\bfX_{\,x=0} = 
\small{
\begin{bmatrix}
y & 0 &   \cdots & 0\\
0 & y &   \cdots  & 0\\
0 & 0 &     \cdots & 0 \\
\cdots & \cdots  & \cdots  & \cdots \\
0 & 0 &  \cdots  & y \\
0 & 0 &  \cdots  & 0 \\
\end{bmatrix}.}
\end{equation}

\begin{cor}\label{CorollaryDescriptionSetTheoreticLocusI}
We have $
\mfX = \V(\sqrt{\mfI}) =
 \big\{ \bfW \in \mathrm{Mat}(n+1,n) \, : \,
I_n(\bfY + \bfW) \subseteq (y^2)  \subseteq \kk[y]\big\}.
$
\end{cor}
\begin{proof}
For every $\bfB \in \mathrm{Mat}(n+1,n)$, we have the chain of equivalences
\begin{align*}
\bfB \in V\big(\sqrt{\mfI}\big) & \Leftrightarrow I_n(\bfX + z\bfB) \subseteq (x,y^2) \subseteq P= \kk[x,y,z]
&& \text{by } \eqref{EqDivisionAlgorithmDeltaBar}
\\
& \Leftrightarrow I_n(\bfX + \bfB) \subseteq (x,y^2) \subseteq R=\kk[x,y]
&& \text{by  \cite[Lemma 6.6]{Constantinescu}}
\\
& \Leftrightarrow 
I_n(\bfX + \bfB)+(x) \subseteq (x,y^2)  \subseteq R
&& \text{since } (x) \subseteq (x,y^2)
\\
& \Leftrightarrow 
I_n(\bfY +\bfB) \subseteq (y^2)  \subseteq \kk[y]
&& \text{going modulo } (x).
\qedhere
\end{align*} 
\end{proof}

\section{A variety of matrices}\label{SectionGeometricTechnique}

In this section, 
we study    the affine variety $\mfX = \V(\sqrt{\mfI}) \subseteq \Mat(n+1,n)$,
cut out set-theoretically by the ideal $\mfI \subseteq A$ introduced in  \eqref{EqDefinitionIdealI}.
The variety $\mfX$ is the reduced scheme of a neighborhood of 
the compressed pair $C_n=[\V(\mm^n)\supseteq \V(x,y^2)]$ in the fiber of the map $\Hilb^{(m,2)}(\AA^2) \rightarrow \Hilb^2(\AA^2)$ over $[\V(x,y^2)]$.
The goal  is to prove the following theorem.

\begin{thm} \label{TheoremVarietyMatrices} 
Assume $\ch (\Bbbk) =0$ and $n \geq 4$. 
The  variety $\mfX = \V(\sqrt{\mfI}) \subseteq \Mat(n+1,n)$ has  rational singularities.
It is a cone over a projective subvariety of $ \P^{n^2+n-1}$  of degree 
$
\frac{1}{12}(n-1)n(n+1)(3n-2).
$
\end{thm}

The variety   admits an action of $\GL_n$, and thus
 it can be treated by means of representation  theoretic techniques.
We follow the point of view of  the theory of rank varieties, 
as in  \cite{ES,Weyman} and
 \cite[Chapter 8]{WeymanBook}.
Our variety $\mfX$ is not among those studied in the classical theory,
since it parametrizes non-square matrices. 
Nevertheless, this connection allows us 
 to  employ the  Kempf-Lascoux-Weyman technique  for calculating syzygies to prove Theorem \ref{TheoremVarietyMatrices}.

\begin{notation}
We  fix the following notation and assumptions throughout  the section.
Let $\Bbbk$ be an algebraically closed field with $\ch (\Bbbk) =0$.
We fix an integer  $n \in \N$ with $n \geq 4$,
and the vector space  $E = \kk^n$ of column vectors.
The affine space $\Mat(n+1,n) $ is identified with 
$\Mat(n,n) \times \Mat(1,n)$ by vertical concatenation of matrices, 
and therefore with
$(E^\vee\otimes E) \oplus E^\vee$.

 Let  $\FF=\mathrm{Flag}(1,2;E)$ denote the flag variety parametrizing 
flags of subspaces $W_1\subseteq W_2 \subseteq E $ with $\dim W_1 = 1, \dim W_2 =2$.
 Consider the trivial vector bundle $\mcE = E \otimes \mcO_\FF$ on $\FF$.
 For each $i =1,2$, we have the tautological sub-bundle
$\mcR_i \subseteq \mcE$ of rank $i$,
parametrizing subspaces $W_i\subseteq E$ with $\dim W_i = i$,
and the tautological quotient bundle $\mcQ_i = \mcE /\mcR_i$ of rank $n-i$,
parametrizing quotients of $E$ of dimension $n-i$.
\end{notation}

Note that we are assuming $n\geq 4$ for technical reasons, 
that is, to simplify the  analysis in some technical results in this section.
In any case,
 this assumption has no impact on our objectives;
in fact,   we may assume $n \gg 0$ by Theorem \ref{TheoremReductionToCompressedPair}.

\subsection{Resolution of singularities}

Our first goal in this section  is the determination of a suitable resolution of singularities of $\mfX$.
In order to construct it, we need some linear algebra facts.

\begin{lemma}\label{LemmaDimensionKernels}
Let $\bfA \in \Mat(n,n)$ and $\bfa \in \Mat(1,n)$.
Then 
$$
\begin{cases}
\dim \big( \ker \bfA \cap \ker \bfa \big) \geq 1\\
\dim \big( \ker \bfA^2 \cap \ker \bfa \bfA \cap \ker \bfa \big) \geq 2
\end{cases}
\hspace{-0.3cm}
\Leftrightarrow
\exists \,
(W_1, W_2) \in \FF
: 
\bfA W_2 \subseteq W_1,\, \bfA W_1 =  \bfa W_2 = 0.
$$
\end{lemma}

\begin{proof}
The backward direction is obvious: 
if 
$(W_1, W_2) \in \mathrm{Flag}(1,2;E)$
is a flag of subspaces such that
$\bfA W_2 \subseteq W_1,\, \bfA W_1 =  \bfa W_2 = 0$,
then $W_1 \subseteq \ker \bfA \cap \ker \bfa $ and $W_2 \subseteq  \ker \bfA^2 \cap \ker \bfa \bfA \cap \ker \bfa$.
Conversely, assume that
$\dim ( \ker \bfA \cap \ker \bfa ) \geq 1$ and $\dim ( \ker \bfA^2 \cap \ker \bfa \bfA \cap \ker \bfa ) \geq 2$.
Observe that $\ker \bfA \subseteq \ker \bfA^2 \cap \ker \bfa \bfA$,
therefore
\begin{equation}\label{EqInclusionKernels}
\ker \bfA \cap \ker \bfa \subseteq \ker \bfA^2 \cap \ker \bfa \bfA \cap \ker \bfa.
\end{equation}
 Suppose \eqref{EqInclusionKernels} is an equality.
 Then there exist linearly independent vectors 
 $\bfw_1, \bfw_2 \in \ker \bfA \cap \ker \bfa$,
 and the flag $W_1 = \Span_\kk( \bfw_1 ), W_2 = \Span_\kk( \bfw_1, \bfw_2 )$ satisfies the desired conditions.
Suppose the inclusion  \eqref{EqInclusionKernels} is strict,
then there exists  $\bfw_2 \in \ker \bfA^2 \cap \ker \bfa \bfA \cap \ker \bfa$ such that $\bfw_1 = \bfA \bfw_2  \ne \mathbf{0}$.
We have $\bfA \bfw_1 = \bfA^2 \bfw_2 = \mathbf{0}$, 
and  $\bfw_1, \bfw_2$ are linearly independent, since $\bfw_2 \notin \ker \bfA$.
It follows  that the flag $W_1 = \Span_\kk( \bfw_1 ), W_2 = \Span_\kk( \bfw_1, \bfw_2 )$ satisfies the desired conditions.
\end{proof}

\begin{lemma}\label{LemmaMinorsNilpotentBlock}
Let $\ell \in \N$.
Consider the following $(\ell+1)\times 	\ell$ matrix
$$
\bfM=
\begin{bmatrix}
y & 1 &  0 &  \cdots & 0\\
0 & y &  1 & \cdots  & 0\\
0 &0 & y &   \cdots  & 0\\
\cdots & \cdots & \cdots  & \cdots  & \cdots \\
0 & 0 & 0 &  \cdots  & 1 \\
0 & 0 & 0 & \cdots  & y \\
a_1 & a_2 & a_3 & \cdots  & a_\ell \\
\end{bmatrix}.
$$
Let $\delta_i$ denote the $\ell\times \ell$ minor of $\bfM$ obtained by eliminating row $i$.
For each $i =1, \ldots, \ell$,
 we have
$$
\delta_i = (-y)^{\ell-i}\big(a_iy^{i-1}-a_{i-1}y^{i-2}+ \cdots + (-1)^{i+1}a_1\big).
$$
\end{lemma}
\begin{proof}
The calculation is straightforward.
\end{proof}

Recall the definition of the matrix $\bfY=\bfX_{x=0}$ in \eqref{EqMatrixY}.

\begin{lemma}\label{LemmaJordanCanonicalForm}
Let $\bfA \in \Mat(n,n)$ be a matrix in Jordan canonical form. 
For every $\bfa\in \Mat(1,n)$, we have

\begin{equation*}
I_n\left(\bfY + \begin{bmatrix}
\bfA\\
\bfa
\end{bmatrix}\right) \subseteq \big(y^2\big)
\,
\Leftrightarrow 
\,
\dim \big( \ker \bfA \cap \ker \bfa \big) \geq 1
\,\, \text{and} \,\,
\dim \big( \ker \bfA^2 \cap \ker \bfa \bfA \cap \ker \bfa \big) \geq 2.
\end{equation*}
\end{lemma}

\begin{proof}
Define the two subsets of $\Mat(1,n)$ from the statement of the lemma:
\begin{eqnarray}
\label{EqSubspaceL}
\quad\mL &=& \Big\{ 
\bfa \in \Mat(1,n) \,:\, I_n\left(\bfY + \begin{bmatrix}
\bfA\\
\bfa
\end{bmatrix}\right) \subseteq \big(y^2\big)
\Big\},
\\
\label{EqSubspaceR}
\quad\mR &=& \Big\{ 
\bfa \in \Mat(1,n) \,:\, 
\dim \big( \ker \bfA \cap \ker \bfa \big) \geq 1,
\,
\dim \big( \ker \bfA^2 \cap \ker \bfa \bfA \cap \ker \bfa \big) \geq 2
\Big\}.
\end{eqnarray}

Let $\bfA_1$ denote the invertible submatrix of $\bfA$ consisting of the Jordan blocks  associated to non-zero eigenvalues,
and  $\bfN_1, \ldots, \bfN_p$  the nilpotent Jordan blocks, i.e. the ones associated to the the eigenvalue 0. 
Let $\ell_i$ be the size of the block $\bfN_i$ and  $s= n -\sum \ell_i$ denote the size of $\bfA_1$.
Note that $\rank \, \bfA = n-p$ and $\rank \, \bfA^2 = n-p - \#\{i \, : \, \ell_i \geq 2\}$.
We denote by $\bfI_r \in \Mat(r,r)$ the identity matrix, and by $\bfe_i$  the $i$-th standard basis vector of $E$.
Given $\bfa = [a_1, \ldots, a_n] \in \Mat(1,n)$, denote by $\bfB$ the vertical concatenation of $\bfA$ and $\bfa$.
Up to reordering the blocks, we have
$$
\bfA = 
\left[\begin{array}{cccc}
\bfA_1 &  &  &\\
 & \bfN_1 & &\\
 & & \ddots & \\
 & & & \bfN_p
\end{array}
\right]
\qquad
\text{and}
\qquad
\bfY + \bfB =
\left[
 \begin{array}{cccc}
\bfA_1 + y\bfI_s&  & &\\
 & \bfN_1 +y\bfI_{\ell_1}& &\\
 & & \ddots & \\
 & & & \bfN_p+y\bfI_{\ell_p}\\
 \hline
 &\quad \bfa& & 
\end{array} 
\right].
$$
Finally, let $\Theta_i$ denote the $n \times n$ minor of $\bfY + \bfB$ obtained by deleting row $i$.
We will prove that $\mL = \mR$ distinguishing several cases, based on the shape of the Jordan decomposition of $\bfA$.

\underline{Case  $p=0$}.
We have $\Theta_{n+1} = \det (\bfA_1+y\bfI_n) \equiv \det \bfA_1 \ne 0 \pmod{y}$.
Thus, $I_n(\bfY+\bfB)\not\subseteq (y)$ and $\mL = \emptyset$.
Moreover,  $\bfA=\bfA_1$ is invertible, so $\dim \ker \bfA= 0$  and $\mR = \emptyset$.

\underline{Case $p=\ell_1=1$}.  
We have $\Theta_{n+1} = y\det (\bfA_1+y\bfI_{n-1}) \equiv y\det \bfA_1 \ne 0 \pmod{y^2}$.
Thus, $I_n(\bfY+\bfB)\not\subseteq (y^2)$ and $\mL = \emptyset$.
Moreover, $\rank \, \bfA^2 = n-1$, so  $\dim \ker \bfA^2 = 1 $ and $\mR = \emptyset$.

\underline{Case $p=1, \ell_1\geq 2$}.
The unique nilpotent block $\bfN_1$ extends from column $s+1$ to column $n$.
Using Laplace expansion and Lemma \ref{LemmaMinorsNilpotentBlock},
we see that $y^2|\Theta_i $ for all $i \ne n, n-1$,
whereas 
\begin{align*}
\Theta_n &= \det(\bfA_1 + y \bfI_s)( \cdots \pm a_{s+2}y \pm a_{s+1}   ) \equiv (\det\bfA_1 + y a ) (\pm a_{s+2}y\pm a_{s+1} ) &\pmod{y^2},
\\
\intertext{and}
\Theta_{n-1} &= \det(\bfA_1 + y \bfI_s)(\cdots \pm a_{s+2}y^2 \pm a_{s+1}y) \equiv \pm (\det\bfA_1) a_{s+1}
 y &\pmod{y^2},
\end{align*}
where $a$ is the coefficient of $y$ in $\det(\bfA_1 + y \bfI_s)$.
Since $\det \bfA_1 \ne 0$,
it follows that $\bfa\in \mL$ if and only if $a_{s+1} = a_{s+2}=0$.
We have $\rank \, \bfA = n-1$ and $\rank \, \bfA^2 = n-2$.
From the Jordan decomposition we see that 
 $ \ker \bfA =\Span_\kk( \bfe_{s+1} )$ and $ \ker \bfA^2 = \Span_\kk( \bfe_{s+1}, \bfe_{s+2} )$.
It follows that $\bfa\in \mR $  if and only if  $a_{s+1} = a_{s+2}=0$.

\underline{Case $p=2, \ell_1=\ell_2 = 1$}.
By Laplace expansion, we see  that $y^2 | \Theta_i$ if $i \ne n,n-1$,
whereas 
$$
\Theta_n = \pm a_{n}y \det(\bfA_1 + y \bfI_s)\quad \text{and} \quad \Theta_n = \pm a_{n-1}y \det(\bfA_1 + y \bfI_s).
$$
As in the previous case, we conclude 
that $\bfa\in \mL$  if and only if $a_{n-1} = a_{n}=0$.
We have $\rank \, \bfA = \rank \, \bfA^2 = n-2$.
Since $\ker \bfA = \ker \bfA^2  = \Span_\kk( \bfe_{n-1}, \bfe_{n} )$,
we have  $\bfa \in \mR$ if and only if 
$a_{n-1} = a_{n}=0$.

\underline{Case $p=2, \ell_1=1,\ell_2 \geq 2$}.
The nilpotent block $\bfN_1$ occupies column $s+1$,
while $\bfN_2$ extends from column $s+2$ to $n$.
As in the previous cases,
using Laplace expansion and Lemma \ref{LemmaMinorsNilpotentBlock}
we see that $y^3|\Theta_i$ for all $i \ne s+1, n-1, n$.
Moreover,  $y^2|\Theta_{s+1}$ since $\Theta_{s+1}$ contains the block
$\bfN_2+y\bfI_{\ell_2}$ and $\ell_2 \geq 2$.
Applying Lemma \ref{LemmaMinorsNilpotentBlock} to the block $\bfN_2+y\bfI_{\ell_2}$,
since $\det(\bfN_1+y\bfI_{\ell_1}) = y$,
we see that 

\begin{align*}
\Theta_n &= \det(\bfA_1 + y \bfI_s)y (\cdots \pm y a_{s+3} \pm a_{s+2}) \equiv \pm \det(\bfA_1)  y a_{s+2}  &\pmod{y^2}
\\
\intertext{and}
\Theta_{n-1} &= \det(\bfA_1 + y \bfI_s)y(\cdots \pm y^2 a_{s+3} \pm ya_{s+2}) \equiv 0 &\pmod{y^2}.
\end{align*}
We conclude that $\bfa \in \mL$ if and only if $a_{s+2}=0$.

Considering the Jordan decomposition, we find the kernels $\ker \bfA = \Span_\kk( \bfe_{s+1}, \bfe_{s+2})$ and 
$\ker \bfA^2 = \Span_\kk( \bfe_{s+1},\bfe_{s+2}, \bfe_{s+3})$.
Assume  $a_{s+2}=0$, so    $\bfe_{s+2}\in  \ker \bfa$.
Since $\dim \ker \bfa \geq n-1$,
there exists   $\mathbf{0}\ne\bfv\in \ker \bfa \cap \Span_\kk( \bfe_{s+1}, \bfe_{s+3})$.
By the Jordan decomposition, we have 
  $ \bfA \big(\Span_\kk(\bfe_{s+1}, \bfe_{s+3})\big) = 
\Span_\kk(\bfe_{s+2})$,
thus $ \bfv \in \ker \bfa \bfA$.
Since $\Span_\kk( \bfe_{s+2}) \subseteq \ker \bfA \cap  \ker \bfa$ and
 $\Span_\kk(\bfe_{s+2}, \bfv) \subseteq \ker \bfA^2 \cap \ker \bfa \bfA \cap  \ker \bfa$,
we conclude that $\bfa \in \mR$.
Now, assume  $a_{s+2}\ne 0$
and let $\bfw \in \ker \bfA^2 \cap \ker \bfa \bfA \cap \ker \bfa$.
Write  $\bfw= \alpha \bfe_{s+1} + \beta \bfe_{s+2}+ \gamma \bfe_{s+3}$, where $\alpha, \beta, \gamma \in \kk$.
Since $\bfA \bfw = \gamma \bfe_{s+2}$ and $ \bfw  \in \ker \bfa \bfA$,
we must have $\gamma =0$.
Since $ \bfw  \in \ker \bfa$, we must have $\alpha a_{s+1} + \beta a_{s+2}=0$,
thus $\beta = -a_{s+2}^{-1} a_{s+1}\alpha$.
In conclusion, we find that 
$\ker \bfA^2 \cap \ker \bfa \bfA \cap \ker \bfa
= \Span_\kk( \bfe_{s+1}  -a_{s+2}^{-1} a_{s+1} \bfe_{s+2})$,
in particular it has dimension 1 and thus $\bfa \notin \mR$ when $a_{s+2}\ne 0$.

\underline{Case $p = 2$ and $\ell_1, \ell_2 \geq 2$}.
Each $\Theta_i$ contains an entire block $\bfN_j+y\bfI_{\ell_j}$ with $\ell_j \geq 2$.
This implies that $y^{\ell_j} = \det(\bfN_j+y\bfI_{\ell_j}) $ divides $ \Theta_i$,
so $I_n(\bfY+\bfB)\subseteq (y^2)$ and $\mL= \Mat(1,n)$.
We  have $\rank \, \bfA = n-2$ and $ \rank \, \bfA^2 = n-4$, 
so $\dim \ker \bfA = 2$ and $\dim \ker \bfA^2 = 4$.
Since $\dim \ker \bfa \geq n-1$ and $\dim \ker \bfa \bfA \geq n-1$
for every $\bfa$, we conclude that $\mR = \Mat(1,n)$. 

\underline{Case  $p\geq 3$}.
The matrix $\bfY + \bfB$ has $p$ rows where $y$ is the only non-zero entry.
Each $\Theta_i$  contains at least  two such rows,
so $y^2| \Theta_i$ and $I_n(\bfY+\bfB)\subseteq (y^2)$,
therefore $\mL = \Mat(1,n)$.
Moreover, since $\dim \ker \bfA = p \geq 3$ and $\ker \bfA \subseteq \ker \bfA^2 \cap \ker \bfa \bfA$,
it follows easily that  $\mR = \Mat(1,n)$.

We have shown that $\mL = \mR$ for every $\bfA$, so the proof is concluded.
\end{proof}

Following Lemma \ref{LemmaDimensionKernels},
we define
$$
\wtX = 
\big\{ 
(\bfA,\bfa, W_1, W_2) \, | \,\bfA W_2  \subseteq W_1,\, \bfA W_1 = 0, \,  \bfa W_2=0 
\big\} 
\subseteq \Mat(n,n) \times \Mat(1,n) \times \FF.
$$
The variety $\wtX$ is the total space of a vector bundle $\mcS$ on $\FF$.
To see this, we  also consider the variety
$
\widetilde{\mathfrak{Y}} = 
\big\{ 
(\bfA, W_1, W_2) \, | \, \bfA W_2  \subseteq W_1,\, \bfA W_1 = 0
\big\} 
\subseteq \Mat(n,n)  \times \FF,
$
which appears in \cite[Section 3]{Weyman} 
as desingularization of the rank variety
$\mathfrak{Y}=\{\bfA \in \Mat(n,n) \, : \, \rank \, \bfA^2 \leq n-2\}$;
in the notation of that paper, they are respectively the varieties $Y_\bfv$  and $X_\bfv$ in the case of the partition $\bfv=(1,1)$.
It is known that
$\widetilde{\mathfrak{Y}} $ is the total space of a vector bundle 
$\mcT\subseteq \mcE^\vee \otimes \mcE$
of $\FF$, which has  rank $n^2-2n+1$;
further properties of the bundle $\mcT$ will be given at the beginning of Subsection \ref{Subsection KLW technique}.
 It follows that  $\wtX$ is  the total space of the following
 decomposable  vector bundle on $\FF$
 \begin{equation}\label{EqVectorBundleS}
  \mcS = \mcT \oplus \mcQ_{n-2}^\vee \subseteq (\mcE^\vee \otimes \mcE)\oplus \mcE^\vee.
  \end{equation}
In particular,  $\wtX$ is a smooth irreducible variety of dimension 
$$
\dim \wtX = \rank\, \mcS + \dim \FF = (n^2-n-1)+ (2n-3)= n^2+n-4.
$$

\begin{thm}\label{TheoremResolutionSingularities}
Let $\rho : \Mat(n+1,n)  \times \FF \rightarrow  \Mat(n+1,n)$ denote the projection map.
The restricted map  $\rho : \wtX \rightarrow \mfX$ is a resolution of singularities.
\end{thm}

\begin{proof}
The identification $\mathrm{Mat}(n+1,n) =  \Mat(n,n) \times \Mat(1,n) = (E^\vee\otimes E) \oplus E^\vee$
gives the  action of $\GL_n $ on $\mathrm{Mat}(n+1,n) $
by 
$$
\varphi \cdot \bfB = (\varphi \oplus 1) \bfB \,\varphi^{-1}
\quad \text{ where}
\quad
\varphi\oplus 1 = 
\begin{bmatrix}
\varphi & 0  \\
0 & 1
\end{bmatrix}
\in \mathrm{GL}_{n+1}.
$$
We note that both $\mfX$ and $\wtX$ are invariant subvarieties under the action of $\GL_n $.
This is obvious for  $\wtX \subseteq \mathrm{Mat}(n+1,n) \times \FF$.
For $\mfX$, recall the description of Corollary \ref{CorollaryDescriptionSetTheoreticLocusI}. 
We have $\varphi \cdot \bfY = \bfY$,
and  determinantal ideals are invariant by this action,
therefore
$$
I_n (\bfY + \bfB) =I_n(\varphi \cdot (\bfY + \bfB)) = I_n(\varphi \cdot \bfY + \varphi \cdot \bfB)= I_n( \bfY + \varphi \cdot \bfB).
$$
This implies  that $\bfB \in \mfX$ if and only if $\varphi \cdot \bfB \in \mfX$, as claimed.

First, we show that $\rho(\wtX) = \mfX$, equivalently,
that for all   $\bfA \in \Mat(n,n),  \bfa \in \Mat(1,n)$ we have

\begin{equation*}
I_n\left(\bfY + \begin{bmatrix}
\bfA\\
\bfa
\end{bmatrix}\right) \subseteq (y^2)
\,
\Leftrightarrow 
\,
\exists \,\,
(W_1, W_2) \in \mathrm{Flag}(1,2;E)
\,
: \,
\bfA (W_2) \subseteq W_1, \bfA W_1 = 0, \bfa W_2 = 0.
\end{equation*}
Exploiting the action of $\GL_n$, 
it suffices to show this equivalence when $\bfA$ is in Jordan canonical form,
and this was done in Lemmas \ref{LemmaDimensionKernels} and \ref{LemmaJordanCanonicalForm}.
Next,
observe that the locus $\mathfrak{U} \subseteq \wtX$ where $\rank\, \bfA = n-1, \rank \, \bfA^2 = n-2$ is non-empty (for instance by Lemma \ref{LemmaJordanCanonicalForm}, case $p=1, \ell_1 \geq 2$) and  open.
It is clear that the restriction of  $\rho $ to  $\mathfrak{U}$ is an injective morphism, 
since the flag $(W_1, W_2)$ is uniquely determined.
In conclusion, the morphism  $\rho : \wtX \rightarrow \mfX$ is birational since it is dominant and generically injective, and  clearly, it is also proper, 
hence a resolution of singularities.
\end{proof}

\begin{cor}\label{CorollaryXIrreducibleDimension}
The variety  $\mfX = \V(\sqrt{\mfI})$ is irreducible of dimension $n^2+n-4$.
\end{cor}

We record here another byproduct, which  will be useful in Section \ref{SectionProofMainTheorem}.

\begin{cor}\label{CorollaryDimensionFibers}
Let $m \in \N$.
Every closed fiber of the  natural morphism $\Hilb^{(m,2)}(\AA^2) \rightarrow \Hilb^2(\AA^2)$ is irreducible of dimension $2m-4$.
\end{cor}
\begin{proof}
 The (set-theoretic) fiber 
over a point  $[Z_2]\in \Hilb^2(\AA^2)$
is $\big\{ [Z_1] \in \Hilb^m(\AA^2) \, : \, Z_1 \supseteq Z_2\big\}$.
It suffices to show that this fiber is the closure of its subset 
$$
\mathcal{L}_m = 
\big\{ [Z_2 \cup Y]  \in \Hilb^m(\AA^2) \, : \, Y \text{ is reduced and disjoint from the support of }Z_2\big\},
$$
since the latter is isomorphic to an open locus in $\Hilb^{m-2}(\AA^2)$
and, hence, it is irreducible of dimension $2m-4$.
This is equivalent to the claim that $Z_1$ is a limit of subschemes of the form $Z_2 \cup Y$, with $Y$ reduced and disjoint from $Z_2$.

By changing  coordinates, every  $Z_2$ is either reduced or of the form $\V(x,y^2)$.
Assume the former.
We can cleave $Z_1$ repeatedly using the deformation of the second paragraph of the proof of Proposition \ref{PropositionCleaving}  for each reduced point  of $Z_2$.
 In this way, we obtain $Z_1$ as a limit of reduced subschemes containing $Z_2$, and the claim holds.
Now, assume   $Z_2 = \V(x,y^2)$.
In order to verify the claim, we may harmlessly replace $Z_1$ with $Z_1 \cup W$, where $W$ is reduced and disjoint from $Z_1$.
Thus, we may assume 
 that $m = {n+1 \choose 2}$ for some $n$. 
By  Corollary \ref{CorollaryXIrreducibleDimension},
the  fiber over $[\V(x,y^2)]$   is irreducible of dimension $2m-4$,
so it must be equal to the closure of  $\mathcal{L}_m$,
and the claim holds.
\end{proof}

\subsection{Kempf-Lascoux-Weyman technique}\label{Subsection KLW technique}

Having established  the resolution of singularities of Theorem \ref{TheoremResolutionSingularities},
we are now in a position to apply the Kempf-Lascoux-Weyman technique 
to prove that $\mfX$ has rational singularities and compute its degree.
The rest of this section  follows the treatment of rank varieties developed in \cite{Weyman} and  \cite[Chapter 8]{WeymanBook}.

The Kempf-Lascoux-Weyman technique \cite[Section 5.1]{WeymanBook} 
is based on the calculation of the cohomology groups of  exterior powers of certain ``syzygy'' bundles
related to   a  desingularization.
Specifically, recall the desingularization 
$\rho:\wtX \to \mfX$ of Theorem \ref{TheoremResolutionSingularities},
and the vector bundles $\mcS$ and $\mcT$ introduced in \eqref{EqVectorBundleS}.
The structure sheaf of $\wtX$ is the symmetric algebra 
$\Sym(\mcS^\vee)$.
By \cite[(3.2)]{Weyman},
there is an exact sequence 
$
 0 \rightarrow \eta \rightarrow \mcE \otimes \mcE^\vee \rightarrow (\mcT)^\vee \rightarrow 0.
$
In turn, by \eqref{EqVectorBundleS}, 
this gives rise to an exact sequence
$$
 0 \rightarrow \eta\oplus \mcR_2 \rightarrow (\mcE \otimes \mcE^\vee) \oplus \mcE \rightarrow \mcS^\vee \rightarrow 0.
$$
The vector bundle $\xi = \eta\oplus \mcR_2$ is called a syzygy bundle of $\mcS^\vee$.
Our goal is to study the cohomology groups $H^q \left(\FF, \, \Wed^p \xi\right)$.
In particular, 
define the free graded $A$-modules
$$
\mathsf{F}_i  = \bigoplus_{j \geq 0} H^j \left(\FF, \, \Wed^{i+j} \xi\right) \otimes A(-i-j).
$$
By  \cite[Theorem 5.1.2, Theorem 5.1.3 (c)]{Weyman}, in order to prove that  $\mfX$ has rational singularities,
it suffices to show that  $\mathsf{F}_i=0$ for all  $i < 0$ and $\mathsf{F}_0 = A$. 
Moreover, 
by   \cite[Theorem 5.1.6 (b)]{WeymanBook},
the degree of $\mfX$ can be computed once we know the numbers $h^q \left(\FF, \, \Wed^p \xi\right)$.

However, computing the groups $H^q \left(\FF, \, \Wed^p \xi\right)$ explicitly is challenging, 
because  $\xi$
is not a direct sum of of tautological bundles on the flag variety.
For this reason, we proceed as in \cite[Section 3]{Weyman} or \cite[Section 3.1]{Lorincz}, 
and replace $\xi$ by an associated graded bundle $\xi'$.
By \cite[(3.3')]{Weyman}, there is a two-step filtration of the bundle $\eta$ with associated graded bundle 
\begin{equation}\label{EqAssociatedGradedBundleEta}
\eta' =\mathrm{gr}(\eta) = \mcE^\vee \otimes \mcR_1\,\, \oplus\,\,  \mcQ_{n-1}^\vee\otimes  \mcR_2/\mcR_1.
\end{equation}
In turn, this  induces a two-step filtration of  $\xi$  with associated graded bundle
\begin{equation}\label{EqAssociatedGradedBundleXi}
\xi' =  \mathrm{gr}(\xi) =\eta' \oplus \mcR_2 = \mcE^\vee \otimes \mcR_1\,\, \oplus\,\,  \mcQ_{n-1}^\vee\otimes  \mcR_2/\mcR_1\,\,  \oplus\,\, \mcR_2.
\end{equation}
The filtration induces  spectral sequences
$
H^q(\FF, \Wed^p \xi') \Rightarrow H^q(\FF, \Wed^p \xi).
$
Thus, if we prove that $H^q \left(\FF, \, \Wed^p \xi'\right)=0$ whenever $p<q$
and  $\bigoplus_{p\geq 0} H^p \left(\FF, \, \Wed^p \xi'\right)=\Bbbk$,
we obtain that $\mfX$ has rational singularities.
Likewise, we can compute the degree of $\mfX$ using  the numbers $h^q \left(\FF, \, \Wed^p \xi'\right)$
instead of $h^q \left(\FF, \, \Wed^p \xi\right)$.
The main advantage of this approach  is that, by \eqref{EqAssociatedGradedBundleXi}, 
$\xi'$ admits a decomposition in terms of the tautological bundles on $\FF$,
and, hence, the same is  true for its exterior powers.
In fact,
taking exterior powers 
of \eqref{EqAssociatedGradedBundleEta} and \eqref{EqAssociatedGradedBundleXi},
and using the fact that $\mcR_1$ and $\mcR_2/\mcR_1$ are line bundles, 
whereas $\mathcal{R}_2$ has rank $2$, we obtain the decompositions
\begin{align} \label{EqExteriorPowerEtaPrime}
\Wed^{r} \eta' & = \bigoplus_{i+j = r} \Wed^i(\mcE^\vee\otimes \mcR_1) \otimes \Wed^j (\mcQ_{n-1}^{\vee}\otimes \mcR_2/\mcR_1) 
	\\ \notag
&= \bigoplus_{i+j = r} \Wed^i \mcE^{\vee}\otimes \text{Sym}^i\mcR_1 \otimes \Wed^j \mcQ_{n-1}^{\vee}\otimes \text{Sym}^j\left(\mcR_2/\mcR_1\right)
\end{align}
and
\begin{align}
\label{EqExteriorPowerXiPrime} 
\Wed^p \xi' & =  \Wed^{p} \eta' \,\, \oplus \,\, \Wed^{p-1} \eta' \otimes \mathcal{R}_2 
\,\, \oplus \,\, \Wed^{p-2} \eta' \otimes \Wed^2 \mathcal{R}_2.
\end{align}
These decompositions allow us to use representation theoretic techniques, such as Bott's theorem, to compute  $H^q \left(\FF, \, \Wed^p \xi'\right)$.
The bulk of this subsection is devoted to performing these calculations, and, at the end, we combine all of this information to prove Theorem \ref{TheoremVarietyMatrices}.

\begin{notation}
We consider the representation theory of the general linear group $\GL_n = \GL(E)$,
see  \cite{FultonHarris,WeymanBook} for background.
An $n$-tuple $\lambda = (\lambda_1, \lambda_2, \ldots, \lambda_n) \in \mathbb{Z}^n$
is called a weight, and, if $\lambda_i \geq \lambda_{i+1}$ for all $i$, then 
$\lambda$ is called a dominant weight.
We often drop $0$s at the end of a dominant weight, 
for example, we  write $(d)$ for $(d, 0, \ldots, 0)$.
We use the exponential notation to abbreviate repeated entries, 
for example, 
we write $(0^{n-1}, -1)$ for $(0, \ldots, 0, -1)$. 
We denote the Schur functor associated to a dominant weight $\lambda $ by $\bS_\lambda$.
Schur functors establish a bijection between irreducible representations of $\GL_n$ and the set of dominant weights.
We follow the convention of \cite{FultonHarris},
so that 
$\bS_{d} = \Sym^d$ is the symmetric power functor
and
$\bS_{1^r} = \Wed^r$ is the exterior power functor.
Since $\bS_\lambda(E^\vee) \cong (\bS_\lambda E)^\vee$, we will typically just write $\bS_\lambda E^\vee$.
We will often use the facts that 
$
\bS_\lambda E \otimes \bS_{1^n} E = \bS_{\lambda + (1^{n})} E$
and
$\bS_\lambda E^\vee = \bS_{-\lambda_n, \ldots, - \lambda_1} E.
$
Finally, we remark that  these constructions are functorial, and can be performed on vector bundles on a projective variety.
\end{notation}

We  now explain the general strategy in computing the groups  $H^q \left(\FF, \, \Wed^p \xi'\right)$. 
Using   \eqref{EqExteriorPowerEtaPrime},  \eqref{EqExteriorPowerXiPrime}, and the exact sequence
$ 0 \to \mcR_1 \to \mcR_2 \to \mcR_2/\mcR_1\to 0$,
we reduce to computing  $H^q(\FF, \omega)$,
where $\omega = \Wed^a \eta' \otimes L$
for certain line bundles $L$ and integers $a$.
We exploit the Leray spectral sequence   
$E_2^{i,j}=H^i(\Gr(1,E), R^j \pi_\star \omega) \Rightarrow H^{i+j}(\FF,\omega)$
induced by the natural projection $\pi: \FF \to \Gr(1,E)$. 
First, 
we compute the higher direct images $ R^j \pi_\star \omega$.
Since $\pi$  identifies  $\FF$ with  the relative Grassmannian  $\Gr(1, \mcQ_{n-1})$,
where $\mcQ_{n-1}$ now denotes the tautological quotient bundle of $\Gr(1,E)$,
we can compute $ R^j \pi_\star \omega$ using Bott's theorem for $\Gr(1, \mcQ_{n-1})$. 
This computation will also show that the Leray spectral sequence abuts at the page $E_2$.
Thus, the calculation of the groups $H^q(\FF, \omega)$
reduces to that of the groups $H^i(\Gr(1,E), R^j \pi_\star \omega)$,
which, again,  can be performed using  Bott's theorem for $\Gr(1,E)$.

\begin{lemma} \label{LemmaR2R1}
We have
$
H^q \big(\FF,\,\Wed^{p-1} \eta'\otimes \mcR_2/\mcR_1\big) = E$
for 
$$(p,q) = (n-1,n-2), (n,n-2), (2n-2,2n-4), (2n-1,2n-4),$$
and the group vanishes otherwise.
\end{lemma}
\begin{proof} 
Using the decomposition \eqref{EqExteriorPowerEtaPrime}
and the fact that $\Wed^i \mcE^{\vee}$ is trivial, we have
\begin{align}\label{Eq1LemmaR2R1}
\notag
H^q \left(\FF,\,\Wed^{p-1} \eta'\otimes \mcR_2/\mcR_1\right)
& =
H^q\left( \FF, \bigoplus_{i+j = p-1} \Wed^i \mcE^{\vee}\otimes \text{Sym}^i\mcR_1 \otimes \Wed^j \mcQ_{n-1}^{\vee}\otimes \text{Sym}^j\left(\mcR_2/\mcR_1\right)
\otimes \mcR_2/\mcR_1
\right)
\\
& =
\bigoplus_{i+j=p-1} \Wed^i E^\vee 
\otimes 
H^q\Big(\FF,\, \bS_{i} \mcR_1 \otimes \bS_{1^j} \mcQ_{n-1}^{\vee} \otimes \bS_{j+1} \left(\mcR_2/\mcR_1\right)\Big).
\end{align}
Our goal is to compute
the cohomology groups in the previous line for  $0 \leq i \leq n, 0 \leq j \leq n-1$.

As outlined above, we first consider the higher direct images by the  projection $\pi: \FF \to \Gr(1,E)$.
By a standard abuse of notation, 
we use the same symbols  $\mcR_1$ and $ \mcQ_{n-1}$ to denote the tautological sub-bundle and quotient bundle on $\Gr(1,E)$.
Then, by functoriality of Schur functors, we have
  $\pi^{\star}(\bS_{\lambda}\mcR_1) =  \bS_{\lambda} \mcR_1$ and $\pi^{\star}(\bS_{\lambda} \mcQ_{n-1}^{\vee}) = \bS_{\lambda} \mcQ_{n-1}^{\vee}$ for any weight $\lambda$. 
The projection formula \cite[Lemma 20.50.2]{stacks} yields
\begin{equation}\label{Eq2LemmaR2R1}
R^{h} \pi_{\star}\left( \bS_{i} \mcR_1 \otimes \bS_{1^j} \mcQ_{n-1}^{\vee} \otimes \bS_{j+1} \left(\mcR_2/\mcR_1\right) \right) = 
\bS_{i} \mcR_1 \otimes \bS_{1^j} \mcQ_{n-1}^{\vee} \otimes R^{h} \pi_{\star} \left(\bS_{j+1} \left(\mcR_2/\mcR_1\right) \right).
\end{equation}
Note that $\text{Flag}(1,2;E) = \Gr(1,\mcQ_{n-1})$ is a relative Grassmannian over $\Gr(1,E)$, 
with tautological sub-bundle isomorphic to $\mcR_2/\mcR_1$ and quotient bundle isomorphic to $\mcQ_{n-2}$. 
Thus, the term
$$
R^{h} \pi_{\star} (\bS_{j+1} (\mcR_2/\mcR_1))
=
R^h \pi_{\star} (\bS_{-j-1} (\mcR_2/\mcR_1)^{\vee} \otimes \bS_{0^{n-2}}\mcQ_{n-2}^{\vee} ) $$
can be calculated using  Bott's theorem \cite[Corollary 4.1.9]{WeymanBook}.
We apply Bott's algorithm 
\cite[Remark 4.1.5]{WeymanBook} to the weight
$(-j-1,0^{n-2})$,
 and find that the term is non-zero precisely 
when $h= n-2$ and $j \geq n-2$. 
Combining this with our earlier bound $j\leq n-1$, we are left with $h = n-2$ and $j \in \{n-2,n-1\}$, where we obtain
$$
R^{n-2} \pi_{\star} \left(\bS_{j+1} \left(\mcR_2/\mcR_1\right) \right) = \bS_{(-1)^{n-2},n-3-j} \mcQ_{n-1}^{\vee}.
$$
Combining with \eqref{Eq2LemmaR2R1},
for both  $j =n-2,n-1$ we obtain
\begin{align*}
R^{n-2} \pi_{\star}\left( \bS_{i} \mcR_1 \otimes \bS_{1^j} \mcQ_{n-1}^{\vee} \otimes \bS_{j+1} \left(\mcR_2/\mcR_1\right) \right)
& = \bS_{i} \mcR_1 \otimes \bS_{1^j} \mcQ_{n-1}^{\vee} \otimes \bS_{(-1)^{n-2},n-3-j} \mcQ_{n-1}^{\vee} \\
& = \bS_{i} \mcR_1 \otimes \bS_{1^j} \mcQ_{n-1}^{\vee} \otimes \bS_{1^{n-2},n-1-j} \mcQ_{n-1}^{\vee} \otimes \bS_{(-2)^{n-1}} \mcQ_{n-1}^{\vee}\\
& = \bS_{i} \mcR_1 \otimes \bS_{2^{n-2},1} \mcQ_{n-1}^{\vee} \otimes \bS_{(-2)^{n-1}} \mcQ_{n-1}^{\vee}\\
& = \bS_{-i} \mcR_1^{\vee} \otimes \bS_{0^{n-2},-1} \mcQ_{n-1}^{\vee},
\end{align*}
where 
$
\bS_{1^j} \mcQ_{n-1}^{\vee} \otimes \bS_{1^{n-2},n-1-j} \mcQ_{n-1}^{\vee}  = \bS_{2^{n-2},1} \mcQ_{n-1}^{\vee} 
$ 
by Pieri's rule \cite[Corollary 2.3.5]{WeymanBook}.
Since 
$R^h\pi_{\star}$ vanishes for all $h \ne n-2$, 
the Leray spectral sequence  \cite[Lemma 20.13.4]{stacks} yields
\begin{equation} \label{Eq3LemmaR2R1}
H^{q}\big(\FF,\, \bS_{i} \mcR_1 \otimes \bS_{1^j} \mcQ_{n-1}^{\vee} \otimes \bS_{j+1} \left(\mcR_2/\mcR_1\right)\big)
= 
H^{q - n+2}\big(\Gr(1,E),\,\bS_{-i} \mcR_1^{\vee} \otimes \bS_{0^{n-2},-1} \mcQ_{n-1}^{\vee}\big)
\end{equation}
for $j \in \{n-2,n-1\}$, and
$H^{q}\big(\FF,\, \bS_{i} \mcR_1 \otimes \bS_{1^j} \mcQ_{n-1}^{\vee} \otimes \bS_{j+1} \left(\mcR_2/\mcR_1\right)\big)
= 0$ otherwise.
We use Bott's theorem again.
Applying Bott's algorithm to the weight $(-i, 0^{n-2},-1)$,
we find  that  \eqref{Eq3LemmaR2R1} vanishes unless $i=0, q- n+2 = 0$ or $i=n-1, q-n+2 = n-2$,
 where we obtain
\begin{align*}
H^{0}\big(\Gr(1,E),\bS_{0^{n-2},-1} \mcQ_{n-1}^{\vee}\big)
&= \bS_{0^{n-1},-1}E^\vee = E\\
\intertext{and}
H^{n-1}\big(\Gr(1,E),\bS_{-n+1} \mcR_1^{\vee} \otimes \bS_{0^{n-2},-1} \mcQ_{n-1}^{\vee}\big)&= \bS_{(-1)^{n}}E^\vee = \Wed^{n}E.
\end{align*}
To summarize, we have
$$
H^{q}(\FF,\, \bS_{i} \mcR_1 \otimes \bS_{1^j} \mcQ_{n-1}^{\vee} \otimes \bS_{j+1} \left(\mcR_2/\mcR_1\right)) = 
\begin{cases}
E& \text{ if }q = n-2, i =0, j = n-2,n-1,\\
\Wed^{n}E& \text{ if }  q = 2n-4, i = n-1, j = n-2,n-1,\\
0 & \text{ else. }\\
\end{cases}
$$
Now, the conclusion follows from \eqref{Eq1LemmaR2R1} tensoring with $\Wed^iE^\vee$.
\end{proof}

\begin{lemma} \label{LemmaR1}
We have
$
H^q \big(\FF,\,\Wed^{p-1} \eta' \otimes \mcR_1\big) = E$ for 
$$(p,q) = (n,n-1),(n+1,n-1),(2n-1,2n-3), (2n,2n-3),$$
and the group vanishes otherwise.
\end{lemma}
\begin{proof} 
We follow the same steps as Lemma \ref{LemmaR2R1}, and  keep the same notation.
We have
\begin{equation}\label{Eq1LemmaR1}
H^q \left(\FF,\,\Wed^{p-1} \eta' \otimes \mcR_1\right) =
\bigoplus_{i+j = p-1} \Wed^i E^\vee \otimes 
H^q\big(\FF,\, \bS_{i+1} \mcR_1 \otimes \bS_{1^j} \mcQ_{n-1}^{\vee} \otimes \bS_{j} \left(\mcR_2/\mcR_1\right)\big)
\end{equation}
and the relevant indices are $0 \leq i \leq n$ and $0 \leq j \leq n-1$.
We have
$$
R^h\pi_\star\big(\bS_{i+1} \mcR_1 \otimes \bS_{1^j} \mcQ_{n-1}^{\vee} \otimes \bS_{j} \left(\mcR_2/\mcR_1\right)\big) =
\bS_{i+1} \mcR_1 \otimes \bS_{1^j} \mcQ_{n-1}^{\vee} \otimes R^h\pi_\star\big(\bS_{j} \left(\mcR_2/\mcR_1\right)\big).
$$
We calculate the  term
$
R^h\pi_\star\big(\bS_{j} (\mcR_2/\mcR_1)\big) 
= 
R^h\pi_\star\big(\bS_{-j} (\mcR_2/\mcR_1)^\vee \otimes \bS_{0^{n-2}} \mcQ_{n-2}^\vee\big) 
$
applying Bott's algorithm to the weight $(-j,0^{n-2})$,
and find  non-zero terms  when $j =0, h=0$ and $j = n-1$, $h = n-2$,
specifically
$
R^0\pi_\star\big(\bS_{0} (\mcR_2/\mcR_1)\big) = \mcO_{\Gr(1,E)}
$
and
$R^{n-2}\pi_\star\big(\bS_{n-1} \big(\mcR_2/\mcR_1\big)\big) = \bS_{(-1)^{n-1}}\mcQ_{n-1}^\vee,
$
so the two non-vanishing higher direct images are 
$
R^0\pi_\star\big(\bS_{i+1} \mcR_1 \big) =
\bS_{i+1} \mcR_1$ 
and
\begin{align*}
R^{n-2}\pi_\star\big(\bS_{i+1} \mcR_1 \otimes \bS_{1^{n-1}} \mcQ_{n-1}^{\vee} \otimes \bS_{n-1} \left(\mcR_2/\mcR_1\right)\big) &=
\bS_{i+1} \mcR_1 \otimes \bS_{1^{n-1}} \mcQ_{n-1}^{\vee} \otimes \bS_{(-1)^{n-1}}\mcQ_{n-1}^\vee=
\bS_{i+1} \mcR_1.
\end{align*}
Since, for each $j$, the higher direct images vanish for all but one $h$,
we  apply the Leray spectral sequence and Bott's theorem to find
\begin{align*}
H^{q}\big(\FF,\, \bS_{i+1} \mcR_1\big) 
&=
H^{q}\big(\Gr(1,E),\,  \bS_{-i-1}   \mcR_1^{\vee} \otimes \bS_{0^{n-1}} \mcQ_{n-1}^{\vee} \big) \\
& = 
\begin{cases}
\bS_{(-1)^{n}} E^{\vee} & \text{ if } i=q = n-1, \\
\bS_{(-1)^{n-1},-2} E^{\vee} & \text{ if } i=n,q=n-1, \\
0 & \text{ else}
\end{cases}
\end{align*}
for $j=0,$ and
\begin{align*}
H^{q}\big(\FF,\, \bS_{i+1} \mcR_1 \otimes \bS_{1^{n-1}} \mcQ_{n-1}^{\vee} \otimes \bS_{n-1} \left(\mcR_2/\mcR_1\right)\big) &
= 
H^{q - n+2}\big(\Gr(1,E),\,  \bS_{-i-1}   \mcR_1^{\vee} \otimes  \bS_{0^{n-1}} \mcQ_{n-1}^{\vee}\big) \\
& = 
\begin{cases}
\bS_{(-1)^{n}} E^{\vee} & \text{ if } i= n-1,q=2n-3, \\
\bS_{(-1)^{n-1},-2} E^{\vee} & \text{ if }  i= n,q=2n-3, \\
0 & \text{ else}
\end{cases}
\end{align*}
for $j=n-1$,
and the conclusion follows from \eqref{Eq1LemmaR1}  tensoring with $\Wed^iE^\vee$.
\end{proof}

\begin{prop} \label{PropositionEtaPrime} 
We have $H^q \big(\FF,\,\Wed^{p} \eta'\big)=\kk$ for 
$$(p,q) = (0,0), (n-1,n-2), (n,n-1), (2n-1,2n-3),$$
and the group vanishes otherwise.
\end{prop}
\begin{proof}
The proof is essentially identical to that of Lemma \ref{LemmaR1}.
\end{proof}

\begin{prop} \label{PropositionR2}
We have
$
H^q \big(\FF,\,\Wed^{p-1} \eta' \otimes \mcR_2\big) =  E$
 for $$ (p,q) = (n-1,n-2), (n+1,n-1), (2n-2,2n-4),   (2n,2n-3),$$
the following equations for
 $ (p,q)= (n,n-2), (n,n-1), (2n-1,2n-4), (2n-1,2n-3) $
\begin{align*}
\dim H^{n-2}\left(\FF,\,\Wed^{n-1} \eta' \otimes \mcR_2\right) 
& = 
 \dim H^{n-1}\left(\FF,\,\Wed^{n-1} \eta' \otimes \mcR_2\right)
\\
\intertext{and}
\dim H^{2n-4}\left(\FF,\,\Wed^{2n-2} \eta' \otimes \mcR_2\right)
& = 
\dim H^{2n-3}\left(\FF,\,\Wed^{2n-2} \eta' \otimes \mcR_2\right),
\end{align*}
and the group vanishes otherwise.
\end{prop}
\begin{proof} 
Tensoring the  exact sequence $ 0 \to \mcR_1 \to \mcR_2 \to \mcR_2/\mcR_1\to 0$
with  $\Wed^{p-1} \eta'$ and taking cohomology,
we obtain exact sequences for each $p,q$

\begin{tikzpicture}[descr/.style={fill=white,inner sep=1.5pt}]
        \matrix (m) [
            matrix of math nodes,
            row sep=2em,
            column sep=2.5em,
            text height=1.5ex, text depth=0.25ex
        ]
        {H^{q}\left(\FF,\,\Wed^{p-1} \eta' \otimes \mcR_1\right) &
		H^q\left(\FF,\, \Wed^{p-1} \eta' \otimes \mcR_2\right) & 
		H^q\left(\FF,\,\Wed^{p-1} \eta' \otimes \left(\mcR_2/\mcR_1\right)\right) \\
          H^{q+1}\left(\FF,\,\Wed^{p-1} \eta' \otimes \mcR_1\right) & 
           	H^{q+1}\left(\FF,\, \Wed^{p-1} \eta' \otimes \mcR_2\right) & 
		H^{q+1}\left(\FF,\,\Wed^{p-1} \eta' \otimes \left(\mcR_2/\mcR_1\right)\right). \\
        };

        \path[overlay,->, font=\scriptsize,>=latex]
        (m-1-1) edge (m-1-2)
        (m-1-2) edge (m-1-3)
        (m-1-3) edge[out=355,in=175]  (m-2-1)
        (m-2-1) edge (m-2-2)
        (m-2-2) edge (m-2-3);
\end{tikzpicture}

We use Lemmas \ref{LemmaR2R1} and \ref{LemmaR1}.
If  $p\ne n, 2n-1$, since $n\geq 4$ and hence $2n-2>n+1$,  there is at most one non-vanishing term in the two outer columns, and the result follows immediately.
For $p= n, 2n-1$, the non-trivial sequences are 
\begin{align*}
0 \longrightarrow H^{n-2}\left(\FF,\,\Wed^{n-1} \eta' \otimes \mcR_2\right) \longrightarrow 
	E \longrightarrow
	E \longrightarrow 
	H^{n-1}\left(\FF,\,\Wed^{n-1} \eta' \otimes \mcR_2\right) \longrightarrow 0
\end{align*}
and
\begin{align*}
0 \longrightarrow H^{2n-4}\left(\FF,\,\Wed^{2n-2} \eta' \otimes \mcR_2\right) \longrightarrow 
	E \longrightarrow
E \longrightarrow 
	H^{2n-3}\left(\FF,\,\Wed^{2n-2} \eta' \otimes \mcR_2\right) \longrightarrow 0.
\end{align*}
This completes the proof.
\end{proof}

\begin{lemma} \label{LemmaR1R2R1}
We have
$H^q \big(\FF,\,\Wed^{p-2} \eta' \otimes \mcR_1 \otimes (\mcR_2/\mcR_1)\big) = \Wed^{2} E$
for  
$$
(p,q) = (2n-2,2n-4), (2n-1,2n-4),(2n,2n-3),(2n+1,2n-3)
$$
and the group vanishes otherwise.
\end{lemma}
\begin{proof} 
Again, the proof follows the exact same outline as Lemmas \ref{LemmaR2R1}, \ref{LemmaR1}.
In brief, one finds
\begin{align*} 
H^q \left(\FF,\,\Wed^{p-2} \eta' \otimes \mcR_1 \otimes \left(\mcR_2/\mcR_1\right)\right) &=
\bigoplus_{i+j=p-2} \Wed^i E^\vee \otimes
H^q\big(\FF,\, \bS_{i+1} \mcR_1 \otimes \bS_{1^j} \mcQ_{n-1}^{\vee} \otimes \bS_{j+1} \left(\mcR_2/\mcR_1\right)\big),
\\
R^h\pi_\star \big(\bS_{i+1} \mcR_1 \otimes \bS_{1^j} \mcQ_{n-1}^{\vee} \otimes \bS_{j+1} \left(\mcR_2/\mcR_1\right)\big) 
&=
\bS_{i+1} \mcR_1 \otimes \bS_{1^j} \mcQ_{n-1}^{\vee} \otimes R^h\pi_\star \big(\bS_{j+1} \left(\mcR_2/\mcR_1\right)\big),
\\
R^h\pi_\star \big(\bS_{j+1} \left(\mcR_2/\mcR_1\right)\big) & \ne 0 
\quad \text{only if} \quad
 h =  n-2, j \in \{n-2,n-1\}, \text{ and } 
\\
H^{q}\big(\FF,\, \bS_{i+1} \mcR_1 \otimes \bS_{1^j} \mcQ_{n-1}^{\vee} \otimes \bS_{j+1} \left(\mcR_2/\mcR_1\right)\big)
&= H^{q - n+2}\big(\Gr(1,E),\,\bS_{-i-1} \mcR_1^{\vee} \otimes \bS_{0^{n-2},-1} \mcQ_{n-1}^{\vee}\big)\\
&
=
\begin{cases}
\bS_{(-1)^n} E^\vee & \text{if } i = n-2, q=2n-4, \\
& j \in \{n-2,n-1\} \\
\bS_{(-1)^{n-2},(-2)^2} E^\vee & \text{if } i = n, q=2n-3, \\
& j \in \{n-2,n-1\}  \\
0 & \text{else}
\end{cases}
\end{align*}
The result now follows by tensoring with $\Wed^iE^\vee$.
\end{proof}

\begin{prop} \label{PropositionWedgeR2}
We have
$
H^q \big(\FF,\,\Wed^{p-2} \eta' \otimes \Wed^2 \mcR_2\big) = 
\Wed^{2} E 
$
for 
$$(p,q) = (2n-2,2n-4), (2n-1,2n-4),(2n,2n-3),(2n+1,2n-3),$$
and the group vanishes otherwise.
\end{prop}
\begin{proof} 
Taking the determinant of  the exact sequence $0 \to \mcR_1 \to \mcR_2 \to \mcR_2/\mcR_1\to 0$  we see that  $ \Wed^2 \mcR_2 = \mcR_1 \otimes \big(\mcR_2/\mcR_1\big)$,
so the result follows directly from Lemma \ref{LemmaR1R2R1}.
\end{proof}

We are now ready to prove the main result of this section.

\begin{proof}[Proof of Theorem \ref{TheoremVarietyMatrices}]
We apply the Kempf-Lascoux-Weyman technique 
to  the desingularization of Theorem \ref{TheoremResolutionSingularities}
and the syzygy bundle $\xi$.
Recall the free graded $A$-modules
$$
\mathsf{F}_i  = \bigoplus_{j \geq 0} H^j \left(\FF, \, \Wed^{i+j} \xi\right) \otimes A(-i-j).
$$
By \eqref{EqExteriorPowerXiPrime} and Propositions 
\ref{PropositionEtaPrime},
\ref{PropositionR2}, \ref{PropositionWedgeR2},
we have 
$H^q \left(\FF, \, \Wed^p \xi'\right)=0$ whenever $p<q$
and 
$\bigoplus_{p\geq 0} H^p \left(\FF, \, \Wed^p \xi'\right)=\Bbbk$.
The filtration on $\xi$ induces, for each $p$, a filtration on $\Wed^p \xi$ with associated graded bundle $\Wed^p \xi'$,
and thus a spectral sequence
$
H^q(\FF, \Wed^p \xi') \Rightarrow H^q(\FF, \Wed^p \xi).
$
We deduce that $\mathsf{F}_i=0$ for all  $i < 0$ and $\mathsf{F}_0 = A$. 
It follows from \cite[Theorem 5.1.2, Theorem 5.1.3 (c)]{Weyman} that $\mfX$ has rational singularities.

Recall that $\mfI$ is a homogeneous ideal of the standard graded polynomial ring $A$,
thus, the affine variety $\mfX = \Spec(A/\mfI)$ can be regarded as a cone over the  projective variety  $\mfX' = \Proj(A/\mfI) \subseteq\P^{n^2+n-1}$.
By  \cite[Theorem 5.1.6 (b)]{WeymanBook}, its degree is given by
\begin{equation}
\label{EqDegreeFormula}
\deg \mfX'  = \sum_{p,q}\frac{(-1)^{p-q}p^4}{4!}h^q\left(\FF,\,\Wed^p\xi\right).
\end{equation}
For each $p$,  the cancellations in the spectral sequence 
$
H^q(\FF, \Wed^p \xi') \Rightarrow H^q(\FF, \Wed^p \xi)
$
occur in pairs with consecutive cohomological degrees, 
therefore we can replace  $\xi$ with $\xi'$ in \eqref{EqDegreeFormula}.
We collect all the non-zero contributions to $h^q(\FF,\, \Wed^p\xi')$ in the following table
\begin{center}
\begin{tabular}{ c|c|c|c|c|c } 
$(p,q)$ & $(-1)^{p-q}$ & $p^4$ & Proposition \ref{PropositionEtaPrime} & Proposition \ref{PropositionR2} & Proposition \ref{PropositionWedgeR2}\\
\hline 
$(0,0)$ & $+$ & $0$ & $1$ & &\\
$(n-1,n-2)$ & $-$ & $(n-1)^4$ & $1$ & $n$ &\\
$(n,n-1)$ & $-$ & $(n)^4$ & $1$ &  &\\
$(n+1,n-1)$ & $+$ & $(n+1)^4$ &  & $n$ &\\
$(2n-2,2n-4)$ & $+$ & $(2n-2)^4$ &  &$n$  &${n \choose 2}$\\
$(2n-1,2n-3)$ & $+$ & $(2n-1)^4$ & $1$ &  &\\
$(2n-1,2n-4)$ & $-$ & $(2n-1)^4$ &  &  &${n \choose 2}$\\
$(2n,2n-3)$ & $-$ & $(2n)^4$ &  & $n$ &${n \choose 2}$\\
$(2n+1,2n-4)$ & $+$ & $(2n+1)^4$ &  &  & ${n \choose 2}$\\
\end{tabular}
\end{center}
where we ignored the terms with 
 $ (p,q)= (n,n-2), (n,n-1), (2n-1,2n-4), (2n-1,2n-3) $,
since they cancel out in \eqref{EqDegreeFormula} by Proposition \ref{PropositionR2}.
We compute

\begin{align*}
 \deg \mfX' & = \sum_{p,q}\frac{(-1)^{p-q}p^4}{4!}h^q\left(\FF,\,\Wed^p\xi'\right) \\
& = 
\frac{1}{24}
\left[
-(n-1)^4(1+n)-n^4 +(n+1)^4n +(2n-2)^4\left( n + {n \choose 2}\right)\right.
\\
& 
\left. \qquad \quad
+(2n-1)^4\left(1-{n\choose 2}\right)-(2n)^4\left( n + {n \choose 2}\right) + (2n+1)^4{n \choose 2}
\right]\\
& = \frac{(n-1)n(n+1)(3n-2)}{12}. \qedhere
\end{align*}
\end{proof}

\section{An initial complex}\label{SectionSquarefreeInitial}

In this section, we prove the primeness of the  ideal $\mfI$ introduced in  \eqref{EqDefinitionIdealI}, 
which  defines an open subscheme of the fiber of the map $\Hilb^{(m,2)}(\AA^2) \to \Hilb^2(\AA^2)$ over the point $[V(x,y^2)]$.

\begin{thm}\label{TheoremIPrime}
Assume $\ch(\Bbbk)=0$.
The ideal $\mfI \subseteq A$ is prime.
\end{thm}
We  prove Theorem \ref{TheoremIPrime} combinatorially.
We introduce a simplicial complex $\Delta$ 
and study the enumerative combinatorics of its facets.
Combining this analysis with the numerical information of Theorem \ref{TheoremVarietyMatrices},
 we  show that 
$\Delta$ is an initial complex of  $\mfI$ and $\sqrt{\mfI}$, 
and deduce  Theorem \ref{TheoremIPrime} and the Cohen-Macaulayness of $\Delta$ as consequences.
These results connect  the variety $\mfX$ 
to the recent developments on 
 square-free Gr\"obner degenerations 
 \cite{CV,CDM},
 suggesting further directions of investigation in this area;  we discuss some problems at the end of the section.

Throughout this section, we assume $n\geq 2$, since $\mfI = (0)$ for $n=1$.
We equip  the polynomial ring $A= \kk[w_{i,j}]$ with the graded reverse lexicographic order induced by the ``opposite'' ordering of the variables 
\begin{equation}\label{EqTermOrderA}
w_{1,1} < w_{1,2} < w_{1,3} <\cdots < w_{n+1,n-1} < w_{n+1,n},
\end{equation}
that is, 
the variables appear in  increasing order in the  $(n+1)\times n$  generic matrix 
$\bfW=(w_{i,j})$.

\begin{remark}
This term order  is antidiagonal, in the sense that
the leading monomial of any minor of $\bfW$ is the product of 
the entries in its antidiagonal.
Antidiagonal term orders give rise to rich combinatorics, and they have been widely used in the literature 
on determinantal ideals, most notably in \cite{KM}.
The simplicial complex  we  introduce in this section is  a sub-complex of the initial complex associated in \cite{KM} to the determinantal ideal of maximal minors of $\bfW$.
\end{remark}

We identify some distinguished monomials of $A$, which arise as leading monomials of certain polynomials in $\mfI$,
see Section \ref{SectionLocalEquations}.
Given subsets $\mcA \subseteq [n+1]$ and $ \mcB \subseteq [n]$ of the same cardinality, 
we denote by  $\bfW_{\mcA,\mcB}$ the square submatrix of $\bfW$ with set of rows $\mcA$ and set of columns $\mcB$.

\begin{prop}\label{PropositionLeadingMonomial}
For each $h = 2, \ldots,n+1,$ we have
\begin{equation}\tag{X}\label{EqMonomialXh}
\LM(f_h) = \bfx_h :=
\prod_{i=2}^{h-1}w_{i,n+2-i}
\prod_{i=h+1}^{n+1} w_{i,n+3-i}
\end{equation}
and
\begin{equation}\tag{Y}\label{EqMonomialYh}
\LM(F_h) = \bfy_h :=
\prod_{i=1}^{h-1}w_{i,n+1-i}
\prod_{i=h+1}^{n+1} w_{i,n+2-i}.
\end{equation}
For each $h = 3, \ldots, n+1,$ consider the monomial
\begin{equation}\tag{Z}\label{EqMonomialZh}
\bfz_h :=
\prod_{i=3}^{h-1} w_{i,n+3-i}
\prod_{i=h}^{n+1} w_{i,n+2-i}
\prod_{i=h+1}^{n+1} w_{i,n+4-i}.
\end{equation}
Then
$\LM(f_1) = \bfz_{n+1}$, and, 
 for each $h = 3, \ldots, n,$
we have  
$$
\LM\big(f_1 \cdot\det \bfW_{(h+1,\ldots, n+1),(2,\ldots, n-h+2)}
+
f_2 \cdot 
    \det \bfW_{(h+1,\ldots, n+1),(1,3, 4,\ldots, n-h+2)}\big)
    = \bfz_h.
    $$
\end{prop}
\begin{proof}
Recall that 
$
f_h = \sum_{j=1}^{n}\det\bfW^{(j,h),j}
$
and 
 $F_h = \det \bfW^h$
by  Lemma \ref{LemmaDescriptionFi}.
The formulas for $\LM(f_h)$ and $\LM(F_h)$ follow immediately by the antidiagonal property of the  revlex order.
For the third formula,
 let $h \in  \{3, \ldots, n\}$
and  consider the polynomial $ p_1 + p_2$, where
\begin{align*}
p_1 &=  f_1 \cdot\det \bfW_{(h+1,\ldots, n+1),(2,\ldots, n-h+2)}
= \sum_{j=1}^{n}\det\bfW^{(j,1),j} \cdot\det \bfW_{(h+1,\ldots, n+1),(2,\ldots, n-h+2)}
\\
\intertext{and}
p_2 &= f_2 \cdot 
    \det \bfW_{(h+1,\ldots, n+1),(1,3,4, \ldots, n-h+2)}
    =\sum_{j=1}^{n}\det\bfW^{(j,2),j}
    \cdot 
    \det \bfW_{(h+1,\ldots, n+1),(1,3,4, \ldots, n-h+2)}.
\end{align*}
We  show that $\LM(p_1+p_2) = \bfz_h$.
Let $q_1$ and $q_2$ be the first non-zero addend in $p_1$ and $p_2$, respectively. Explicitly, they are
\begin{align*}
q_1 
&= \det\bfW^{(2,1),2} \cdot\det \bfW_{(h+1,\ldots, n+1),(2,\ldots, n-h+2)} \\
&= -\det\bfW^{(1,2),2} \cdot\det \bfW_{(h+1,\ldots, n+1),(2,\ldots, n-h+2)}\\
&= -\det\bfW_{(3, \ldots, n+1),(1,3,4,\ldots,n)} \cdot\det \bfW_{(h+1,\ldots, n+1),(2,\ldots, n-h+2)}\\
\intertext{and}
q_2 
&= \det\bfW^{(1,2),1}\cdot \det \bfW_{(h+1,\ldots, n+1),(1,3,4, \ldots, n-h+2)}\\
&= \det\bfW_{(3,\ldots,n+1),(2,\ldots,n)} \cdot  \det \bfW_{(h+1,\ldots, n+1),(1,3,4, \ldots, n-h+2)}.
\end{align*}
By the revlex order,
the monomials in the addend $q_1$ are higher than the monomials in the remaining addends of $p_1$, 
and likewise for $q_2$ and $p_2$.
Therefore, 
it suffices to show that $\LM(q_1+q_2) = \bfz_h$.

Again, by the revlex order, the monomials in  $w_{3,n}w_{4,n-1}\cdots w_{h-1,n-h+4}	\cdot
\det\bfW_{(h,\ldots,n+1),(1,3,\ldots, n-h+3)}$ 
are higher than the remaining monomials in $\det\bfW_{(3, \ldots, n+1),(1,3,4,\ldots,n)}$,
 likewise for the addend
$w_{3,n}w_{4,n-1}\cdots w_{h-1,n-h+4}	\cdot
\det\bfW_{(h,\ldots,n+1),(2,\ldots, n-h+3)}$ 
of $\det\bfW_{(3, \ldots, n+1),(2,\ldots,n)}$.
Thus, we let 
\begin{align*}
r_1 &= 
-\det\bfW_{(h,\ldots,n+1),(1,3,\ldots, n-h+3)} \cdot\det \bfW_{(h+1,\ldots, n+1),(2,\ldots, n-h+2)}\\
\intertext{and}
r_2 &= 
\det\bfW_{(h,\ldots,n+1),(2,\ldots, n-h+3)}
    \cdot 
        \det \bfW_{(h+1,\ldots, n+1),(1,3,4, \ldots, n-h+2)},
        \end{align*}
and it suffices to show that
$\LM\big(w_{3,n}w_{4,n-1}\cdots w_{h-1,n-h+4}(r_1+r_2)\big) = \bfz_h$,
equivalently, that 
 \begin{equation}\label{EqLeadingMonomialR1R2}
  \LM(r_1+r_2) = \prod_{i=h}^{n+1} w_{i,n+2-i}
\prod_{i=h+1}^{n+1} w_{i,n+4-i}.
 \end{equation}
We prove \eqref{EqLeadingMonomialR1R2} by decreasing  induction on $h$, starting with $n$. 
When $h=n$, we have
\begin{align*}
r_1 + r_2 
& =  
	-\begin{vmatrix}
	w_{n,1} & w_{n,3} \\
	w_{n+1,1} & w_{n+1,3}
	\end{vmatrix}
	w_{n+1,2}
	+
	\begin{vmatrix}
	w_{n,2} & w_{n,3} \\
	w_{n+1,2} & w_{n+1,3}
	\end{vmatrix}
	w_{n+1,1}\\
& = - (w_{n,1}w_{n+1,3} - w_{n,3}w_{n+1,1})w_{n+1,2} 
		+ (w_{n,2}w_{n+1,3}-w_{n,3}w_{n+1,2})w_{n+1,1} \\
& = - w_{n,1}w_{n+1,2}w_{n+1,3}+ w_{n,2}w_{n+1,1}w_{n+1,3}.
\end{align*}
Thus, $\LM(r_1+r_2) = w_{n,2}w_{n+1,1}w_{n+1,3}$ as required.
Now, assume $h \leq n-1$.
As before, 
the monomials of $w_{h,n-h+3}\cdot \det\bfW_{(h+1,\ldots,n+1),(1,3,\ldots, n-h+2)}$ are higher than the remaining monomials of $\det\bfW_{(h,\ldots,n+1),(1,3,\ldots, n-h+3)}$,
and likewise for 
$w_{h,n-h+3} \cdot \det\bfW_{(h+1,\ldots,n+1),(2,\ldots, n-h+2)}$
 and
$ \det\bfW_{(h,\ldots,n+1),(2,\ldots, n-h+3)}$.
The corresponding addends of $r_1$ and $r_2$ are
\begin{align*}
-w_{h,n-h+3}\cdot \det\bfW_{(h+1,\ldots,n+1),(1,3,\ldots, n-h+2)} \cdot\det \bfW_{(h+1,\ldots, n+1),(2,\ldots, n-h+2)}\\
\intertext{and}
 w_{h,n-h+3} \cdot \det\bfW_{(h+1,\ldots,n+1),(2,\ldots, n-h+2)}
    \cdot 
 \det \bfW_{(h+1,\ldots, n+1),(1,3,4, \ldots, n-h+2)},
 \end{align*}
hence all these terms cancel out in $r_1+r_2$.
The next highest monomials in $\det\bfW_{(h,\ldots,n+1),(1,3,\ldots, n-h+3)}$
are the ones in the addend $w_{h,n-h+2}w_{h+1,n-h+3}\cdot \det\bfW_{(h+2,\ldots,n+1),(1,3,\ldots, n-h+1)}$,
and likewise 
for 
$ \det\bfW_{(h,\ldots,n+1),(2,\ldots, n-h+3)}$
and 
$w_{h,n-h+2}w_{h+1,n-h+3}\cdot \det\bfW_{(h+2,\ldots,n+1),(2,\ldots, n-h+1)}$.
Proceeding as in the previous reduction, 
we let 
\begin{align*}
r'_1 &= 
-\det\bfW_{(h+2,\ldots,n+1),(1,3,\ldots, n-h+1)} \cdot\det \bfW_{(h+1,\ldots, n+1),(2,\ldots, n-h+2)}\\
\intertext{and}
r'_2 &= 
\det\bfW_{(h+2,\ldots,n+1),(2,\ldots, n-h+1)}
    \cdot 
        \det \bfW_{(h+1,\ldots, n+1),(1,3,4, \ldots, n-h+2)},
        \end{align*}
and it suffices to show that 
$
  \LM(r'_1+r'_2) = \prod_{i=h+1}^{n+1} w_{i,n+2-i}
\prod_{i=h+2}^{n+1} w_{i,n+4-i}.
$
This is exactly the induction hypothesis  (reversing the order of the factors),
so the proof is concluded.
\end{proof}

Following Proposition \ref{PropositionLeadingMonomial}
 we define the square-free monomial ideal 
\begin{equation}\label{EqDefinitionMonomialIdealK}
\mfK = \big(\bfx_2, \ldots, \bfx_{n+1},
\bfy_2, \ldots, \bfy_{n+1},
\bfz_3, \ldots, \bfz_{n+1}\big)
\subseteq A.
\end{equation}
Note that, by Proposition \ref{PropositionLeadingMonomial}, we have the inclusion of monomial ideals $\mfK \subseteq \iin(\mfI)$.

\begin{example}\label{ExampleMonomialsN5}
It is helpful to visualize the    generators of $\mfK$  pictorially, 
identifying each squarefree monomial with the set of positions of  variables in the $(n+1)\times n$ matrix $\bfW$.
In this example, we illustrate the  case $n=5$.
The diagrams below highlight the antidiagonal configurations of the variables in these monomials, which originate from the antidiagonal property of the term order, and which dominate the combinatorics of the simplicial complex defined by $\mfK$. 

The monomials \eqref{EqMonomialXh}, \eqref{EqMonomialYh}, and \eqref{EqMonomialZh} are

\begin{align*}
\bfx_2 &= w_{3,5}w_{4,4}w_{5,3}w_{6,2},  &\bfy_2 =  w_{1,5}w_{3,4}w_{4,3}w_{5,2}w_{6,1}, 
\quad & \bfz_3 = w_{3,4}w_{4,3}w_{5,2}w_{6,1}w_{4,5}w_{5,4}w_{6,3},\\
\bfx_3 &= w_{2,5}w_{4,4}w_{5,3}w_{6,2}, & \bfy_3 =  w_{1,5}w_{2,4}w_{4,3}w_{5,2}w_{6,1},
\quad & \bfz_4 = w_{3,5}w_{4,3}w_{5,2}w_{6,1}w_{5,4}w_{6,3},\\
\bfx_4 &= w_{2,5}w_{3,4}w_{5,3}w_{6,2}, & \bfy_4 =  w_{1,5}w_{2,4}w_{3,3}w_{5,2}w_{6,1},
\quad & \bfz_5 = w_{3,5}w_{4,4}w_{5,2}w_{6,1}w_{6,3},\\
\bfx_5 &= w_{2,5}w_{3,4}w_{4,3}w_{6,2}, & \bfy_5 =  w_{1,5}w_{2,4}w_{3,3}w_{4,2}w_{6,1},
\quad & \bfz_6 = w_{3,5}w_{4,4}w_{5,3}w_{6,1},\\
\bfx_6 &= w_{2,5}w_{3,4}w_{4,3}w_{5,2}, & \bfy_6 =  w_{1,5}w_{2,4}w_{3,3}w_{4,2}w_{5,1}.
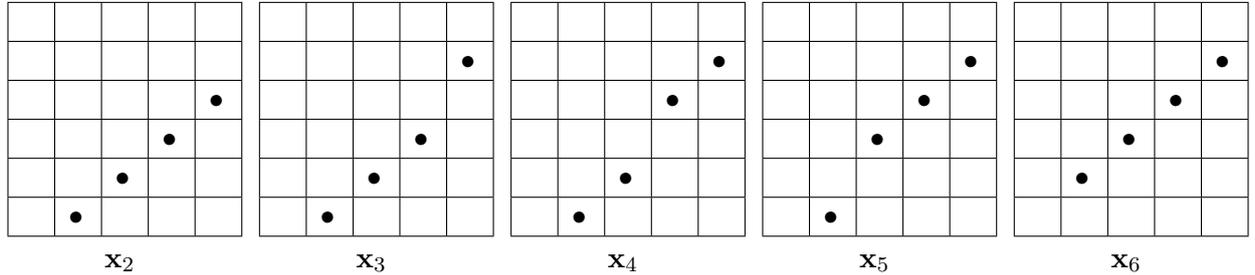
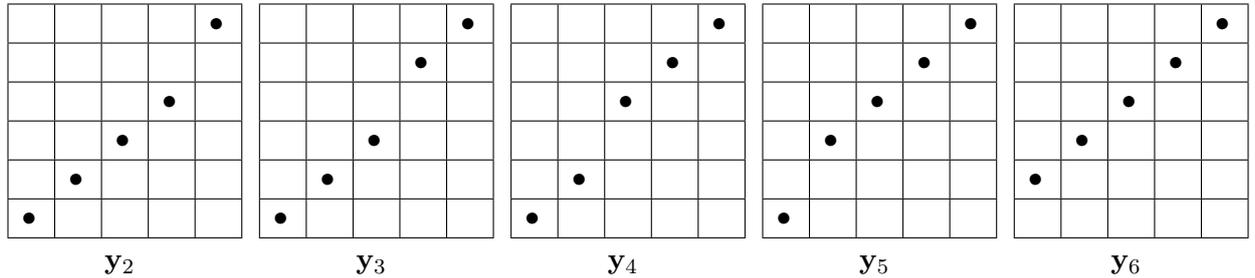
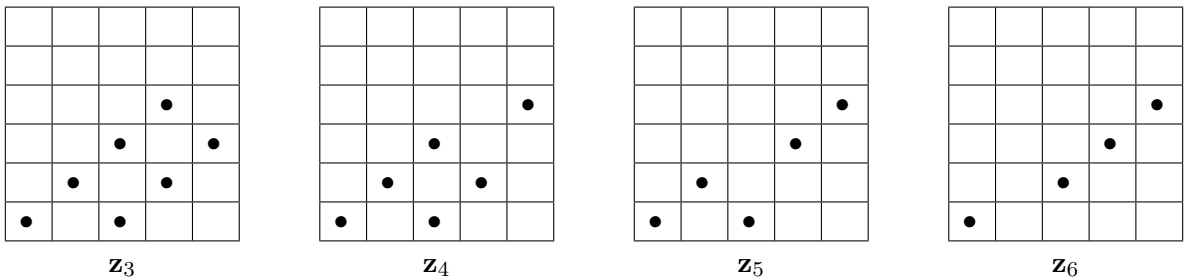
\quad &
\end{align*}

\vspace*{-0.5cm}
\begin{center}
\begin{multicols}{5}
$$
\begin{tabular}{ | m{0.25cm}| m{0.25cm}| m{0.25cm}| m{0.25cm}| m{0.25cm}| } 
 \hline
 &  &  & &  \\ 
 \hline
 &  &  & &  \\ 
 \hline
 &  &  & & \textbullet \\ 
 \hline
 &  &  & \textbullet&  \\ 
 \hline
 &  & \textbullet & &  \\ 
 \hline
 & \textbullet &  & &  \\ 
 \hline
 \end{tabular}
$$
$$
\bfx_2 
$$
\columnbreak

$$
\begin{tabular}{ | m{0.25cm}| m{0.25cm}| m{0.25cm}| m{0.25cm}| m{0.25cm}| } 
 \hline
 &  &  & &  \\ 
 \hline
 &  &  & & \textbullet \\ 
 \hline
 &  &  & &  \\ 
 \hline
 &  &  & \textbullet&  \\ 
 \hline
 &  & \textbullet & &  \\ 
 \hline
 & \textbullet &  & &  \\ 
 \hline
 \end{tabular}
$$
$$
\bfx_3
$$

\columnbreak

$$
\begin{tabular}{ | m{0.25cm}| m{0.25cm}| m{0.25cm}| m{0.25cm}| m{0.25cm}| } 
 \hline
 &  &  & &  \\ 
 \hline
 &  &  & & \textbullet \\ 
 \hline
 &  &  & \textbullet &  \\ 
 \hline
 &  &  & &  \\ 
 \hline
 &  & \textbullet & &  \\ 
 \hline
 & \textbullet &  & &  \\ 
 \hline
 \end{tabular}
$$
$$
\bfx_4
$$

\columnbreak

$$
\begin{tabular}{ | m{0.25cm}| m{0.25cm}| m{0.25cm}| m{0.25cm}| m{0.25cm}| } 
 \hline
 &  &  & &  \\ 
 \hline
 &  &  & & \textbullet \\ 
 \hline
 &  &  & \textbullet&  \\ 
 \hline
 &  & \textbullet & &  \\ 
 \hline
 &  &  & &  \\ 
 \hline
 & \textbullet &  & &  \\ 
 \hline
 \end{tabular}
$$
$$
\bfx_5 
$$

\columnbreak
$$
\begin{tabular}{ | m{0.25cm}| m{0.25cm}| m{0.25cm}| m{0.25cm}| m{0.25cm}| } 
 \hline
 &  &  & &  \\ 
 \hline
 &  &  & & \textbullet \\ 
 \hline
 &  &  & \textbullet&  \\ 
 \hline
 &  & \textbullet & &  \\ 
 \hline
 & \textbullet &  & &  \\ 
 \hline
  &  &  & &  \\ 
 \hline
 \end{tabular}
$$
$$
\bfx_6
$$

\end{multicols}
\end{center}

\vspace*{-0.8cm}
\begin{center}
\begin{multicols}{5}

$$
\begin{tabular}{ | m{0.25cm}| m{0.25cm}| m{0.25cm}| m{0.25cm}| m{0.25cm}| } 
 \hline
 &  &  & & \textbullet \\ 
 \hline
 &  &  & &  \\ 
 \hline
 &  &  & \textbullet &  \\ 
 \hline
   &  & \textbullet& &  \\ 
 \hline
   & \textbullet & &  & \\ 
 \hline
  \textbullet &  & & & \\ 
 \hline
 \end{tabular}
$$
$$
\bfy_2
$$

\columnbreak

$$
\begin{tabular}{ | m{0.25cm}| m{0.25cm}| m{0.25cm}| m{0.25cm}| m{0.25cm}| } 
 \hline
 &  &  & & \textbullet \\ 
 \hline
 &  &  & \textbullet &  \\ 
 \hline
 &  &  & &  \\ 
 \hline
   &  & \textbullet& &  \\ 
 \hline
   & \textbullet & &  & \\ 
 \hline
  \textbullet &  & & & \\ 
 \hline
 \end{tabular}
$$
$$
\bfy_3
$$

\columnbreak

$$
\begin{tabular}{ | m{0.25cm}| m{0.25cm}| m{0.25cm}| m{0.25cm}| m{0.25cm}| } 
 \hline
 &  &  & & \textbullet \\ 
 \hline
 &  &  & \textbullet &  \\ 
 \hline
   &  & \textbullet& &  \\ 
 \hline
 &  &  & &  \\ 
 \hline
   & \textbullet & &  & \\ 
 \hline
  \textbullet &  & & & \\ 
 \hline
 \end{tabular}
$$
$$
\bfy_4
$$

\columnbreak

$$
\begin{tabular}{ | m{0.25cm}| m{0.25cm}| m{0.25cm}| m{0.25cm}| m{0.25cm}| } 
 \hline
 &  &  & & \textbullet \\ 
 \hline
 &  &  & \textbullet &  \\ 
 \hline
   &  & \textbullet& &  \\ 
 \hline
   & \textbullet & &  & \\ 
 \hline
 &  &  & &  \\ 
 \hline
  \textbullet &  & & & \\ 
 \hline
 \end{tabular}
$$
$$
\bfy_5
$$

\columnbreak
$$
\begin{tabular}{ | m{0.25cm}| m{0.25cm}| m{0.25cm}| m{0.25cm}| m{0.25cm}| } 
 \hline
 &  &  & & \textbullet \\ 
 \hline
 &  &  & \textbullet &  \\ 
 \hline
   &  & \textbullet& &  \\ 
 \hline
   & \textbullet & &  & \\ 
 \hline
  \textbullet &  & & & \\ 
 \hline
  &  &  & &  \\ 
 \hline
 \end{tabular}
$$
$$
\bfy_6
$$

\end{multicols}
\end{center}

\vspace*{-0.5cm}
\begin{center}
\begin{multicols}{4}

$$
\begin{tabular}{ | m{0.25cm}| m{0.25cm}| m{0.25cm}| m{0.25cm}| m{0.25cm}| } 
 \hline
 &  &  & &  \\ 
 \hline
 &  &  & &  \\ 
 \hline
 &  &  & \textbullet &  \\ 
 \hline
   &  & \textbullet& & \textbullet \\ 
 \hline
   & \textbullet & &\textbullet  & \\ 
 \hline
  \textbullet &  &\textbullet & & \\ 
 \hline
 \end{tabular}
$$
$$
\bfz_3
$$

\columnbreak

$$
\begin{tabular}{ | m{0.25cm}| m{0.25cm}| m{0.25cm}| m{0.25cm}| m{0.25cm}| } 
 \hline
 &  &  & &  \\ 
 \hline
 &  &  & &  \\ 
 \hline
 &  &  &  & \textbullet \\ 
 \hline
   &  & \textbullet& &  \\ 
 \hline
   & \textbullet & &\textbullet  & \\ 
 \hline
  \textbullet &  &\textbullet & & \\ 
 \hline
 \end{tabular}
$$
$$
\bfz_4
$$

\columnbreak

$$
\begin{tabular}{ | m{0.25cm}| m{0.25cm}| m{0.25cm}| m{0.25cm}| m{0.25cm}| } 
 \hline
 &  &  & &  \\ 
 \hline
 &  &  & &  \\ 
 \hline
 &  &  &  & \textbullet \\ 
 \hline
   &  & &\textbullet &  \\ 
 \hline
   & \textbullet & &  & \\ 
 \hline
  \textbullet &  &\textbullet & & \\ 
 \hline
 \end{tabular}
$$
$$
\bfz_5
$$

\columnbreak

$$
\begin{tabular}{ | m{0.25cm}| m{0.25cm}| m{0.25cm}| m{0.25cm}| m{0.25cm}| } 
 \hline
 &  &  & &  \\ 
 \hline
 &  &  & &  \\ 
 \hline
 &  &  &  & \textbullet \\ 
 \hline
   &  & &\textbullet &  \\ 
 \hline
   &  & \textbullet&  & \\ 
 \hline
  \textbullet &  & & & \\ 
 \hline
 \end{tabular}
$$
$$
\bfz_6
$$

\end{multicols}
\end{center}
\end{example}

Consider the set of vertices  $V = [n+1]\times [n]$,
which 
corresponds to the variables of $A$.
For a subset $\mcA \subseteq V$, 
we denote by $w_\mcA = \prod_{(i,j)\in \mcA} w_{i,j}$ 
the associated square-free monomial.
Let $\Delta$ be the Stanley-Reisner simplicial complex defined by 
$\mfK$.
Recall that a face of $\Delta$ is a subset $\mcA\subseteq V$ such that $w_\mcA \notin\mfK$, and a facet is a face that is maximal with respect to inclusion.
Since  $\Delta$ has  low codimension,
instead of analyzing its faces, it is more convenient
to focus on their complements. 
We call the complement of a face  a {\bf c-face}, 
and  the complement of a facet a {\bf c-facet}.
It is easy to see that, for a subset  $\mathcal{C}\subseteq V$, we have
\begin{equation}\label{EqCharacterizationCFaces}
\mathcal{C}\subseteq V
\text{ is a c-face of } \Delta
\Leftrightarrow
\gcd(w_\mcC, \bfu) \ne 1
\text{ for every  generator }
\bfu \in \mfK,
\end{equation}
and, clearly, a c-facet is a c-face that is minimal with respect to inclusion.
Moreover, 
there is a bijection between the minimal primes of $\mfK$ and the 
 c-facets of $\Delta$, 
defined by mapping each (monomial) prime ideal to the  set of variables which generate it.
For readers familiar with Alexander duality,
we point out that c-faces of $\Delta$ correspond to  monomials in the Stanley-Reisner ideal of the Alexander dual $\Delta^\vee$, and  c-facets to  the generators of this ideal;
however, we will not  use this point of view.

A set of vertices of the form  $\{ (i,j)\, | \, i+j = p\}$, for some  $p$, 
is called an antidiagonal of $V$.
A c-facet  of $\Delta$ must satisfy 
several restrictions on the number and position of vertices on each antidiagonal;
we list them in the next lemma.

\begin{lemma}\label{LemmaPropertiesCFacets}
Let $\mcC$  be a c-facet of $\Delta$.
\begin{enumerate}
\item
Every $(i,j) \in \mcC$ lies on one of the four antidiagonals 
$i+j = p$ with $p = n+1, \ldots, n+4$.
\item
There is exactly one  $(i,j) \in \mcC $ on the antidiagonal $i+j= n+1$.
\item 
There are either one or two   $(i,j) \in \mcC $ on the antidiagonal 
$i+j= n+2$.
\item
There are either one or two   $(i,j) \in \mcC $ on the antidiagonal 
$i+j= n+3$.
\item
There is at most  one  $(i,j) \in \mcC $ on the antidiagonal  $i+j= n+4$.
\item
If $(1,n)\notin \mcC$, then  there are 
$(i_1, j_1), (k_1, l_1) \in \mcC$ such that
$i_1+j_1 = n+1, k_1 +l_1 = n+2$, and 
$k_1 > i_1$.
Every monomial \eqref{EqMonomialYh} contains at least one of $(i_1, j_1), (k_1, l_1)$.

\item
Let $(i_2,j_2) \in \mcC$ with  $i_2+j_2 = n+2$ and minimal $i_2$.
There is a vertex $(k_2,l_2) \in \mcC$ such that
$ k_2+l_2 = n+3$ and 
$k_2 > i_2$.
Every monomial \eqref{EqMonomialXh} contains at least one of 
$(i_2, j_2), (k_2, l_2)$.
\item 
Let 
$(i_3,j_3) \in \mcC$ with  
$i_3+j_3 = n+3$ and minimal $i_3$.
There  is a vertex $(k_3, l_3) \in \mcC$ such that either $k_3+l_3 = n+4$ and 
$k_3 > i_3$,
or $k_3+l_3 = n+2$ and 
$k_3 \geq  i_3$.
Every monomial \eqref{EqMonomialZh} contains at least one of 
$(i_3, j_3), (k_3, l_3)$.
\end{enumerate}

\end{lemma}
\begin{proof}
All statements will follow inspecting the generators 
\eqref{EqMonomialXh}, \eqref{EqMonomialYh}, \eqref{EqMonomialZh}
and using the fact that a c-facet is a minimal subset  $\mathcal{C}\subseteq V$ satisfying condition \eqref{EqCharacterizationCFaces}, that is, $\gcd(w_\mcC, \bfu) \ne 1$
 for every  generator 
$\bfu \in \mfK$.
Thus, it might be helpful to keep Example \ref{ExampleMonomialsN5} at hand while going through this proof.

Item (1) is immediate, since any vertex $(i,j)$ appearing in a generator
lies in one of the four antidiagonals $n+1 \leq i+j \leq n+4$.
Moreover, from the three  coprime monomials 
$$
\bfx_2= 
\prod_{i=3}^{n+1} w_{i,n+3-i}, \quad
\bfx_{n+1}= 
\prod_{i=2}^{n}w_{i,n+2-i}, \quad
\bfy_{n+1} = 
\prod_{i=1}^{n}w_{i,n+1-i}
$$
we see
 that $\mcC$ contains at least one vertex
$(i,j)$ for each antidiagonal $i+j = n+1,n+2, n+3$. 

Any generator containing a vertex  $(i,j)$ in  the antidiagonal $i+j = n+1$  is of the form \eqref{EqMonomialYh} and contains a top-right segment of the antidiagonal. 
It follows that if a subset $\mcD \subseteq V$ satisfies \eqref{EqCharacterizationCFaces} and has more than one vertex on the antidiagonal 
$i+j =n+1$,
removing  the lower vertex will preserve condition \eqref{EqCharacterizationCFaces}.
By minimality of c-facets, item (2)  must hold.

Any generator containing a vertex  $(i,j)$ in  the antidiagonal $i+j = n+2$  contains either a top-right segment of the antidiagonal, 
if it is of the form \eqref{EqMonomialXh}, or a bottom-left segment  of the antidiagonal,
if it is of the form  \eqref{EqMonomialYh} or \eqref{EqMonomialZh}.
If a subset $\mcD \subseteq V$ satisfies \eqref{EqCharacterizationCFaces} and has more than two vertices on the antidiagonal 
$i+j =n+2$,
removing an intermediate vertex  will preserve condition \eqref{EqCharacterizationCFaces}, and thus (3) holds.
Items (4) and (5) are proved analogously.

Now, assume $(1,n)\notin \mcC$.
By item (2), there is a vertex
 $(i_1,j_1) \in \mcC$ such that  $i_1+j_1 = n+1$,
and therefore  $2 \leq i_1 \leq n$.
Consider the generator  $\bfy_{i_1} = 
\prod_{i=1}^{i_1-1}w_{i,n+1-i}
\prod_{i=i_1+1}^{n+1} w_{i,n+2-i}.
$
By \eqref{EqCharacterizationCFaces}, $\mcC$ contains a vertex $(k_1,l_1)$
appearing in $\bfy_{i_1}$.
By item (2), this vertex  cannot appear in  $\prod_{i=1}^{i_1-1}w_{i,n+1-i}$,
so $k_1 >i_1$ and $k_1 + l_1 = n+2$.
It is straightforward to check that 
every monomial \eqref{EqMonomialYh} contains at least one of $(i_1, j_1), (k_1, l_1)$,
so item (6) is proved.

Let $(i_2,j_2) \in \mcC$ be such that  $i_2+j_2 = n+2$ and with the least possible $i_2$.
Considering $\bfx_{n+1}$ we may assume that $i_2 \leq n$,
and, since  $i_2+j_2 = n+2$, we also have $i_2 \geq 2$.
Item (7) is now proved in the same way as (6), using the monomial
$\bfx_{i_2}
=\prod_{i=2}^{i_2-1}w_{i,n+2-i}
\prod_{i=i_2+1}^{n+1} w_{i,n+3-i}
$.
Finally,
item (8) is also proved in the same way,
 using 
$\bfz_{i_3} = 
\prod_{i=3}^{i_3-1} w_{i,n+3-i}
\prod_{i=i_3}^{n+1} w_{i,n+2-i}
\prod_{i=i_3+1}^{n+1} w_{i,n+4-i}.
$
\end{proof}

\begin{prop}\label{PropositionDeltaCodimension4}
The simplicial complex $\Delta$ has codimension 4.
\end{prop}
\begin{proof}
 By Proposition \ref{PropositionLeadingMonomial}, we have  $\mfK \subseteq \iin(\mfI)$,
 and, combining with  Corollary \ref{CorollaryXIrreducibleDimension}, we obtain  
 $$
 \codim\, \Delta = \codim\, \mfK \leq \codim\, \iin(\mfI) = \codim\, \mfI = \codim\, \mfX= 4.
 $$
For the other inequality, 
it suffices to show that the cardinality of every c-facet  of $\Delta $ is  at least 4.
This follows from Lemma \ref{LemmaPropertiesCFacets} (2), (3), (4), (7), (8).
In facts, by items (2), (3), and (4), every c-facet contains at least 3 vertices.
Assume by contradiction that $\mathcal{C}$ is a c-facet with $|\mcC|=3$.
This implies that $\mathcal{C}$ contains exactly one element on each of the three antidiagonals $n+1 \leq i+j \leq n+3$, and no element on the antidiagonal $i+j = n+4$.
By item (7), there are $(i_2, j_2) , (k_2, l_2) \in \mcC$ such that $i_2+j_2 = n+2, k_2+l_2=n+3$, and $k_2> i_2$.
By item (8), there are $(i_3,j_3), (k_3, l_3) \in \mcC$ such that $i_3+j_3=n+3, k_3+l_3 = n+2$, and $k_3 \geq i_3$.
Note that we must have  $(k_2, l_2)=(i_3,j_3)$ and $(i_2, j_2)=(k_3, l_3)$, but this implies the contraditcion $i_3 = k_2 >i_2 = k_3 \geq i_3$.
\end{proof}

Our next goals are  to  prove that $\Delta$ is pure,
i.e., that all  facets have the same dimension, 
 and to determine the number of facets.
 Clearly, it suffices to prove the corresponding statements for the c-facets of $\Delta$.

\begin{prop}\label{PropEnumerationCFacetsLastColumn}
Let $\mcC$ be a c-facet of $\Delta$ such that $(i,n) \in \mcC$ for some $i$.
Then $|\mcC|=4$.
Moreover, the number of such c-facets is  $(n-1)^2n$.
\end{prop}

\begin{proof}
By Lemma \ref{LemmaPropertiesCFacets} (1), the possible values of $i$ such that 
$(i,n) \in \mcC$ are $i = 1, \ldots, 4$.
We prove the proposition by analyzing four different possibilities.
In each case, we consider the vertices $(i_1,j_1), (k_1,l_1), (i_2,j_2), (k_2,l_2), (i_3,j_3), (k_3, l_3) \in \mcC$ as in  the statements of Lemma \ref{LemmaPropertiesCFacets}
(note that these six vertices are not necessarily distinct).
For the purposes of condition \eqref{EqCharacterizationCFaces},
the vertices $ (i_2,j_2), (k_2,l_2)$ cover all generators \eqref{EqMonomialXh},
$(i_3,j_3), (k_3, l_3)$ cover all generators \eqref{EqMonomialZh},
and either $(1,n)$ or the vertices $(i_1,j_1), (k_1,l_1)$ cover all generators \eqref{EqMonomialYh}.

\underline{Case 1: $(1,n) \in \mcC$.}
Suppose $i_2 \geq i_3$, then we may assume that $(k_3, l_3) = (i_2, j_2)$.
It follows that $\mcC'=\{(1,n), (i_2, j_2),  (i_3, j_3), (k_2, l_2)\}\subseteq \mcC$ satisfies \eqref{EqCharacterizationCFaces}, 
so, by minimality, we deduce that $\mcC= \mcC'$.
The c-facets of this form correspond to triples $(i_2, i_3, k_2)$ such that 
$3 \leq i_3 \leq i_2 < k_2  \leq n+1$, and therefore their number is  
$
{n-1 \choose 3} +{n-1 \choose 2}.
$

Now, suppose $i_2 < i_3$, then we may assume that $(k_2, l_2) = (i_3, j_3)$.
As before,
  we conclude that
$\mcC = \{(1,n), (i_2, j_2), (i_3, j_3), (k_3, l_3) \}$. 
There are two classes of c-facets of this form: 
those with $k_3 + l_3 = n+2$ correspond to triples $(i_2, i_3, k_3)$ such that
$2 \leq i_2 < i_3 \leq k_3\leq n+1 $, 
while those with $k_3 + l_2 = n+4$ correspond to triples $(i_2, i_3, k_3)$ such that
$2 \leq i_2 < i_3 < k_3\leq n+1 $.
The number of these c-facets in the two classes is respectively   
$
{n \choose 3} +{n \choose 2}
$
and 
$
{n \choose 3}.
$

\underline{Case 2: $(1,n) \notin \mcC, (2,n) \in \mcC$.}
Observe that $i_1 \geq 2$ and $i_2 = 2$.
Since $i_2 < i_3$, we may assume $(k_2, l_2) = (i_3, j_3)$.
We claim that  $k_1 \geq i_3$:
if $k_1 < i_3 $, then   $(k_1,l_1), (i_3, j_3)$ cover  all generators \eqref{EqMonomialXh},
so we may remove  $(2,n)$ from $\mcC$ and still obtain a c-face by \eqref{EqCharacterizationCFaces}, contradiction.
Since   $k_1 \geq i_3$,  we may assume $(k_3, l_3) = (k_1, l_1)$,
and we conclude $\mcC = \{(2,n),(i_1,j_1),  (i_3, j_3), (k_1, l_1) \}$.
The c-facets of this form correspond to triples $(i_1, i_3, k_1)$ such that 
$2 \leq i_1 < k_1 \leq n+1$ and $3 \leq i_3 \leq k_1 \leq n+1$. 
Distinguishing whether $i_1 < i_3$, $i_1 = i_3$, or $i_1 > i_3$ we see that the number of c-facets is  
$
\left({n \choose 3} + {n \choose 2} \right) + {n-1 \choose 2} + {n-1 \choose 3}
$.

\underline{Case 3: $(1,n), (2,n) \notin \mcC, (3,n) \in \mcC$.}
Observe that $i_1 \geq 2, i_2 \geq 3, i_3 = 3$.
We claim that  $i_1 < i_2$.
If $i_1 \geq i_2$ and   $k_1 \geq k_2$, then $(k_1,l_1), (k_2,k_2)$ cover all 
 generators \eqref{EqMonomialZh},
 so we may remove $(3,n)$ from $\mcC$ and still obtain a c-face, contradiction.
 If $i_1 \geq i_2$ and   $k_1 < k_2$, then 
$(3,n), (k_1,l_1)$ cover all  generators \eqref{EqMonomialZh}, while 
$(k_1, l_1), (k_2, l_2)$ cover all  generators \eqref{EqMonomialXh},
so we may remove $(i_2,j_2)$ from $\mcC$ and still obtain a c-face, contradiction.
Thus, we have  $i_1 < i_2 $, and therefore  we can assume that $(k_1, l_1) = (k_3, l_3) =(i_2, j_2)$ and 
conclude that $\mcC = \{ (3,n),(i_1, j_1), (i_2, j_2), (k_2,l_2)\}$.
The c-facets of this form correspond to triples $(i_1, i_2, k_2)$ such that 
$2 \leq i_1 < i_2 <  k_2\leq n+1$,  and their number is  
$
{n \choose 3}.
$

\underline{Case 4: $(1,n), (2,n), (3,n) \notin \mcC, (4,n) \in \mcC$.}
The vertex $(4,n) $ appears only in $\bfz_3$. 
By minimality,  $\mcC$ does not contain any other vertex appearing in 
$\bfz_3$,
 otherwise we could remove $(4,n)$ from $\mcC$ and still obtain a c-face.
Comparing with $\bfx_{n+1}$, we see that  $\mcC$ must contain $(2,n)$, contradiction.
Thus, there are no c-facets in this case.

In conclusion, 
we have proved that $|\mcC|=4$ in all cases, and  the total number of c-facets is
$$
{n-1 \choose 3} +{n-1 \choose 2}+
{n \choose 3} +{n \choose 2}+
{n \choose 3}+
{n \choose 3} + {n \choose 2} + {n-1 \choose 3} +{n-1 \choose 2}+
{n \choose 3} 
= (n-1)^2n. \qedhere
$$ 
\end{proof}

Since $\Delta$ depends only on the  parameter $n$,
it  lends itself   to inductive arguments.
In the rest of the section, 
we use the superscript ${(n)}$ to emphasize the dependence on $n$.
For example, we denote 
the simplicial complex by $\Delta^{(n)}$,
the Stanley-Reisner ideal by $\mfK^{(n)}$, and 
the vertex set by $V^{(n)}$.

There is a map  $\sd: V^{(n-1)} \rightarrow V^{(n)}$  defined by
$\sd(i,j) = (i+1,j)$,
in other words, by  shifting each entry in the matrix one step down.
It  induces maps for subsets of vertices  and for monomials, 
which we also denote  by $\sd$.

\begin{lemma}\label{LemmaShiftingGenerators}
For each generator  $\bfu\in \mfK^{(n)}$,
there exists a  generator $\bfv\in\mfK^{(n-1)}$ such that 
$\sd(\bfv)$  divides $\bfu$.
For each generator $\bfv\in\mfK^{(n-1)}$,
there exists a  generator $\bfu\in \mfK^{(n)}$
and a variable $w_{i,n}$ such that 
 $\bfu = \sd(\bfv) \cdot w_{i,n}$.
\end{lemma}

\begin{proof}
The lemma follows by checking the equations
\begin{align*}
\bfx_{2}^{(n)} & = \sd\big(\bfx_{2}^{(n-1)}\big)\cdot w_{3,n},
&
\bfx_{h+1}^{(n)} &= \sd\big(\bfx_h^{(n-1)}\big)\cdot w_{2,n}
\quad\text{ for all} \quad h = 2, \ldots, n,\\
\bfy_{2}^{(n)} &= \sd\big(\bfx_{n}^{(n-1)}\big)\cdot w_{1,n}w_{n+1,1},
&
\bfy_{h+1}^{(n)} &= \sd\big(\bfy_h^{(n-1)}\big)\cdot w_{1,n}
\quad\text{ for all}\quad h = 2, \ldots, n,\\
\bfz_{3}^{(n)} &= \sd\big(\bfz_3^{(n-1)}\big)\cdot w_{3,n-1}w_{4,n},
&
\bfz_{h+1}^{(n)} &= \sd\big(\bfz_h^{(n-1)}\big)\cdot w_{3,n} 
\quad\text{ for all} \quad h = 3, \ldots, n. \qedhere
\end{align*}
\end{proof}

\begin{cor}\label{CorollaryBijectionDLRectangle}
There is an inclusion-preserving bijection
$$
\sd: \Big\{ \text{c-faces of } \Delta^{(n-1)}\Big\}
\rightarrow
\Big\{ \text{c-faces } \mcC \text{ of } \Delta^{(n)} \, \big| \, \mcC
\subseteq \{2, \ldots, n+1\} \times [n-1] \Big\}.
$$
\end{cor}
\begin{proof}
The fact that $\sd$ sends c-faces of $\Delta^{(n-1)}$ to c-faces of $\Delta^{(n)}$
follows from 
\eqref{EqCharacterizationCFaces} and the first statement of Lemma \ref{LemmaShiftingGenerators}.
Clearly, the map is injective and  preserves inclusions.
Surjectivity follows from 
\eqref{EqCharacterizationCFaces} and the second statement of Lemma \ref{LemmaShiftingGenerators}.
\end{proof}

\begin{thm}\label{TheoremSimplicialComplex}
The  complex $\Delta$ is pure of codimension 4, 
and it has $ \frac{(n-1)n(n+1)(3n-2)}{12}$ facets.
\end{thm}
\begin{proof}
The statement is easy to check for $n=2$, so we assume $n \geq 3$.
By Lemma \ref{LemmaPropertiesCFacets} (1),
for every c-facet $\mcC$ and every $(i,j)\in \mcC$ we have either $j=n$ or $i \geq 2$.
Therefore,
 we may partition the set of c-facets of $\Delta^{(n)}$ in two subsets:
\begin{align*}
C_1^{(n)} &= 
\Big\{ \text{c-facets } \mcC \text{ of } \Delta^{(n)} \, \big| \, \mcC
\subseteq \{2, \ldots, n+1\} \times [n-1] \Big\},
\\
C_2^{(n)} & = \Big\{ \text{c-facets } \mcC \text{ of } \Delta^{(n)} \, \big| \, 
(i,n) \in 
\mcC
\text{ for some } i
\Big\}.
\end{align*}
By Proposition \ref{PropEnumerationCFacetsLastColumn}, all c-facets in 
$C_2^{(n)}$ have size 4. 
The same is true for $C_1^{(n)}$ by induction and Corollary \ref{CorollaryBijectionDLRectangle}.
Thus, $\Delta^{(n)}$ is pure of codimension 4.

Denote the cardinalities by  $c_i^{(n)} = |C_i^{(n)}|$ for $i = 1,2,$
and let $c^{(n)}= \frac{(n-1)n(n+1)(3n-2)}{12}$.
We need to show that $c_1^{(n)} +c_2^{(n)} = c^{(n)}$.
By Proposition \ref{PropEnumerationCFacetsLastColumn}, we have
$c_2^{(n)}=(n-1)^2n$.
By Corollary \ref{CorollaryBijectionDLRectangle} and induction, we have 
$c_1^{(n)}=c^{(n-1)} $.
In conclusion,  the number of c-facets of $\Delta^{(n)}$ is 
\begin{align*}
c_1^{(n)} +c_2^{(n)}  =
\frac{(n-2)(n-1)n(3n-5)}{12}+ (n-1)^2n = c^{(n)}. & \qedhere
\end{align*}
\end{proof}

We are now ready to prove the main result of this section.

\begin{proof}[Proof of Theorem \ref{TheoremIPrime}]
The statement can be verified directly for $n = 2,3$,
for instance using \cite{Macaulay2},
so assume $n \geq 4$.
Since $\V(\sqrt{\mfI})$ is irreducible  by Corollary \ref{CorollaryXIrreducibleDimension},
it suffices to prove that $\mfI = \sqrt{\mfI}$.
Since $\mfK \subseteq \iin(\mfI) \subseteq \iin(\sqrt{\mfI})$
by Proposition \ref{PropositionLeadingMonomial},
 it suffices to prove that $\mfK = \iin(\sqrt{\mfI})$.
In fact,  this  forces the equality $\iin(\mfI) = \iin(\sqrt{\mfI})$, hence, the  Hilbert functions of $\mfI$ and $\sqrt{\mfI}$ must coincide,
 and  the inclusion $\mfI \subseteq \sqrt{\mfI}$ must be an equality.

Consider the inclusion  of monomial ideals $\mfK \subseteq \iin(\sqrt{\mfI})$.
By Theorems \ref{TheoremVarietyMatrices} and \ref{TheoremSimplicialComplex}, 
the two ideals have the same codimension and multiplicity.
It follows from the associativity formula for multiplicities
\cite[Exercise 12.11.e]{Eisenbud} that 
they have the same  set of minimal primes of maximal dimension.
Moreover, $\mfK$ is unmixed, since $\Delta$ is pure,
so every associated prime of $\mfK$ is  an associated prime of $\iin(\sqrt{\mfI})$.
Hence the inclusion  $\mfK \subseteq \iin(\sqrt{\mfI})$ is an equality locally at every associated prime of the smaller ideal $\mfK$, 
and this implies that $\mfK = \iin(\sqrt{\mfI})$.
\end{proof}

On the way to proving  Theorem \ref{TheoremIPrime}, 
we have determined an explicit 
Gr\"obner basis of $\mfI$.

\begin{cor}\label{CorollaryGrobnerBasis}
The ideal $\mfI$ has the following square-free Gr\"obner basis 
$$
\Big\{f_1, \ldots, f_{n+1}, 
F_2, \ldots, F_{n+1}\Big\}\cup
\Big\{f_1 \det \bfW_{(h,\ldots, n),(2,\ldots, n-h+3)}
+
f_2  
    \det \bfW_{(h,\ldots, n),(1,3, \ldots, n-h+3)}\Big\}_{h = 4}^{n+1}
$$
with respect to  the reverse lexicographic order on the opposite ordering of the variables.
\end{cor}
\begin{proof}
It follows from Proposition \ref{PropositionLeadingMonomial} and the proof of Theorem \ref{TheoremSimplicialComplex}.
\end{proof}

Combining with the main result of
\cite{CV},
we obtain another interesting byproduct.

\begin{cor}\label{CorollaryCohenMacaulayComplex}
The  simplicial complex $\Delta$ is Cohen-Macaulay in characteristic 0.
\end{cor}
\begin{proof}
The  ring $A/\mfI$ 
is Cohen-Macaulay
by Theorem \ref{TheoremVarietyMatrices}  and Lemma \ref{LemmaPropertiesRationalSingularities} (5).
The conclusion 
 follows from Corollary  \ref{CorollaryGrobnerBasis} and  \cite[Corollary 2.7]{CV}.
\end{proof}

To conclude this section, we  discuss some related  questions and potential future directions.
\smallskip

We have introduced a new simplicial complex $\Delta$, 
which  is a pure codimension 2 sub-complex of the classical antidiagonal complex of the determinantal variety.
It might be interesting to study the combinatorics and topology of $\Delta$  more in detail.
In particular, we ask the following:

\begin{question}
Is the simplicial complex $\Delta$ Cohen-Macaulay in all characteristics? 
Is it shellable, or vertex-decomposable?
\end{question}

The  variety of matrices $\mfX$ is a new addition to the large family of smooth or mildly singular varieties with square-free initial ideals. 
Given the close connection of $\mfX$ to nilpotent orbit closures and rank varieties established in  Section \ref{SectionGeometricTechnique},
it seems natural to search for square-free Gr\"obner degenerations among those varieties.
Following the notation in \cite[Section 2]{Weyman}, we ask:

\begin{question}
For which integer partitions $\bfv$ do the nilpotent orbit closure or rank variety 
$X_\bfv$ admit a square-free Gr\"obner degeneration?
\end{question}

A notable aspect  of the varieties  $X_\bfv $ is that, 
unlike our $\mfX$, 
they are always Gorenstein \cite[Theorem 1]{ES}.
The Gorenstein property has interesting topological implications for simplicial complexes, so it might be worth  investigating the existence of Gorenstein initial
complexes.
See  \cite{CHT} for related considerations.

\section{Primeness, flatness, and proof of the main theorem}\label{SectionProofMainTheorem}

Assume $\ch(\kk)=0$ throughout this section.
Recall the polynomial rings 
$
A= \kk[w_{i,j}] $ and $
B = A \otimes_\kk \kk[v_1, v_2, v_3, v_4]
$
 from Definition \ref{DefinitionRingsAB},
and the ideals $\mfI \subseteq A$ and $\mfL \subseteq B$ from \eqref{EqDefinitionIdealI} and \eqref{EqDefinitionIdealL}.
Define the intermediate polynomial rings
$B^{(j)} = A [v_{j+1},\ldots, v_4]$ for $j = 0, \ldots, 4$,
so $B^{(0)} = B$ and $B^{(4)} = A$.
The  inclusion $ B^{(j+1)} \subseteq  B^{(j)} $ and projection $\pi^{(j)} : B^{(j)}  \twoheadrightarrow B^{(j+1)} = B^{(j)}/(v_{j+1}) $ define an algebra retraction for each $j = 0, \ldots, 3$.

We  consider the $\Z$-grading  induced on each $B^{(j)}$ from the grading  $\deg_2(\cdot)$ of Definition \ref{DefinitionRingT}.
Then, $B^{(0)}=B$ is non-negatively graded, $B^{(1)} , B^{(2)} , B^{(3)} $ are positively graded, and $B^{(4)}=A$ is standard graded.
Extend the monomial order \eqref{EqTermOrderA} from $A$ to $B^{(j)}$  by equipping $B$ and each subring $B^{(j)}$ with the graded reverse lexicographic order on the  ordering of variables
$$
v_1 < v_2 < v_3 < v_4 < 
w_{1,1} < w_{1,2} < w_{1,3} <\cdots < w_{n+1,n-1} < w_{n+1,n}. 
$$
The next  property is a basic consequence of the revlex order.

\begin{lemma}\label{LemmaInclusionConsequenceRevlex}
Let $\mfb \subseteq B^{(j)}$ be a homogeneous ideal and 
$\mfa =  \pi^{(j)}(\mfb) \subseteq B^{(j+1)}$ its image.
There is an inclusion of monomial ideals $\iin(\mfa)B^{(j)} \subseteq \iin(\mfb)$.
\end{lemma}

\begin{proof}
Pick a monomial $\bfu \in \iin(\mfa)$
and let $f\in \mfa$ be a homogeneous polynomial with 
$\LM(f) = \bfu$.
Let $g\in \mfb $ be homogeneous with  $\pi^{(j)}(g) = f$,
then 
$
g = f+h
$
where $h\in (v_{j+1})$
is homogeneous of the same degree as $g$ and $f$.
Since no term of $f$ is divisible by $v_{j+1}$,
it follows
by revlex order
 that  $\LM(g)=\LM(f)$, and this implies the desired inclusion.
\end{proof}

\begin{prop}\label{PropositionCodimension4IntermediateIdeals}
Let $\mfL^{(j)}= \frac{\mfL+(v_1, \ldots, v_j)}{(v_1, \ldots, v_j)}\subseteq B^{(j)}$, the image of $\mfL$ in $B^{(j)}$.
Then  $\mfL^{(j)}$ has codimension 4.
\end{prop}
\begin{proof}
By Krull's principal ideal theorem,
we have 
$$
\dim(B^{(j)}/\mfL^{(j)}) - 1 \leq \dim(B^{(j)}/(\mfL^{(j)}+v_{j+1})) = \dim(B^{(j+1)}/\mfL^{(j+1)}),
$$
hence 
$
\codim(\mfL^{(j)}) \geq \codim(\mfL^{(j+1)})
$.
The  conclusion follows, since
 $\codim \, \mfL^{(0)} = \codim\, \mfL = 4$ by 
 Theorems \ref{TheoremIrreducibleNested} and \ref{TheoremNeighborhoodL},
 and $\codim\, \mfL^{(4)} = \codim\, \mfI = 4$ by Corollary \ref{CorollaryXIrreducibleDimension}.
\end{proof}

Recall the monomial ideal $\mfK\subseteq A$ 
defined in \eqref{EqDefinitionMonomialIdealK}.

\begin{prop}\label{PropositionIntermediateLiPrime}
The ideal $\mfL^{(j)}$ is prime 
with  $\iin(\mfL^{(j)}) = \mfK B^{(j)}$
for each  $j \in \{1,2,3,4\}$.
\end{prop}
\begin{proof}
We prove the proposition by reverse induction on $j$.
For the base case $j=4$, we have $\mfL^{(4)}= \mfI$, so it follows from Theorem \ref{TheoremIPrime} and Corollary \ref{CorollaryGrobnerBasis}.

Let $1\leq j \leq 3 $.
For simplicity, 
rename $v= v_{j+1}$, 
$\mfa=\mfL^{(j+1)}$,
$\mfb= \mfL^{(j)}$,
$\pi = \pi^{(j)}$.
Assuming that $\mfa$ is prime with 
$\iin(\mfa) = \mfK B^{(j+1)}$,
 we need to show the same for $\mfb$.
 By Proposition \ref{PropositionCodimension4IntermediateIdeals},
both  $\mfb \subseteq B^{(j)}$ and $\mfa= \frac{\mfb + (v)}{(v)} \subseteq B^{(j+1)}$ 
have codimension 4.
We deduce that   $v$ 
 does not belong to any prime $\mfp$ of codimension 4 containing $\mfb$,
otherwise  $\frac{\mathfrak{\mfp}+v}{v}$ would be an ideal of codimension 3 containing $\mfa$.
Fix one such prime $\mathfrak{p} \subseteq B^{(j)}$.
Then $\bfa$ is prime, $\bfa \subseteq \pi(\mfp)= \frac{\mfp + (v)}{(v)}$,
and both ideals have codimension 4,
so   we must have $\bfa = \pi(\mfp)$.

Since $v \notin \mfp$,
the variable  $v$ is  regular  on  $\frac{B^{(j)}}{\mfp}$. 
Since the ideal $\iin(\mfa)B^{(j)}$ is extended from $B^{(j+1)}$,
the variable  $v$ is also regular on
$\frac{B^{(j)}}{\iin(\mfa)B^{(j)}}$. 
Thus, we have graded 
 short exact sequences
$$
0 \rightarrow \frac{B^{(j)}}{\mfp}\big(-\deg(v)\big) \xrightarrow{\,\cdot v\,}  \frac{B^{(j)}}{\mfp}
\rightarrow  \frac{B^{(j+1)}}{\mfa}\rightarrow 0
$$
and
$$
0 \rightarrow \frac{B^{(j)}}{\iin(\mfa)B^{(j)}}\big(-\deg(v)\big) \xrightarrow{\,\cdot v\,}  \frac{B^{(j)}}{\iin(\mfa)B^{(j)}}
\rightarrow  \frac{B^{(j+1)}}{\iin(\mfa)}\rightarrow 0,
$$
from which we obtain the equations of Hilbert series
$$
(1-t^{\deg(v)})\mathrm{HS}\big(B^{(j)}/\mfp\big) 
= \mathrm{HS}\big({B^{(j+1)}}/{\mfa}\big) 
= \mathrm{HS}\big({B^{(j+1)}}/{\iin(\mfa})\big) 
=
(1-t^{\deg(v)})\mathrm{HS}\big(B^{(j)}/\iin(\mfa) B^{(j)}\big). 
$$
Since  $\deg(v)>0$,
we deduce that
$
\mathrm{HS}\big(B^{(j)}/\mfp\big) 
=
\mathrm{HS}\big(B^{(j)}/\iin(\mfa) B^{(j)}\big) 
$.
By Lemma \ref{LemmaInclusionConsequenceRevlex},
we have  $\iin(\mfa)B^{(j)}\subseteq \iin(\mfb)\subseteq \iin(\mfp)$.
Observe that  $B^{(j)}$ is positively graded, hence,
 every graded component of any ideal has finite dimension.
In conclusion, the equality of Hilbert series forces the inclusions to be equalities, 
that is, 
$\iin(\mfa)B^{(j)}= \iin(\mfb)=  \iin(\mfp)$
and  $\mfb=  \mfp$.
\end{proof}

The argument  does not work for the last step $j=0$, since $\deg(v_1)=0$ and $B$ is not positively graded.
However, we can prove primeness after localizing at the ideal of the origin, by exploiting  Theorem \ref{TheoremIrreducibleNested}.
Denote by $\mfB = (w_{i,j}, v_h)$ the irrelevant maximal ideal of $ B $. 

\begin{prop}\label{PropositionLocalizedLPrime}
The localization $\mfL_\mfB 	\subseteq B_\mfB$ is prime.
\end{prop}

\begin{proof}
By Theorems \ref{TheoremIrreducibleNested} and \ref{TheoremNeighborhoodL},
the subscheme
$\V(\mfL)$ is irreducible and generically reduced.
Localizing at   $\mfB$, 
this implies that $\mfL_\mfB = \mfP_1 \cap \mfQ$,
where $\mfP_1= \sqrt{\mfL_\mfB}$ is prime of codimension 4
and  $\mfQ\subseteq \mfB$ is a (possibly redundant) ideal with codimension at least 5 
such that  $\mfP_1  \subseteq \sqrt{\mfQ}$.

Since  $\mfL^{(1)}_\mfB =  \frac{\mfL_\mfB + (v_1)}{
(v_1)}$ 
has codimension 4 by 
Proposition \ref{PropositionCodimension4IntermediateIdeals},
we deduce that $v_1 \notin \mfP_1$.
Denote  $\pi = \pi^{(0)}_\mfB : B_\mfB \twoheadrightarrow B^{(1)}_\mfB$ and consider the prime
$\mfP_2 = \pi^{-1}(\mfL_\mfB^{(1)})$.
Clearly, 
$\mfL_\mfB+(v_1) \subseteq \mfP_2$ and
$\frac{\mfP_2+(v_1)}{(v_1)}= \pi(\mfP_2) =  \frac{\mfL_\mfB+(v_1)}{(v_1)}$,  so $\mfP_2  = \mfL_\mfB+(v_1)$.
Observe also that  $\mfP_1 = \sqrt{\mfL_\mfB}\subseteq \sqrt{\mfP_2}=\mfP_2$.

Pick any element $f_1 \in \mfP_1$, and write $f_1 = g_1 + f_2 v_1$
with $g_1 \in \mfL_\mfB$.
Thus,  $f_2v_1 \in \mfP_1$ and, therefore, $f_2 \in \mfP_1$.
Repeating this step for $f_2$ etc., 
we obtain  $\mfP_1\subseteq \mfL_\mfB  + (v_1^s)$ for all $s >0$.
Applying Krull's intersection theorem in the local ring  $(B/\mfL)_\mfB$,
we conclude that $\mfP_1=\mfL_\mfB$, as desired.
\end{proof}

\begin{cor}\label{CorollaryRegularSequence}
The sequence $v_1, \ldots, v_4$ is regular  on $(B/\mfL)_\mfB$.
\end{cor}
\begin{proof}
It follows immediately from 
Propositions \ref{PropositionIntermediateLiPrime} and \ref{PropositionLocalizedLPrime}.
\end{proof}

\begin{cor}\label{CorollaryFlatRingMap}
The ring map
$\kk[v_1,\ldots, v_4]_{(v_1, \ldots, v_4)} \rightarrow (B/\mfL)_\mfB$ is flat.
\end{cor}
\begin{proof}
It follows  by a version of the  local criterion for flatness 
\cite[Lemma 10.128.2]{stacks}.
\end{proof}

Now, we can finally combine all the results of the paper and prove our main theorem.

\begin{proof}[Proof of Theorem \ref{MainTheorem}]
By Corollary \ref{CorollaryReductionPlane} (2) and Theorem \ref{TheoremIrreducibleNested}, 
$\Hilb^{(m,2)}(S)$ is nonsingular in codimension 3.
By Corollary \ref{CorollaryReductionPlane} (1),
it remains to show that 
$\Hilb^{(m,2)}(\AA^2)$ has rational singularities.
By Theorem \ref{TheoremReductionToCompressedPair},
it suffices to show that $\Hilb^{(m,2)}(\AA^2)$ has a rational singularity at the $n$-th compressed pair, for $m = {n+1 \choose 2}$.
By Theorem  \ref{TheoremNeighborhoodL},
this amounts to showing that the 
affine scheme $\V(\mfL) \subseteq \Spec(B)$ has a rational singularity at the origin $\V(\mfB)$.

Consider the morphism $\Phi : \V(\mfL) \rightarrow \AA^4$ defined by the map of rings $\kk[v_1,\ldots, v_4] \rightarrow B/\mfL$.
By    Corollary \ref{CorollaryFlatRingMap}  and
openness of the flat locus 
\cite[Theorem 37.15.1]{stacks},
there exists a principal open set $U \subseteq \V(\mfL) $ containing the origin such that the restriction $\Phi: U \rightarrow \AA^4$ is flat. 
The fiber of $\Phi$ over $\mathbf{0}\in \AA^4$ is the scheme
$\V(\mfI) \subseteq \Spec(A)$, which has rational singularities
by Theorems  \ref{TheoremVarietyMatrices} and  \ref{TheoremIPrime}.
By Lemma \ref{LemmaPropertiesRationalSingularities} (2), we obtain the desired statement.
\end{proof}

As  byproducts, we observe that  Proposition \ref{PropositionLocalizedLPrime} and Corollary \ref{CorollaryFlatRingMap} hold more generally.

\begin{cor}
The ideal $\mfL\subseteq B$ is prime.
\end{cor}
\begin{proof}
It follows from Theorems \ref{MainTheorem} and \ref{TheoremNeighborhoodL}.
\end{proof}

\begin{cor}\label{CorollaryFlatMorphism}
The natural morphism $\Hilb^{(m,2)}(\AA^2) \rightarrow \Hilb^2(\AA^2)$ is flat.
\end{cor}
\begin{proof}
It follows by Miracle Flatness \cite[Lemma 10.128.1]{stacks},
since $\Hilb^2(\AA^2)$ is smooth,
$\Hilb^{(m,2)}(\AA^2)$ is Cohen-Macaulay by Theorem \ref{MainTheorem} and Lemma \ref{LemmaPropertiesRationalSingularities} (5),
and all fibers have the expected dimension by Corollary
\ref{CorollaryDimensionFibers}.
\end{proof}

We point out that the other natural morphism 
$\Hilb^{(m,2)}(\AA^2) \rightarrow \Hilb^m(\AA^2)$ is not flat.
For example, for $m=3$ the fiber over $[\V(\mm^2)]$ has dimension 1,  whereas the fiber over a general $[Z_1]$ is a finite scheme.
This fact contrasts with the case of  $\Hilb^{(m,1)}(\AA^2)$
\cite{Fogarty2,Song},
and is  the main reason why  the 
the study of  $\Hilb^{(m,2)}(\AA^2)$ is much more complicated than $\Hilb^{(m,1)}(\AA^2)$.

\section{Conclusions and open problems}\label{SectionConclusions}

In this paper we  studied the singularities of nested Hilbert schemes $\Hilb^\lambda(\AA^2)$,
where $\lambda$ is an integer partition.
In particular, 
we described the geometry of $\Hilb^{(m,2)}(\AA^2)$
and  proved that it has rational singularities.
Our methods may be used to tackle the problem of singularities for 
other classes of nested Hilbert schemes of $\AA^2$.
In this final section, we discuss some potential future directions, and 
collect the open problems suggested by our work.

\subsection{$F$-singularities}
The analogue of rational singularities for schemes over a field $\kk$ of positive characteristic is the notion of $F$-rational singularities.
It is natural to ask whether $\Hilb^\lambda(\AA^2)$ is $F$-rational,  when $\ch(\kk)>0$, 
for those partitions $\lambda$ such that $\Hilb^\lambda(\AA^2)$ is known to have rational singularities when $\ch(\kk)=0$.
We ask whether the  characteristic $p$ version of  Theorem \ref{MainTheorem} holds:

\begin{question}
Assume $\ch(\kk)>0$. 
Is the nested Hilbert scheme  $\Hilb^{(m,2)}(\AA^2)$  $F$-rational?
\end{question}

The proof of Theorem \ref{MainTheorem}
is mostly characteristic-free;
the assumption  $\ch(\kk)=0$ is required for the analysis of the variety $V(\mfI)$ in Section \ref{SectionGeometricTechnique}, 
where we used the representation theory of the general linear group.
Following our method, we can settle the $F$-rationality  problem for  $\Hilb^{(m,1)}(\AA^2)$, 
for which  the characteristic 0 version was proved in \cite[Theorem 1.1]{Song}
(the argument also gives a new proof of rational singularities in characteristic 0).

\begin{prop}\label{PropositionFRational}
Assume $\ch(\kk)>0$.
The   nested Hilbert scheme  $\Hilb^{(m,1)}(\AA^2)$ is $F$-rational.
\end{prop}
\begin{proof}[Proof sketch]
Adapting the analysis of Section 5 to the compressed pair $[\V(\mm^n)\supseteq \V(\mm)]$ one obtains polynomials 
$\Gamma_1 = x+v_1 z, \Gamma_2 = y +v_2z$,
with grading  $\deg_2(v_1)=\deg_2(v_2)=1$.
The  ideals produced by the division algorithm are 
$\mfL = (G_1, \ldots, G_{n+1})\subseteq B = A \otimes_\kk \kk[v_1,v_2]$
and 
$\mfI = (F_1, \ldots, F_{n+1})\subseteq A$.
Thus, $\mfI = I_n(\bfW)$ is the (prime)  ideal of a determinantal variety, which is $F$-rational 
\cite[Example 8.12]{MP}.
The theorem now follows as in Section \ref{SectionProofMainTheorem},
since the analogue of Lemma \ref{LemmaPropertiesRationalSingularities} holds for $F$-rationality,
see \cite[Theorem 5.1, Proposition 6.4, Theorem 6.16]{MP}.
\end{proof}

We point out that there are two further known classes of nested Hilbert schemes with rational singularities, namely $\Hilb^{(m+2,m+1,m)}(\AA^2)$ and $\Hilb^{(m+1,m,1)}(\AA^2)$, see \cite[Corollary 5.5, Corollary 5.7]{RT}.

The proof of Proposition \ref{PropositionFRational}  suggests yet another question about $F$-singularities:
determinantal varieties are strongly $F$-regular,
 a stronger condition than $F$-rational.
Thus, we ask:

\begin{question}
Assume $\ch(\kk)>0$. 
Is  $\Hilb^{(m,1)}(\AA^2)$ strongly $F$-regular?
\end{question}

The answer to this question does not follow as in Proposition \ref{PropositionFRational},
since
strong $F$-regularity does not deform \cite[Example 8.9]{MP}.

\subsection{Two-step nested Hilbert schemes}
The most natural  questions concern  the case of arbitrary two-step nested Hilbert schemes
$\Hilb^{(m_1,m_2)}(\AA^2)$.
We observed in Section 3 that their irreducibility  is an open problem: 
\begin{question}
Is the nested Hilbert scheme $\Hilb^{(m_1,m_2)}(\AA^2)$ irreducible for every $m_1>m_2$?
\end{question}

This problem is equivalent to asking whether any arbitrary pair $Z_1 \supseteq Z_2$ of finite subschemes of $\AA^2$ is simultaneously smoothable.
More generally, nothing is known about their singularities, e.g. whether they are reduced, normal, Cohen-Macaulay, or rational.
 
\begin{question}
What are the singularities of   $\Hilb^{(m_1,m_2)}(\AA^2)$?
\end{question}

We point out that our method reduces this question to the study of a fairly concrete object in commutative algebra. 
In fact, applying Section \ref{SectionReductionCompressed}, it suffices to consider 
$\Hilb^{(m_1,m_2)}(\AA^2)$, where $m_1 = {n+1 \choose 2}$, $m_2 = {n \choose 2}$, 
around the compressed pair 
$[\V(\mm^n)\supseteq \V(\mm^{n-1})]$ consisting of two consecutive fat points.
Section \ref{SectionLocalEquations} then produces an explicit ideal $\mfL$, which encodes the containment of two determinantal ideals  associated to generic $(n+1)\times n$ and $n\times(n-1)$ matrices.

\subsection{Square-free initial ideals of varieties of matrices}
A crucial ingredient in our proof of Theorem \ref{MainTheorem}
was the existence of a square-free Gr\"obner degeneration for the variety of matrices $\mfX$.
As discussed in Section \ref{SectionGeometricTechnique},
$\mfX$ is  related to the rank variety $\mathfrak{Y} = X_{(1,1)}$.
More generally, for a partition $\bfv=(v_1, v_2, \ldots)$ of an integer $\ell \leq n$,
we have a rank variety
$$
X_\bfv = \big\{ \bfA\in \Mat(n,n) \, : \, \dim \ker(\bfA^i) \geq v_1 + \cdots + v_i\text{ for all } i\big\}.
$$
These varieties,
introduced in \cite{ES}, are natural generalizations of the well-known nilpotent orbit closures (which correspond to the case $\ell = n$).
Like the variety $\mfX$, they are  irreducible with rational singularities; in addition, they are also Gorenstein.
Motivated by \cite{CHT,CV} we ask:
\begin{question}
For which   $\bfv$ does 
$X_\bfv$ admit a (Gorenstein) square-free Gr\"obner degeneration?
\end{question}

A final question, of combinatorial nature, concerns the simplicial complex $\Delta$ of Section \ref{SectionSquarefreeInitial}.

\begin{question}
Is  $\Delta$ Cohen-Macaulay in all characteristics? 
Is it shellable, or vertex-decomposable?
\end{question}

\subsection*{Acknowledgments}
The authors would like to thank
David Eisenbud, Alessandro De Stefani, Paolo Lella, Michael Perlman, and Jerzy Weyman for  helpful conversations.
They also thank the anonymous referees for their suggestions. 
Alessio Sammartano was partially supported by 
 PRIN 2020355B8Y “Squarefree Gr\"obner degenerations, special varieties and related topics”.
Computations with Macaulay2 \cite{Macaulay2} 
 provided valuable insights during the preparation of this paper.

\bibliographystyle{amsalpha} 
\bibliography{references}

\end{document}